\documentclass[a4paper,11pt]{article}

\usepackage{amssymb}
\usepackage{amsmath}
\usepackage{amsthm}
\usepackage{array}
\usepackage{authblk}
\usepackage{algorithm}
\usepackage{url}
\usepackage{algorithmicx}%
\usepackage{algpseudocode}%
\usepackage{graphicx}
\usepackage[usenames]{color}
\usepackage{multirow}
\usepackage{amssymb}
\DeclareMathOperator*{\argmin}{arg\,min}
\DeclareMathOperator*{\argmax}{arg\,max}
\usepackage{hyperref}
\newcommand{\R}{\mathbb{R}}

\newcommand{\F}{\mathcal{F}}

\theoremstyle{remark}

\makeatletter

\gdef\listctr{list\romannumeral\the\@listdepth}\expandafter
  
\makeatother

\theoremstyle{definition}
\newtheorem{theorem}{Theorem}
\newtheorem{proposition}{Proposition}

\newtheorem{lemma}{Lemma}
\newtheorem{assumption}{Assumption}
\newtheorem{definition}{Definition}
  {\end{list}}

\title{SMOP: Stochastic trust region method for multi-objective problems}

\author{Nata\v{s}a Kreji\'c\footnote{Department of Mathematics and Informatics, Faculty of Sciences, University of Novi Sad, Trg Dositeja Obradovi\' ca 4, 21000 Novi Sad, Serbia. e-mail: \texttt{natasak@uns.ac.rs}}, Nata\v{s}a Krklec Jerinki\'c \footnote{Department of Mathematics and Informatics, Faculty of Sciences, University of Novi Sad, Trg Dositeja Obradovi\' ca 4, 21000 Novi Sad, Serbia. e-mail: \texttt{natasa.krklec@dmi.uns.ac.rs}},
Luka Rute\v{s}i\'c \footnote{Department of Mathematics and Informatics, Faculty of Sciences, University of Novi Sad, Trg Dositeja Obradovi\' ca 4, 21000 Novi Sad, Serbia. e-mail: \texttt{luka.rutesic@dmi.uns.ac.rs} } \footnote{Corresponding author} }

\begin{document}
\maketitle
\begin{abstract}
 The problem we consider is a multi-objective optimization problem, in which the goal is to find an optimal value of a vector function representing various criteria. The aim of this work is to develop an algorithm which utilizes the trust region framework with probabilistic model functions, able to cope with noisy problems, using inaccurate functions and gradients.  {\color{red} The key novelty is approximation of each function in the multiobjective problem with probabilistically fully linear model which yields the composite model defined by max operator as a satisfactory approximation for the nonsmooth scalarized objective function. }We prove the almost sure convergence of the proposed algorithm to a Pareto critical point. 
 Numerical  results demonstrate effectiveness of the probabilistic trust region by comparing it to competitive stochastic multi-objective solvers. The application in supervised machine learning is showcased by training non discriminatory Logistic Regression models on different size data groups. Additionally, we use several test examples with irregularly shaped fronts to exhibit the efficiency of the algorithm.\end{abstract}
\section{Introduction}
Multi-objective optimization problems arise in many real-world applications, such as finance, scientific computing, social sciences, engineering, and beyond. These problems are characterized by the need to simultaneously optimize multiple, often conflicting objectives, which significantly complicates the decision making process. Whether {\color{red} one is } maximizing efficiency while minimizing computational cost or minimizing risk while maximizing income, identifying the optimal trade-offs is far from straightforward. The complexity comes from the competing nature of the objectives, where improving one criterion comes at the expense of other. The {\color{red} considered} problem can formally be stated as 
{\color{red} \begin{equation}\label{mop}
\min_{x} f(x)=\min_{x \in \R^n} (f_1(x),...,f_q(x))^T
\end{equation}
where $f:\R^n\rightarrow \R^q$. The main goal of multi-objective optimization is to identify the set of Pareto optimal points. A locally Pareto optimal point is a point such that there exists a neighborhood around it in which no other point improves all objective function values simultaneously, see \cite{FS},  \cite{TOMOP}. If the point can not be improved on the entire domain, the point is globally Pareto optimal. When extending this concept to global solutions, the Pareto front is defined as the set of nondominated objective values corresponding to Pareto optimal points. The algorithms for solving problem \eqref{mop} are designed to find a broader set of Pareto critical points; points for which no common descent direction exists that improves all objectives simultaneously. By finding Pareto critical points, it is possible by localize the stationary form of the Pareto front}, see \cite{CMV},\cite{LV}. The insight into the structure of the entire set of solutions can be crucial in the decision making process, hence it is important for the model to be able to approximate the front.  

Trust region methods for solving this kind of problems work within the standard trust region framework, building a model for each function $ f_i $, generating a direction by solving a multi-model optimization problem and performing the acceptance check as in the classical one dimensional case, see \cite{VOS}. Therein it is shown  that the method converges to a Pareto critical point  under standard assumptions. The convergence towards a stationary point is a common main result of papers dealing with multi-objective problems. The complexity of the problem greatly increases if the functions involved are costly. Computing efficiency and high cost of obtaining exact information play an important role and motivation in opting for the stochastic and derivative free approaches. When creating models within a trust region framework, it is possible to use inexact gradient information. Such derivative free trust region approach can be seen in \cite{TE}. In the mentioned paper, one  criterion  is assumed to be a black box function with a  difficulty to calculate derivative, while other functions and their derivatives can be easily computed. The convergence towards a Pareto critical point is proved. Another version of a derivative free multi-objective trust region approach is discussed in \cite{BP}, where radial basis surrogate models are used. 

It is also possible to approach this problem within a line search framework. Armijo-like condition with the steepest descent and Newton direction is discussed in \cite{SOM}. In \cite{SOM} authors also analyze the projected gradient method for constrained cases. Stochastic multi-gradient multi-criteria approach can be found in \cite{LV}. The authors of \cite{LV}   successfully extend the classical stochastic gradient (SG, see \cite{ROMO}) 
method for single-objective optimization to a multi criteria method, and prove  sublinear convergence for convex and strongly convex functions.

Random models are also frequently used within the trust region framework in the case of a single objective function, i.e., for the case $ q=1. $ A number of approaches are available in literature. Probabilistic trust region method which uses approximate models can be seen in \cite{VS}. It is shown there that with probability one the method converges towards a stationary point, if the models are accurate enough with high probability. 
Trust region method for scalar optimization problems utilizing both approximate functions and gradients can be found in \cite{CH}.  The
analysis therein requires   that the model and function estimates are sufficiently accurate with fixed, sufficiently high probability. These probabilities are predetermined and constant throughout the optimization process  and  almost sure convergence towards a stationary point is proved. Additionally, an adaptive subsampling technique for problems involving functions expressed as finite sums, which are common in applications such as machine learning, is proposed therein. Unlike the traditional subsampling techniques with monotonically increasing size, that method adjusts the size based on the progress. The literature also covers methods specifically designed for optimization of finite sums, which exploit the form through the use of different subsampling strategies, and other various techniques. Some papers in the literature on this topic are \cite{NKSB},\cite{BBN},\cite{BBN2}-\cite{BCHN},\cite{NKNKJ}, \cite{RM},\cite{PROXI}. 

The method we  propose here is based on {\color{red} probabilistically fully linear  models for each function $ f_i $ separately,} as  introduced in \cite{VS} and used later on in \cite{CH,CS}. The concept of full (probabilistic) linearity is extended to vector function in a natural way as explained further on. 

Having a fully linear  model, one has to deal with the fact that at each step of the trust region method we compute the ratio function using approximations of the function values at subsequent steps. Therefore we can not rely on the true model reduction and the decreasing monotonicity. Thus some additional conditions are needed to control the errors. One possibility is to assume that we work with sufficiently small $ \varepsilon_F$ accurate values as done in \cite{CH}. We propose a different assumption here, see ahead Assumption \ref{A1new}, motivated by the applications from machine learning problems. Roughly speaking we are assuming that the approximate gradient $ g_i $ is close enough to the true gradient of the approximate function $ \tilde{f}_i, {\color{red} i=1,\ldots,q}$ which is common in the case of finite sums where one subsamples functional values and takes the approximate gradient as the true gradient of the subsampled function, see \cite{RM}. The assumption also holds if one approximates the gradient by finite differences for example. 

The quality of approximate models is controlled by a probability sequence $ \alpha_k$ which is approaching 1 sufficiently fast. This way one can take advantage of relatively poor model at the beginning of iterative process, hoping to save some computational costs and yet achieve good approximate solution at the end using high quality models. 

Pareto optimal points can be characterized as zeros of the so called marginal functions, see \cite{FS}. This characterization reduces to the usual  first order optimality conditions  (gradient equal to zero) in the case of  $ q=1. $ The concept of marginal function is used in \cite{VOS} to define the trust region method. However, as we work with the approximate functions and gradients, an approximate marginal function is used  together with the  corresponding scalar representation,  see \cite{TE}.

{\color{red}\subsection{Contributions}
%The main contribution of this paper is the following. 
We propose a trust region algorithm for solving multiobjective optimization problem. The problem is first transformed into composite optimization problem with $max$ operator yielding a nonsmooth objective function. We proceed by considering random models per function $ f_i, i=1,\ldots,q $. The standard property of these per function models is assumed - $ \alpha$- probabilistic full linearity. Despite the fact that the scalarization function is nonsmooth we prove that the aggregate random model has sufficiently good agreement with the scalarized function under reasonable assumptions. The trust region method is then defined exploiting the random structure in an asymmetric way - the criteria for search direction is slightly weaker as it is based on approximate stationarity measure of an approximate model. On the other hand the acceptance criteria is slightly  stronger than usual in trust region. This asymmetry seems to work well, taking into account randomness in the models and at the same time allowing us to prove theoretically strong result of almost sure convergence.  The problem we analyse in detail is the multiobjective problem with finite sums. Hence the random models are based on subsampling of functions  and  gradients. Numerical results are presented,  for the case of per function random models of the first order. These experiments  demonstrate the advantages of the proposed approach, in partuclar for the case of large dimensions and large data sets.  Full Pareto front is also considered. The proposed method is tested against the state-of-the-art SMG method \cite{LV} and deterministic trust region for multiobjective problem from \cite{VOS}. }

\section{Preliminaries}
{\color{red} The considered problem is  unconstrained multiobjective minimization problem \eqref{mop}}
%$${\color{red} \begin{equation} \label{mop}
%\min_{x \in \R^n} f(x)=\min_{x\in \R^n} %(f_1(x),...,f_q(x))^T
%\end{equation}
where $f:\R^n\rightarrow \R^q$ and functions  $ f_j, j=1,...,q$ are smooth. Assuming that the explicit evaluation of these functions and its gradients and Hessians  is unavailable or  too costly, we will rely on approximating them with $\tilde{f}_i(x)$, $g_i(x)$ {\color{red} and $H^i(x)$,} respectively.

For problem \eqref{mop} one can define efficient and weakly efficient solution as follows.  
\begin{definition} \cite[{\color{red}Definition 3.1.2}]{TOMOP}.
    A point $ x^* \in  \mathbb{R}^n $ is called an efficient solution for \eqref{mop} (or Pareto optimal) if there exists no point $ x \in \mathbb{R}^n$  satisfying $ f_i(x) \leq  f_i(x^*) $ for all $ i \in \{1,2,...,q\} $ and $f(x)\neq f(x^*)$.
    A point $ x^* \in \mathbb{R}^n $ is called a weakly efficient solution for \eqref{mop} (or weakly Pareto optimal) if there exists no point $ x \in  \mathbb{R}^n $ satisfying $ f_i(x) < f_i(x^*) $ for all $i \in  \{1,2,...,q\}.$ 
\end{definition}
%{\color{red} In other words, $x^*$ is Pareto optimal if, for every direction $ d \in \mathbb{R}^n$, there exists at least one component function  $f_i$  such that the directional derivative satisfies  $$\langle\nabla f_i(x^*),d\rangle\geq 0. $$}
{\color{red}  A point $x^*$ is Pareto critical if and only if there is no direction along which all objective functions decrease simultaneously. In other words, for every direction $ d \in \mathbb{R}^n$, there exists at least one component function  $f_i$  such that the directional derivative satisfies  $$\langle\nabla f_i(x^*),d\rangle\geq 0. $$ Pareto optimality implies Pareto criticality, however the converse is not necessarily true.} 

{\color{red} A stationarity condition for  \eqref{mop} can be derived exploiting the marginal function  
\begin{equation}
    \label{marginal}
\omega(x)=-\min_{\|d\|\leq 1}\left(\max_{i \in \{1,...,q\}}\langle \nabla f_i(x),d\rangle\right).
\end{equation}
It plays a similar role to that of the norm of the gradient of the objective function for single objective problems. In fact, if $q=1$, one gets  $\omega(x)=\|\nabla f(x)\|$.

Let us define 
$${\cal D}(x)=\argmin_{\|d\| \leq 1} \left(\max_{i \in \{1,...,q\}}\langle \nabla f_i(x),d\rangle\right).$$}
The following  lemma  will be used for further considerations. 
\begin{lemma} \cite[{\color{red}Lemma 3}]{FS}. 
    \label{lmarginal}
    {\color{red} The following statements hold: }
    \begin{itemize}
        \item[a)]$ w(x) \geq 0,$  for every  $  x \in \mathbb{R}^n;$
        \item[b)] If $ x $ is Pareto critical for \eqref{mop} then $ 0 \in {\cal D}(x)$ and $ w(x)=0;$
        \item[c)] If $  x $ is not Pareto critical of \eqref{mop} then $w(x)>0 $ and any $ d \in {\cal D}(x)$ is a descent direction for \eqref{mop};
        \item[d)] The mapping $ x \to w(x)$ is continuous.
    \end{itemize}
\end{lemma}
{\color{red} One possible}  scalar representation of the multiobjective problem  \eqref{mop}  is 
 \begin{equation} \label{phiorig} \min_{{\color{red}x \in \R^n}}  \phi(x), \; \; \phi(x)=\max_{i\in\{1,...,q\}}f_i(x).\end{equation}
 %We assume that  $-\infty<\inf\{\phi(x):x\in \R^n\}$. 
 This problem is not equivalent to problem \eqref{mop}, but every  solution of this scalar problem is  a Pareto {\color{red} critical} point. {\color{red} In order to see that, let us recall that subdifferential of $\phi$ is given by 
 $\partial \phi(x)=co\{\nabla f_i(x)  : i \in I_f(x)\} $, where $I_f(x):=\{i \in \{1,...,q\}: f_i(x)=\phi(x)\}$ and  $co$ denotes the convex hull of the stated vectors. The corresponding stationarity condition is $0 \in \partial \phi(x)$, where $0$ represents vector of zeros of dimension $n$. Thus, a stationarity measure can be defined as 
 $$\omega_{\phi}(x)=-\min_{\|d\|\leq 1}\left(\max_{i \in I_f(x)}\langle \nabla f_i(x),d\rangle\right)$$
 Therefore, it follows that $0 \leq \omega(x)\leq \omega_{\phi}(x)$ for every $x \in \R^n$ and thus $\omega_{\phi}(\tilde{x})=0$ implies  $\omega(\tilde{x})=0$.
 }

{\color{red} Let  $B(x,\delta)$ denote a closed ball centered at $ x $ with  radius $ \delta.$  In deterministic trust region methods, at iteration $k$, $\phi$ is usually approximated locally (on a ball $B(x_k,\delta_k)$, where $\delta_k$  represents trust region radius),  by a quadratic model  
$$m^{true}_k(d)=\max_{i\in\{1,...,q\}} \{ f_i(x_k)+\langle \nabla f_i(x_k),d\rangle \}+\frac{1}{2}\langle d,H^i_{k} d\rangle\},$$ where $H^i_{k}$ approximates the Hessian of function $f_i$ for $ i=1,...,q$. A measure of proximity for the models is defined as follows,
. 
\begin{definition} {\color{red}\cite[Definition 6.1]{CSVDFO} }
Function $m_{k}$ is $(c_f,c_g)$ fully linear (FL) model of function  $h$ on $B(x_k,\delta_k)$ if for every $d\in B(0,\delta_k)$ the following two inequalities hold
\begin{equation} \label{def1a} |h(x_k+d)-m_k(d)|\leq c_f\delta_k^2
\end{equation}
\begin{equation} \label{def1b} \|\nabla h(x_k+d)-\nabla m_k(d)\|\leq c_g\delta_k.
\end{equation}
%Model $\tilde{m}_k$ is $(c_f,c_g)$ fully linear model of $\phi$ on $B(x_k,\delta_k)$ if for every $i=1,...,q$, $\tilde{m}_{k,i}$ is $(c_f,c_g)$ fully linear model of $f_i$. 
\end{definition}

    }
{\color{red} Given that we assume that the computation of exact functions and their derivatives is not feasible,  approximate functions are to be used in general. Therefore} we  define the approximate quadratic model for $\tilde{\phi}$ analogously, i.e., {\color{red} 
\begin{equation} \label{agregatni} \tilde{m}_k(d)=\max_{i\in\{1,...,q\}} \tilde{m}_{k,i}(d), \end{equation}}
where 
{\color{red} \begin{equation} \label{novomki} \tilde{m}_{k,i}(d)=\tilde{f}_i(x_k)+\langle g_i(x_k),d\rangle+\frac{1}{2}\langle d,H^i_{k} d\rangle, \; i=1,...,q.\end{equation}
} 
Notice that $\nabla \tilde{m}_{k,i}(0)=g_i(x_k)$ and $\tilde{m}_{k,i}(0)=\tilde{f}_i(x_k)$ for each $i=1,...,q$. 
{\color{red} Furthermore, following \cite{TE}, we consider the approximate marginal function}
\begin{equation}
    \label{amarginal}
    \omega_m(x)=-\min_{\|d\|{\leq}1}\left(\max_{i \in \{1,...,q\}}\langle g_i(x),d\rangle\right)
    \end{equation}
{\color{red} as a stationarity measure of the approximate  multi-objective problem 
\begin{equation*} \label{mopapp}
\min_{x \in \R^n} \tilde{f}(x)=\min_{x\in \R^n} (\tilde{f}_1(x),...,\tilde{f}_q(x))^T.
\end{equation*}
 Analogously to \eqref{phiorig} we denote by $\omega_{\tilde{\phi}}$ stationarity measure of the following scalar problem 
    \begin{equation*} \label{ascalar} 
    \min_{{\color{red}x \in \R^n}} \tilde{\phi}(x), \; \; \tilde{\phi}(x)=\max_{i\in\{1,...,q\}}\tilde{f}_i(x) 
    \end{equation*}
and conclude that  $0 \leq \omega_m(x)\leq \omega_{\tilde{\phi}}(x)$ for every $x \in \R^n$.   These approximate versions are going to be used within the algorithm proposed in the next section, while the convergence analysis will relay on the true marginal function $\omega$. }
    
{\color{red} \subsection{Stochastic framework} }

Our main motivation comes from observing machine learning problems where the functions $f_i, i=1,...,q$ are in the form of finite sums. In that case, the functions are usually approximated by random subsampling which  induces randomness in the optimization process, yielding random sequence of iterates. {\color{red} We will use upper case letters to emphasize random quantities where appropriate  e.g., $X_k$ for random iterates, and lowercase letters to denote the corresponding realizations e.g., $x_k$. To be more precise, let us denote by $(\Omega, \F, P)$ the probability space where:  $\Omega$ represents the set of all possible outcomes, i.e., all possible sample paths of the algorithm to be stated;  $\F$ is a $\sigma$-algebra on $\Omega$; and $P$ is a probability function on a measurable space $(\Omega, \F)$. 

We assume that the stochastic influence comes exclusively from random choices of approximate functions and its derivatives. The stochastic counterparts of $\tilde{f}_i,  g_i$ and $H^i$ will be denoted by $\tilde{F}_i,  G_i$ and $\chi^i$, respectively. These random objects constitute the model functions \eqref{novomki} which are also random and thus denoted by $\tilde{M}_{k,i}, i=1,..,q$. The corresponding agregate model function is denoted accordingly by $\tilde{M}_k$. Since the iterates update will be based on random models, we will also have $X_k$ as random vectors. The same is true for the trust region radius whose stochastic counterpart will be denoted by $\Delta_k$. Although random sampling is an original generator of stochastic influence within the considered framework, we set  $\{X_k\}_{k \in \mathbb{N}}$ as a representative stochastic process as common in the literature. Then, we denote by $\F_k$ the sub-$\sigma$-algebra of $\F$ generated by $X_1,...,X_k$. Thus, $\{\F_k\}_{k \in \mathbb{N}}$ is the natural filtration of $\F$ with respect to $\{X_k\}_{k \in \mathbb{N}}$ and there holds 
$\F_1\subseteq \F_2 \subseteq...\subseteq \F$.
We denote by $E(\cdot | \F_k)$ the conditional expectation at iteration $k$, and by $E(\cdot)$ unconditional expectation with respect to all possible sample paths $v \in \Omega$. 

In convergence analysis we will use the concept of probabilistic fully linear models {\color{red}\cite{VS}}. Instead of fixing the probability parameter $\alpha$, we introduce a sequence of relevant probabilities $\alpha:=\{\alpha_k\}_{k \in \mathbb{N}}$ and give the following definition. 
\begin{definition} {\color{red}\cite[Definition 3.2]{VS}}
A sequence of random models  $\{\tilde{M}_{k,i}\}_{k \in \mathbb{N}}$ is $\alpha$ -probabilistically $(c_f,c_g)$ fully linear  with respect to the corresponding sequence of $B(X_k,\Delta_k) $ if   the events 
\begin{equation*} \label{Iki} I_{k,i}=\{\tilde{M}_{k,i} \text{ is } (c_f,c_g) \text{ fully linear model of } f_i \text{ on } B(X_k,\Delta_k) \}\end{equation*} 
satisfy the condition 
$P(I_{k,i}|\F_k)\geq \alpha_k$ for all $k$. 
\end{definition}
Since the multi-objective setup requires multiple models, we will introduce the following definition of Jointly {\color{red} Independent Probabilistically Fully Linear models. }
\begin{definition}  (JIPFL condition)
We say that a sequence of multiple  random models $\{\tilde{M}_{k,1},...,\tilde{M}_{k,q}\}_{k \in \mathbb{N}}$ is jointly {\color{red} independent } $\alpha$ -probabilistically $(c_f,c_g)$ fully linear  with respect to the corresponding sequence of $B(X_k,\Delta_k) $ if the  sequence of random models  $\{\tilde{M}_{k,i}\}_{k \in \mathbb{N}}$ is $\alpha$ -probabilistically $(c_f,c_g)$ fully linear  with respect to the corresponding sequence of $B(X_k,\Delta_k) $ for each $i=1,...,q$ and the events $I_{k,1},...,I_{k,q}$ are mutually  independent conditionally on $\F_k$ for all $k \in \mathbb{N}$.  
\end{definition} 
Notice that the  JIPFL condition implies 
\begin{equation} \label{IkJIP} P(I_k| \F_k):=P(\bigcap_{i=1}^q I_{k,i}|\F_k)=\prod_{i=1}^qP(I_{k,i}|\F_k)\geq\alpha_k^q, \quad \mbox{for all } \; k \in \mathbb{N}.
\end{equation}
{\color{red} The above stated condition of independence } is  often fulfilled  in the finite sum setup since each  function $f_i$ is usually approximated by independent random sampling.
} {\color{red} The main point of the above definition is to allow us to connect the per function individual probabilistically fully linear models with the scalarization function $ \phi, $ which is nonsmooth and hence full linearity of the multiple random models with respect to $ \phi $ can not even be defined. However we will see later on that JIPFL condition allow us to prove that the multiple random models are good enough for almost sure convergence. }

\section{Algorithm}

{\color{red} We state the algorithm as follows.  Although a vast majority of objects in the algorithm is stochastic, we present them  in small letters for readability.  }

\noindent {\bf{Algorithm 1.}}  \\ \textit{(SMOP: Stochastic trust region method for Multi-Objective Problems)}\label{Alg1}
\begin{itemize}
\item[] Step 0. Input parameters:  $x_0 \in \R^n$,  {\color{red}$\Theta>0, \delta_{max}>0$,  $\delta_0\in(0,\delta_{max})$}, $\gamma_1,\eta_1\in(0,1),\gamma_2=1/\gamma_1$. 
\item[] Step 1. {\color{red} Form a model $\tilde{m}_k(d)$ {\color{red} by \eqref{novomki} and  \eqref{agregatni}. } 
\item[] Step 2.  Find a step $d_k\in B(0,\delta_k)$ such that:
\begin{equation} \label{Cauchy}
\tilde{m}_k(0)-\tilde{m}_k(d_k)\geq\frac{1}{2}\omega_m(x_k)\min\{\delta_k,\frac{\omega_m(x_k)}{\beta_k}\}, 
\end{equation}
where 
{\color{red} $\beta_k:=1+\max_{i \in \{1,...,q\}}\|H^i_k\|.$}
\item[] Step 3.  Compute
$$\rho_k=\frac{\tilde{\phi}(x_k)-\tilde{\phi}(x_k+d_k)}{\tilde{m}_k(0)-\tilde{m}_k(d_k)}$$
{\color{red}(Successful iteration)} If $\rho_k\geq\eta_1$ and $\omega_m(x_k)>\Theta\delta_k$, set $x_{k+1}=x_k+d_k$ and $\delta_{k+1}=\min\{\delta_{\max},\gamma_2\delta_k\}$.\\{\color{red}(Unsuccessful iteration)} Else,  {\color{red}set $x_{k+1}=x_k$ and $\delta_{k+1}=\gamma_1\delta_k$}}.
\item[] Step 4.  Set $k=k+1$ and go to Step 1.
\end{itemize}

{\color{red} In the initialization we choose a starting point together with several hyperparameters. {\color{red} The parameters  $\Theta$ and $\eta $ are used to define successful iterations. } These parameters also influence the trust region radius update, while the intensity of the update is controlled by parameter $\gamma_1$. {\color{red} Moreover, we set the initial value and the upper bound of the trust region radius to $\delta_0$ and $\delta_{max}$, respectively. }

In Step 1 we form {\color{red} approximate random quadratic models   for each function $f_i, i=1,...,q$ and the aggregate model  {\color{red} by \eqref{novomki} and  \eqref{agregatni}. }  } Considering the stochastic framework, the inflow of randomness happens here since the constituting approximate functions ($\tilde{f}_i$) and the derivatives ($g_i, H_k^i$) are constructed/sampled within this step. Later on we will see that the JIPFL condition {\color{red} ensures  the almost sure convergence.  JIPFL condition also } guides the sampling approach from two perspectives. {\color{red} First}, the samples are to be drawn independently across functions $f_i, i=1,...,q$. {\color{red} Beside that}, JIPFL condition guides the sample size update since it  influences the  condition \eqref{IkJIP}.

In Step 2 we search for a suitable direction which provides a sufficient decrease of the approximate aggregate model function. A suitable decrease is defined through $\delta_k, \beta_k$ and the stationarity measure  $\omega_m$ \eqref{amarginal}. In Lemma \ref{sufred} we prove that it is possible to find such direction and therefore Step 2 is well defined.  

In Step 3 we calculate an agreement between the approximate model reduction and the reduction of approximate scalarization function. If the agreement is sufficiently {\color{red} large} and the stationarity measure $\omega_m(x_k)$ is relatively {\color{red} large} with respect to $\delta_k$, then we accept the proposed iterate update and increase the trust region radius if possible. We will refer to these iterations as successful iterations in the sequel. Otherwise, if the iteration is not successful, we reject the update and decrease the trust region radius in order to give better chances to the model to be a good representative of the function $\tilde{\phi}$   on smaller region in the next iteration. 

One can notice that we alternate between the approximate multi-objective problem and its scalarization since the models are targeting $\tilde{\phi}$, while we use the measure of stationarity $\omega_m$ of an approximate multiobjective problem instead of  the stationarity measure of its scalarized version $\omega_{\tilde{\phi}}$. Since it is known that $\omega_{\tilde{\phi}}\geq \omega_m$, this results in relaxed condition \eqref{Cauchy}. The  reasoning behind this includes the fact that we are dealing with approximate (stochastic) versions in general. Imposing strict conditions while having possibly poor approximations of the objective functions is usually far from beneficial. On the other hand, this kind of relaxation usually needs to be compensated. We compensate in Step 3 where $\omega_m$ is used within the  acceptance criteria instead of $\omega_{\tilde{\phi}}$. This seems to provide a good balance on average.  }

The following lemma shows that the algorithm is well defined.

\begin{lemma}
\label{sufred}
For all $k$,  there exists $d_k$ such that \eqref{Cauchy} holds. 

\end{lemma}
\begin{proof}
 We will prove that the condition \eqref{Cauchy} holds  for the Cauchy direction, $d_k^c=\alpha_k d_k^*$, where $d_k^*$ is a solution of the problem stated in \eqref{amarginal}, i.e., 
$$\omega_m(x_k)=-\min_{\|d\|\leq1}\left(\max_{i \in \{1,...,q\}}\langle g_i(x_k),d\rangle\right)=-\max_{i \in \{1,...,q\}}\langle g_i(x_k),d_k^*\rangle$$ and $\alpha_k=\argmin_{0\leq\alpha\leq\delta_k}\{\tilde{m}_k(\alpha d_k^*)\}$.  Since $\|d_k^*\|\leq 1$, we have $\alpha_k d_k^*\in B_k(0,\delta_k)$. Notice that $$\alpha_k=\argmin_{0\leq\alpha\leq\delta_k}\{\tilde{m}_k(\alpha d_k^*)\}=\argmax _{0\leq\alpha\leq\delta_k}\{\tilde{m}_k(0)-\tilde{m}_k(\alpha d_k^*)\}.$$  
Next, we lower bound $\tilde{m}_k(0)-\tilde{m}_k(\alpha d_k^*)$ by a quadratic function of $\alpha$.
\begin{align*}
    \tilde{m}_k(0)-\tilde{m}_k(\alpha d_k^*)&=\max_{i \in \{1,...,q\}}\tilde{f}_i(x_k)-\max_{i \in \{1,...,q\}}\{\tilde{f}_i(x_k)+\langle g_i(x_k),\alpha d_k^*\rangle-\frac{1}{2}\alpha^2\langle d_k^*,H^i_kd_k^*\rangle{\color{red}\}}\\
    &\geq-\alpha\max_{i \in \{1,...,q\}}\langle g_i(x_k),d_k^*\rangle-\frac{1}{2}\alpha^2 {\color{red}\max_{i \in \{1,...,q\}}\langle d_k^*,H^i_k d_k^*\rangle}\\
    &\geq\alpha\omega_m(x_k)-\frac{1}{2}\alpha^2\|d_k^*\|^2{\color{red}\beta_k}\geq\alpha\omega_m(x_k)-\frac{1}{2}\alpha^2\beta_k.
\end{align*}
Thus, we conclude that 
\begin{equation}
    \label{novon1} \tilde{m}_k(0)-\tilde{m}_k(d^c_k)=\max_{0\leq\alpha\leq\delta_k}\{\tilde{m}_k(0)-\tilde{m}_k(\alpha d_k^*)\}\geq\max_{0\leq\alpha\leq\delta_k}\{\alpha \omega_m(x_k)-\frac{1}{2}\alpha^2\beta_k\}.
    \end{equation}
Notice that the solution of the problem at the right-hand side of \eqref{novon1} is  given by $\alpha^*=\min \{\frac{\omega_m(x_k)}{\beta_k}, \delta_k\}. $ 
If $\frac{\omega_m(x_k)}{\beta_k}\leq \delta_k$, then we have 
$$\max_{0\leq\alpha\leq\delta_k}\{\alpha \omega_m(x_k)-\frac{1}{2}\alpha^2\beta_k\}=\frac{\omega_m(x_k)^2}{\beta_k}-\frac{1}{2}\frac{\omega_m(x_k)^2}{\beta_k^2}\beta_k=\frac{\omega_m(x_k)^2}{2\beta_k}.$$
Else, if $\frac{\omega_m(x_k)}{\beta_k}> \delta_k$, we obtain 
\begin{align*}
\max_{0\leq\alpha\leq\delta_k}\{\alpha \omega_m(x_k)-\frac{1}{2}\alpha^2\beta_k\}&=\delta_k \omega_m(x_k)-\frac{1}{2}\delta_k^2 \beta_k\\
    &\geq \delta_k\omega_m(x_k)-\frac{1}{2}\delta_k \omega_m(x_k)\\
    &=\frac{1}{2}\delta_k\omega_m(x_k).
\end{align*} 
Thus, having in mind both cases and  using \eqref{novon1} we obtain 
$$\tilde{m}_k(0)-\tilde{m}_k(d_k^c)\geq\frac{1}{2}\omega_m(x_k)\min\{\frac{\omega_m(x_k)}{\beta_k},\delta_k\},$$
which completes the proof.
 
\end{proof}

\section{Convergence analysis}
%We start this section by stating some of the assumptions needed for the analysis. 

{\color{red} This section provides convergence analysis of the proposed method. We start the analysis by stating some basic assumptions. 
In Lemmas \ref{implikacija}-\ref{ll3} we provide some bounds that hold under assumption of full linearity. More precisely, we show that if the realizations of random models $\tilde{m}_{k,i}=\tilde{M}_{k,i}(v), i=1,...,q$, for some $v \in \Omega$ are fully linear models of $f_i, i=1,...,q$, respectively, on $B(x_k, \delta_k), $ then the distance between the true function $f_i$ and approximate function $\tilde{f}_i$ is controllable by $\delta_k^2$ on $B(x_k, \delta_k)$ for every $i=1,...,q$. Then, we show that the same is true for the distance between the function  $\phi$ and aggregate approximate model $\tilde{m}_k$. Under the same settings, we also show that  the distance between the marginal function $\omega(x_k)$ and its approximation $\omega_m (x_k)$ is controllable by $\delta_k$.  In Lemma \ref{thm} we also consider the same setup, but we prove that one of the acceptance criteria ($\rho_k\geq \eta_1$) is satisfied provided that the trust region radius is small enough. 

The analysis is continued by introducing Lyapunov function $\Psi_k$ that combines $\phi(X_k)$  and $\Delta_k^2$.  Under uniformly bounded iterates assumption and JIPFL condition, we prove that the sequence $\{\Psi_k\}_{k\in \mathbb{N}}$ converges almost surely and that the sequence $\{\Delta_k\}_{k\in \mathbb{N}}$  is square sumable almost surely, provided that the sequence $\{\alpha_k\}_{k\in \mathbb{N}}$ tends to 1 fast enough (Theorem \ref{prop1}). {\color{red} Theorems 4-5  provide further properties and yield the main } result  stated in  Theorem \ref{glavnate} where we prove that $\{\omega(X_k)\}_{k\in \mathbb{N}}$ tends to zero almost surely. The remaining lemmas and theorems provide some important intermediate results.}

\begin{assumption} \label{A1}
Functions $f_i, i=1,...,q$ are twice continuously differentiable and bounded from {\color{red}below}.
\end{assumption}

\begin{assumption} \label{A2}
There exists a positive constant $c_{h}$ such that for all $x\in \R^n$ and $i=1,...,q$ there holds
$\|\nabla^2 f_i(x)\|\leq c_{h}.$ Furthermore there exists a positive constant $c_{b}$ such that $\beta_k \leq c_b$ for every $k$.
\end{assumption}

\begin{assumption} \label{A1new}
Approximate functions $\tilde{F}_i, i=1,...,q$ are continuously-differentiable with $L$-Lipschitz continuous gradients satisfying the following inequality 
$\|\nabla \tilde{F}_i(x_k)-G_i(x_k)\| \leq c_a \delta_k$ with some $c_a>0.$
\end{assumption}

{\color{red} Assumption \ref{A2} is strong, but it is fulfilled in some important classes of machine learning problems such as logistic regression and linear least squares. } Assumption \ref{A1new} is also satisfied in many applications. For instance, subsampling strategies for finite sums usually yield $ \nabla \tilde{F}_i(x_k)=G_i(x_k)$. Alternatively, one can apply finite differences to approximate the relevant gradients with a controllable accuracy. 

%The following lemma quantifies the distance between the true and approximate functions on $ B(x_k,\delta_k).           $

\begin{lemma} \label{implikacija}
Assume that A\ref{A1}-A\ref{A1new} hold. {\color{red} Suppose that  $v \in \Omega$ is such that $\tilde{m}_{k,i}=\tilde{M}_{k,i} (v)$ is $(c_f,c_g)$-fully linear model of $f_i$ 
on $B(x_k,\delta_k)$, where $x_k=X_k(v)$ and $\delta_k=\Delta_k(v)$. 
Then there exists $c_e>0$ such that for all $d_k\in B(0,\delta_k)$ 
%and $i=1,..,q$ 
there holds
\begin{equation} \label{def2}  |\tilde{f}_i(x_k+d_k)-f_i(x_k+d_k)|\leq c_e\delta_k^2
\end{equation}
where $\tilde{f}_i=\tilde{F}_i (v).$
}
\end{lemma}

\begin{proof} 
%Let us observe an arbitrary $i \in \{1,2,...,q\}.$ 
%Putting $d=0$ in \eqref{def1a} and using the fact that $\tilde{m}_{k,i}(0)=\tilde{f}_i(x_k)$ we obtain 
%\begin{equation} \label{imp1} |\tilde{f}_i(x_k)-f_i(x_k)|\leq c_f\delta_k^2
%\end{equation}
Let us take any $d_k\in B(0,\delta_k)$, i.e., any $d_k$ satisfying    $\|d_k\| \leq \delta_k$. Then there exist $\tau^i_k$ and $ v^i_k $ on the {\color{red} line segment between $x_k$ and $x_k+d_k$ } such that 
\begin{align*}
   & |\tilde{f}_i(x_k+d_k)-f_i(x_k+d_k)|=|\tilde{f}_i(x_k)+\nabla \tilde{f}_i(\tau^i_k)d_k-f_i(x_k+d_k)|\\\nonumber 
   &=|\tilde{f}_i(x_k)+\nabla^T \tilde{f}_i(\tau^i_k)d_k-f_i(x_k)-\nabla^T f_i(x_k)d_k-\frac{1}{2} d_k^T \nabla^2 f_i(v^i_k) d_k|\\\nonumber 
   &\leq |\tilde{f}_i(x_k)-f_i(x_k)| +\|\nabla \tilde{f}_i(\tau^i_k)-\nabla f_i(x_k)\|\|d_k\|+\frac{1}{2} \|d_k\|^2 \|\nabla^2 f_i(v^i_k) \|\\\nonumber 
   &\leq c_f \delta_k^2+\|\nabla \tilde{f}_i(\tau^i_k)-\nabla f_i(x_k)\|\delta_k+\frac{1}{2} \delta_k^2 c_{h}.
\end{align*}
 Moreover, {\color{red} by using the second full linearity condition} \eqref{def1b} and the fact that $g_i(x_k)= \nabla \tilde{m}_{k,i} (0)$, {\color{red} where $g_i=G_i(v)$,} we can upper bound  $\|\nabla \tilde{f}_i(\tau^i_k)-\nabla f_i(x_k)\|$ as follows.  \begin{align*}
   & \|\nabla \tilde{f}_i(\tau^i_k)-\nabla f_i(x_k)\|  =\|\nabla \tilde{f}_i(\tau^i_k)-\nabla f_i(x_k)+g_i(x_k)-g_i(x_k)+\nabla \tilde{f}_i(x_k)-\nabla \tilde{f}_i(x_k)\| \\\nonumber 
   &\leq \|\nabla \tilde{f}_i(\tau^i_k)-\nabla \tilde{f}_i(x_k)\|+\|\nabla f_i(x_k)-g_i(x_k)\|+\|\nabla \tilde{f}_i(x_k)-g_i(x_k)\| \\\nonumber 
   &\leq L\|\tau^i_k-x_k\|+c_g \delta_k+c_a\delta_k \leq L\delta_k+c_g \delta_k+c_a\delta_k=(L+c_g+c_a) \delta_k.
\end{align*}
Thus we conclude that 
$$|\tilde{f}_i(x_k+d_k)-f_i(x_k+d_k)|\leq \delta_k^2 c_e,$$
where $c_e=c_f+L+c_g+c_a+c_{h}/2$, which completes the proof. 
\end{proof}

%Now we state the conditions under which the difference between the true and the approximate marginal function is of the order $\delta_k$.

\begin{lemma}
\label{margfun}
{\color{red} Assume that  A\ref{A1} holds
%-A\ref{A1new}hold. 
Suppose that  $v \in \Omega$ is such that $\tilde{m}_{k,i}=\tilde{M}_{k,i} (v)$ is $(c_f,c_g)$-fully linear model of $f_i$ 
on $B(x_k,\delta_k)$ for all $i=1,..,q$, where $x_k=X_k(v)$ and $\delta_k=\Delta_k(v)$.} Then
\begin{equation}
\label{omega}
    |\omega(x_k)-\omega_m(x_k)|\leq\delta_k c_g
\end{equation}
\end{lemma}

\begin{proof}  
{\color{red} By using the following notation 
$h_k(d):=\max_{i \in \{1,...,q\}} \langle\nabla f_i(x_k), d\rangle$,  $\tilde{h}_k(d):=\max_{i \in \{1,...,q\}} \langle g_i(x_k), d\rangle$ we get 
\begin{eqnarray*}
    \label{wwm}
 |\omega(x_k)-\omega_m(x_k)|&=& \left | \min_{\|d\|\leq 1} h_k(d)-\min_{\|d\|\leq 1} \tilde{h}_k(d)\right| \\\nonumber
 &\leq& \sup_{\|d\|\leq 1} \left|h_k(d)-\tilde{h}_k(d)\right|\\\nonumber
 &\leq& \sup_{\|d\|\leq 1} \max_{i \in \{1,...,q\}} \left|\langle\nabla f_i(x_k)-g_i(x_k), d\rangle\right|
 \\\nonumber
 &\leq&  \max_{i \in \{1,...,q\}} \|\nabla f_i(x_k)-g_i(x_k)\|  \sup_{\|d\|\leq 1} \|d\|\\\nonumber
 &\leq& c_g \delta_k,
\end{eqnarray*}
where 
the second full linearity condition \eqref{def1b} is used {\color{red} for the the final inequality}. }
\end{proof}

%The following lemma shows that the distance of the model from the approximate scalar representation of $f$ is small enough under the stated assumptions.
{\color{red} In the next {\color{red}lemma} we prove that although $ \phi $ is nonsmooth the multidimensional random model approximates $ \phi$ and $ \tilde{\phi}$ with the order of $ \delta_k^2.$ }
\begin{lemma}
\label{ll3} {\color{red} Assume that  A\ref{A1} -A\ref{A1new} hold. 
Suppose that  $v \in \Omega$ is such that $\tilde{m}_{k,i}=\tilde{M}_{k,i} (v)$ is $(c_f,c_g)$-fully linear model of $f_i$ 
on $B(x_k,\delta_k)$ for all $i=1,..,q$, where $x_k=X_k(v)$ and $\delta_k=\Delta_k(v)$. Then the following inequalities hold for all $d_k \in B(0,\delta_k)$:
\begin{equation}
    \label{fimtil1}
|\phi(x_k+d_k)- \tilde{m}_k(d_k)|\leq c_{f} \delta_k^2,
\end{equation}
\begin{equation}
\label{fifitil}
    |\tilde{\phi}(x_k+d_k)-\phi(x_k+d_k)|\leq c_e\delta_k^2,
\end{equation}  
\begin{equation}
\label{fitilmtil}
    \vert \tilde{\phi}(x_k+d_k)-\tilde{m}_k(d_k) \vert \leq c_{\tilde{\Phi}}\delta_k^2,
\end{equation}
where $c_{\tilde{\Phi}}=\max\{c_f, c_e\}$.
}
\end{lemma}
\begin{proof}
{\color{red}
Inequality \eqref{fimtil1} is obtained by using the first inequality of fully linear models \eqref{def1a} 
%and $\|d_k\| \leq \delta_k$ 
as follows
\begin{eqnarray*}
    |\phi(x_k+d_k)- \tilde{m}_k(d_k)|&=& \left|\max_{i \in \{1,...,q\}}f_i(x_k+d_k)- \max_{i \in \{1,...,q\}}\tilde{m}_{k,i} (d_k)\right|\\\nonumber
    &\leq& \max_{i \in \{1,...,q\}}|f_i(x_k+d_k)-\tilde{m}_{k,i} (d_k)|\\\nonumber
    &\leq& \max_{i \in \{1,...,q\}}c_f \delta_k^2=c_f \delta_k^2.
\end{eqnarray*}
We obtain \eqref{fifitil}
by using similar arguments and \eqref{def2} instead of \eqref{def1a}, while \eqref{fitilmtil} is obtained by using the fact that 
$$|\tilde{\phi}(x_k+d_k)-\tilde{m}_k(d_k) |\leq  |\tilde{\phi}(x_k+d_k)-\phi(x_k+d_k)|+|\phi(x_k+d_k)-\tilde{m}_k(d_k) |$$
and applying \eqref{fimtil1} and \eqref{fifitil}.
}
\end{proof}

\begin{lemma}
\label{thm}
{\color{red} Assume that  A\ref{A1} -A\ref{A1new} hold. 
Suppose that  $v \in \Omega$ is such that $\tilde{m}_{k,i}=\tilde{M}_{k,i} (v)$ is $(c_f,c_g)$-fully linear model of $f_i$ 
on $B(x_k,\delta_k)$ for all $i=1,..,q$, where $x_k=X_k(v)$ and $\delta_k=\Delta_k(v)$. } Suppose that $ d_k $ satisfies \eqref{Cauchy}. Then $\rho_k\geq\eta_1$ provided that 
\begin{equation}
\label{deltabound}
    \delta_k\leq\min\{\frac{\omega_m(x_k)}{c_{b}},\frac{\omega_m(x_k)(1-\eta_1)}{2c_{\tilde{\Phi}}}\}.
\end{equation}
\end{lemma}
\begin{proof}
 From \eqref{Cauchy} it follows
\begin{align*}
\tilde{m}_k(0)-\tilde{m}_k(d_k)&\geq \frac{1}{2}\omega_m(x_k)\min\{\frac{\omega_m(x_k)}{\beta_k},\delta_k\}\\
&\geq \frac{1}{2}\omega_m(x_k)\min\{\frac{\omega_m(x_k)}{c_{b}},\delta_k\}\\
&=\frac{1}{2}\omega_m(x_k)\delta_k
\end{align*}
 Furthermore, using $\tilde{\phi}(x_k)=\tilde{m}_k(0)$ and \eqref{fitilmtil}, we obtain 
$$|\rho_k-1|=\left|\frac{\tilde{\phi}(x_k+d_k)-\tilde{\phi}(x_k)-\tilde{m}_k(d_k)+\tilde{m}_k(0)}{\tilde{m}_k(d_k)-\tilde{m}_k(0)}\right|=$$
$$\leq \left|\frac{\tilde{\phi}(x_k+d_k)-\tilde{m}_k(d_k)}{\tilde{m}_k(d_k)-\tilde{m}_k(0)}\right|\leq\frac{2c_{\tilde{\Phi}}\delta_k^2}{\omega_m(x_k)\delta_k}=\leq 1-\eta_1,$$
and thus we conclude that $\rho_k\geq\eta_1$.
\end{proof}
To continue with the convergence analysis, let us define  an auxiliary {\color{red} Lyapunov} function as usual in this type of analysis, \cite{CH} 
{\color{red} $$ \Psi_k:=\nu\phi(X_k)+(1-\nu)\Delta_k^2, \quad \nu\in(0,1).$$}
We are going to show that we can choose the algorithm parameters such that  the following inequality holds
\begin{equation*}
\label{psi}
E[{\color{red}\Psi_{k+1}-\Psi_k}|\mathcal{F}_{k}]\leq -\sigma{\color{red}\Delta_k^2}+(1-\alpha_k^q)\tilde{\sigma}, \quad k=0,1,...
\end{equation*}
for some $ \sigma, \tilde{\sigma} >0.$ 
{\color{red} Let us define {\color{red} the event of successful } iteration $k$ as 
\begin{equation*} \label{Skevent} S_k=\{{\cal{R}}_k\geq\eta_1 \text{ and }  \omega_m(X_k)>\Theta\Delta_k   \},\end{equation*}
where  ${\cal{R}}_k$ denotes the stochastic counterpart of $\rho_k$.
We also define  the  complementary event (unsuccessful iteration $k$) by 
\begin{equation*} \label{barSkevent} \bar{S}_k=\{{\cal{R}}_k<\eta_1 \text{ or }  \omega_m(X_k)\leq \Theta\Delta_k   \}.\end{equation*}
Notice that 
\begin{equation}
\label{uns}
    E(\Psi_{k+1}-\Psi_k| \F_k, \bar{S}_k)=(1-\nu)(\gamma_1^2-1)\Delta_k^2=:-c_1 \Delta_k^2, 
\end{equation}
for some  $c_1>0$.
%since $X_{k+1}|\F_k, \bar{S}_k=X_k$.  
Thus, in the subsequent lemma we focus our attention on estimating 
%$E(\Psi_{k+1}-\Psi_k| \F_k, S_k)$.
$$E(\Psi_{k+1}-\Psi_k| \F_k, S_k)=E(\nu(\phi(X_{k+1})-\phi(X_k))+(1-\nu)(\gamma_2^2-1)\Delta_k^2| \F_k, S_k).$$}

The proof of the following lemma resembles the analysis of  \cite{CH}. However, having the multi-objective problem requires nontrivial modifications {\color{red} due to the fact that the random models are defined per function $ f_i $ and that $ \phi $ is nonsmooth. Two additional assumptions are needed for the analysis. The first one is JIPFL property defined in Section 3, while the second assumption states that the iterative sequence is uniformly bounded. } 
Although relatively strong, the later assumption is often used in stochastic optimization, \cite{BCN}, \cite{CSVDFO}, \cite{BKAP},\cite{STORM3},\cite{POLY}.
\begin{assumption} \label{A4JIP}
{\color{red} The sequence of multiple  random models $\{\tilde{M}_{k,1},...,\tilde{M}_{k,q}\}_{k \in \mathbb{N}}$ satisfies JIPFL condition  with respect to the corresponding sequence of $B(X_k,\Delta_k) $ .} 

\end{assumption}
\begin{assumption} \label{Aboundediterates}
The sequence $\{X_k\}_{k \in \mathbb{N}}$  is uniformly bounded.
\end{assumption}

For the purpose of the convergence analysis, {\color{red} let  us define the following events 
\begin{equation*} \label{Jki} J_{k,i}=\{|\tilde{F}_i(X_k+d_k)-f_i(X_k+d_k)|\leq c_e\Delta_k^2 \text{ for all  } d_k\in  B_k(0,\Delta_k) \},\end{equation*} 
$i=1,...,q$ and 
\begin{equation*} \label{Jk}
J_k:=\bigcap_{j=1}^q J_{k,j}. 
\end{equation*}
}
Then, under assumptions A\ref{A1}-A\ref{A1new}, according to Lemma \ref{implikacija}  there holds \\ $P(J_k| I_k,\F_k)=1.$ Moreover, 
    \begin{equation} \label{palfaq} P(I_k, J_k|\mathcal{F}_k)=P(I_k|\mathcal{F}_k)P(J_k|I_k, \mathcal{F}_k)\geq \alpha_k^q 1=\alpha_k^q
    \end{equation} and we can also conclude that $P(I_k, \bar{J}_k|\mathcal{F}_k)=0$, $P(\bar{I}_k, J_k|\mathcal{F}_k)\leq 1-\alpha_k^q$ and $P(\bar{I}_k, \bar{J}_k|\mathcal{F}_k)\leq 1-\alpha_k^q ,$
    {\color{red}
    where $\bar{I}_k$ and $\bar{J}_k$ denote the complementary events of $I_k$ and $J_k$, respectively. We also use $D_k$ to denote stochastic counterpart of the step size $d_k$ determined in Step 2 of the SMOP algorithm.}

\begin{lemma} \label{prvapomocna}
     Suppose that {\color{red} A\ref{A1}-A\ref{Aboundediterates}  hold } and  there exists $\bar{\alpha}>0$ such that $\alpha_k \geq \bar{\alpha}$ for all $k$. Then there exist positive constants $c_6, c_7$ such that the following holds for all $k$ 
    \begin{equation*} \label{a} 
E({\color{red} \Psi_{k+1}-\Psi_{k}}|\mathcal{F}_{k},S_k)\leq -c_6{\color{red}\Delta_k^2} +c_7 (1-\alpha_k^q), 
\end{equation*}
if  
    \begin{equation}
        \label{thtbnd}
           {\color{red}\Theta}\geq\max\{c_b,5c_f,\frac{4c_e}{\eta_1}\} {\color{red} \mbox{ and } \frac{\nu}{1-\nu}\geq\frac{4\gamma_2^2-2}{\min\{c_e,c_f\}} }
    \end{equation}

\end{lemma}

\begin{proof}
{\color{red} Given $\F_k \bigcap S_k$, the following events make mutually exclusive and collectively exhaustive events at iteration $k$
\begin{equation*}
\label{Udogadjaji}
U^1_k := I_k, \quad  U^2_k:=\overline{I}_k \bigcap J_k, \quad   U^3_k:=\overline{I}_k \bigcap \overline{J}_k.
\end{equation*}
We analyze these three cases separately and gather them together at the end of the proof to obtain the result. }

a){\color{red} Let us consider $E(\Psi_{k+1}-\Psi_{k}|\mathcal{F}_{k},S_k, U_k^1)$ first. Since $I_k$ implies $J_k$, and $S_k$ implies that $\omega_m(X_k)\geq{\color{red}\Theta}\Delta_k$, using {\color{red}\eqref{fimtil1}}, and Lemma \ref{sufred} we obtain 
\begin{eqnarray}
    \label{final}
&& E(\phi(X_{k+1}) - \phi(X_k) |\mathcal{F}_{k},S_k, U_k^1) \\\nonumber & = & E(\phi(X_{k+1}) - \tilde{M}_k(D_k) + \tilde{M}_k(0) - \phi(X_k) + \tilde{M}_k(D_k) - \tilde{M}_k(0) |\mathcal{F}_{k},S_k, U_k^1)  \nonumber\\
& \leq & E(2 c_f \Delta_k^2 - \frac{1}{2} w_m(X_k) \min \{\Delta_k,\frac{w_m(X_k)}{c_b}\} |\mathcal{F}_{k},S_k, U_k^1) \nonumber\\\nonumber
& \leq & E(2 c_f \Delta_k^2 - \frac{1}{2} w_m(X_k) \Delta_k |\mathcal{F}_{k},S_k, U_k^1)\\\nonumber &\leq& E( 2 c_f \Delta_k^2 - \frac{1}{2} \Theta \Delta_k^2|\mathcal{F}_{k},S_k, U_k^1)\\\nonumber
& < &E(-\frac{1}{2}c_f{\color{red}\Delta^2_k}|\mathcal{F}_{k},S_k, U_k^1)=-c_1 \Delta_k^2,
\end{eqnarray}
for $\Theta\geq\max\{c_b,5c_f\}$ and $ c_1 = \frac{1}{2}c_f>0$. This further implies 
\begin{eqnarray*}
\label{novo1}
&& E(\Psi_{k+1}-\Psi_k|\mathcal{F}_{k},S_k, U_k^1)\\\nonumber
&=& E(\nu(\phi(X_{k+1})-\phi(X_k))+(1-\nu)(\gamma_2^2-1)\Delta_k^2|\mathcal{F}_{k},S_k, U_k^1)\\\nonumber 
&\leq& [-\nu c_1+(1-\nu)(\gamma_2^2-1)]\Delta_k^2,
\end{eqnarray*}
and thus  by choosing $\nu$ as in \eqref{thtbnd} we obtain  $$\frac{\nu}{1-\nu}\geq\frac{2\gamma_2^2-1}{c_1}$$ and
\begin{equation}
\label{b2dk2}
    E(\Psi_{k+1}-\Psi_k|\mathcal{F}_{k},S_k, U_k^1)\leq-\gamma_2^2\Delta_k^2=-c_2\Delta_k^2, 
\end{equation}
with $c_2=\gamma_2^2>0$.

b)Now let us consider $E(\Psi_{k+1}-\Psi_{k}|\mathcal{F}_{k},S_k, U_k^2)$.  Using \eqref{fifitil} and \eqref{Cauchy}, and the fact that $S_k$ implies  ${\cal{R}}_k\geq\eta_1$ we get
\begin{eqnarray*} 
&&E(\phi(X_{k+1})-\phi(X_{k})|\mathcal{F}_{k},S_k, U_k^2) \\\nonumber &=& E(\phi(X_{k+1})-\tilde{\phi}(X_{k+1}) + \tilde{\phi}(X_{k}) - \phi(X_k) + \tilde{\phi}(X_{k+1}) -\tilde{\phi}(X_{k})|\mathcal{F}_{k},S_k, U_k^2) \\
& \leq & 2 E(c_e\Delta_k^2+\tilde{\phi}(X_{k+1}) -\tilde{\phi}(X_{k})|\mathcal{F}_{k},S_k, U_k^2))\\\nonumber
&=& E(2c_e\Delta_k^2-{\cal{R}}_k(\tilde{M}_k(D_k)-\tilde{M}_k(0)) |\mathcal{F}_{k},S_k, U_k^2)\\\nonumber 
&\leq& E(2c_e\Delta_k^2-\eta_1(\tilde{M}_k(D_k)-\tilde{M}_k(0)) |\mathcal{F}_{k},S_k, U_k^2) \\\nonumber 
&\leq& E( 2 c_e\Delta_k^2-\frac{\eta_1\omega_m(X_k)}{2}\min\{\frac{\omega_m(X_k)}{c_{b}},\Delta_k\}|\mathcal{F}_{k},S_k, U_k^2)\\ \nonumber 
&\leq& [2 c_e-\frac{\eta_1\Theta}{2}\min\{\frac{\Theta}{c_{b}},1\}]\Delta_k^2=[2 c_e-\frac{\eta_1\Theta}{2}]\Delta_k^2\leq-\frac{1}{2}c_e\Delta_k^2=-c_3\Delta_k^2.
\end{eqnarray*}
for $\Theta\geq\frac{5c_e}{\eta_1}$, and $c_3=\frac{1}{2}c_e>0$. Again, from \eqref{thtbnd} we obtain 
$$\frac{\nu}{1-\nu}\geq\frac{2\gamma_2^2-1}{c_3}$$
and we get that 
\begin{equation}
\label{c1c}
    E(\Psi_{k+1}-\Psi_k|\mathcal{F}_{k},S_k, U_k^2)\leq[-\nu c_3+(1-\nu)(\gamma_2^2-1)]\Delta_k^2\leq -\gamma_2^2\Delta_k^2= -c_4 \Delta_k^2.
\end{equation}
for $c_4=\gamma_2^2>0$

c)Finally, let us consider $E(\Psi_{k+1}-\Psi_{k}|\mathcal{F}_{k},S_k, U_k^3)$. 
 Again, $S_k$ implies  $\omega_m(X_k)\geq\Theta\Delta_k$, but an increase of  $\Psi_k$ can happen.  However, using Taylor expansion, A\ref{A2}, and the Cauchy Schwartz inequality, we obtain the following bound regardless of the scenario $U_k^3$.
\begin{eqnarray}\nonumber
\phi(X_{k+1})-\phi(X_k)& = &\max_{i \in \{1,...,q\}} f_i(X_{k+1}) - \max_{i \in \{1,...,q\}} f_i(X_k) \\\nonumber
 &\leq  & \max_{i \in \{1,...,q\}} |\nabla^T f_i(x_k)d_k+\frac{1}{2}D_k^T\nabla^2 f_i(T_k)D_k|\\\nonumber
\label{dcase}
    &\leq & \max_{i \in \{1,...,q\}}(\|\nabla f_i(X_k)\|\|D_k\|+\frac{1}{2}\Delta_k^2 c_{h}).
\end{eqnarray}
Since the iterates are assumed to be bounded, the continuity of the gradients implies  the existence of $G>0$ such that $\max_{i \in \{1,...,q\}}\|\nabla f_i(X_k)\|\leq G$. Since $\Delta_k \leq \delta_{max}$ there exists a constant $c_5$ such that 
 \begin{equation*} \label{bezBk} \phi(X_{k+1})-\phi(X_k)\leq G\delta_{max}+\frac{1}{2}\delta_{max}^2c_h=: c_5
 \end{equation*} and thus 
 \begin{equation} \label{bezBkpsi} E(\Psi_{k+1}-\Psi_k |\mathcal{F}_{k},S_k, U_k^3)\leq c_5 +(1-\nu)(\gamma_2^2-1)\Delta_k^2.
 \end{equation}
 }

 Now, we combine inequalities \eqref{b2dk2},\eqref{c1c} and \eqref{bezBkpsi} to estimate $E(\Psi_{k+1}-\Psi_{k}|\mathcal{F}_{k},S_k)$. 
Using the total probability formula we obtain 
\begin{eqnarray*}
   &&  E(\Psi_{k+1}-\Psi_{k}|\mathcal{F}_{k},S_k)\\\nonumber
   &=& \sum_{i=1}^{3} P(U_k^i | \mathcal{F}_{k},S_k) E(\Psi_{k+1}-\Psi_{k}|\mathcal{F}_{k},S_k, U_k^i)\\\nonumber
   &\leq & P(U_k^1 | \mathcal{F}_{k},S_k) E(\Psi_{k+1}-\Psi_{k}|\mathcal{F}_{k},S_k, U_k^1)\\\nonumber
   &+& P(U_k^3 | \mathcal{F}_{k},S_k) E(\Psi_{k+1}-\Psi_{k}|\mathcal{F}_{k},S_k, {\color{red}U_k^3}),
\end{eqnarray*}
where the last inequality follows from the fact that 
 $E(\Psi_{k+1}-\Psi_{k}|\mathcal{F}_{k},S_k, U_k^2) \leq -c_3 \Delta_k^2<0$. 
Moreover, notice that \eqref{b2dk2} implies $E(\Psi_{k+1}-\Psi_{k}|\mathcal{F}_{k},S_k, U_k^1)\leq -c_2 \Delta_k^2<0$  and that the conditional expectation $E(\Psi_{k+1}-\Psi_{k}|\mathcal{F}_{k},S_k, U_k^3 ) $ is upper bounded by the  positive quantity given in \eqref{bezBkpsi}. Thus, by  \eqref{palfaq} we obtain 
$$E(\Psi_{k+1}-\Psi_{k}|\mathcal{F}_{k},S_k)\leq -\alpha_k^q  c_2 \Delta_k^2+(1-\alpha_k^q) (c_5 +(1-\nu)(\gamma_2^2-1)\Delta_k^2)$$ and the result follows with $ c_6 = \bar{\alpha}^q c_2 $ and $ c_7 = c_5 +(1-\nu)(\gamma_2^2-1)\delta_{\max}$ due to $\alpha_k \geq \bar{\alpha}$ and $\delta_k \leq \delta^2_{max}.$

\end{proof} 
Now we show that the sequence of trust region radii is square sumable under the {\color{red} following} assumption. 

\begin{assumption}
    \label{A5} The sequence $\{\alpha_k\}_k$ satisfies $\sum_{k=0}^{\infty} (1-\alpha_k^q) \leq c_{\alpha} <\infty$.
\end{assumption}
{\color{red} We also use the  following {\color{red} result for further analysis. }
\begin{theorem} \label{RSthm}
    \cite[Theorem 1]{RS} Let $U_k,\beta_k,\xi_k,\rho_k\geq 0$ be $\mathcal{F}_k$ measurable random variables such that
    $$E(U_{k+1}|\mathcal{F}_k)\leq(1+\beta_k)U_k+\xi_k-\rho_k$$
    If $\sum\beta_k<\infty$ and $\sum\xi_k<\infty$ then $U_k\rightarrow U$ a.s. and $\sum\rho_k<\infty$ a.s.
\end{theorem}}

\begin{theorem}
\label{prop1}
Suppose that {\color{red} A\ref{A1}-A\ref{A5} } and   \eqref{thtbnd} hold.
 Then the sequence $\{\Psi_k\}_{k\in \mathbb{N}}$ converges a.s. and there holds 
\begin{equation} \label{deltaksum} \sum_{k=0}^{\infty}{\color{red} \Delta_k^2}<\infty \quad \mbox{a.s. } 
\end{equation}  

\end{theorem}

\noindent

\begin{proof} Assumption \ref{A5} implies that $\lim_{k \to \infty} \alpha_k=1$, so without loss of generality we can assume that $\alpha_k \geq \bar{\alpha}>0$ for all $k$. Then, according to \eqref{uns} and Lemma \ref{prvapomocna} we obtain 
    \begin{eqnarray}\nonumber
    \label{newEN}
        & & E(\Psi_{k+1}-\Psi_k|\mathcal{F}_k)\\\nonumber 
        &=& E(\Psi_{k+1}-\Psi_k|\mathcal{F}_k,S_k)P(S_k|\mathcal{F}_k)+E(\Psi_{k+1}-\Psi_k|\mathcal{F}_k,\bar{S}_k)P(\bar{S}_k|\mathcal{F}_k)\\\nonumber 
        &\leq& (-c_6\Delta_k^2 +c_7 (1-\alpha_k^q))P(S_k|\mathcal{F}_k)-c_1 \Delta_k^2 P(\bar{S}_k|\mathcal{F}_k)\leq\\\nonumber
        &\leq& -\min\{c_1,c_6\}( P(S_k|\mathcal{F}_k)+P(\bar{S}_k|\mathcal{F}_k))\Delta_k^2+c_7 (1-\alpha_k^q)\\
        &=:& -c_8 \Delta_k^2+c_7 (1-\alpha_k^q)
    \end{eqnarray}
Since $\Psi_k$ is bounded from bellow by $\Psi^*$, by adding and subtracting $\Psi^*$ in the conditional expectation above and using the fact that $\Psi_k$ is $\F_k$-measurable we obtain
$$E(\Psi_{k+1}-\Psi^*|\mathcal{F}_k)\leq\Psi_k-\Psi^*-c_8 \Delta_k^2+c_7 (1-\alpha_k^q)$$
and the result follows from Theorem \ref{RSthm}.
\end{proof}

Now we show that under the {\color{red}stated} conditions a.s. there exists an infinite sequence of iterations with fully linear models. 
{\color{red} We employ the following  result.} 
{\color{red}
\begin{theorem} {\color{red}\cite[Theorem 5.3.1]{DUR}. } \label{mart}
\cite{DUR} Let $G_k$ be a sequence of integrable random variables such that  $E(G_k | {\cal{V}}_{k-1})\geq G_{k-1}$, 
where ${\cal{V}}_{k-1}$ is a $\sigma$-algebra generated by $G_0,...,G_{k-1}$. Assume further that $|G_k-G_{k-1}|\leq M<\infty$ for every $k$. Consider the random events $C=\{\lim_{k\rightarrow\infty}G_k$ exists and is finite$\}$ and $D=\{\limsup_{k\rightarrow\infty} G_k=\infty\}$. Then $P(C \cup D)=1$
\end{theorem}
}

\begin{theorem} 
\label{infFL} 
Suppose that the assumptions of Theorem \ref{prop1} hold. Then a.s. there exists an infinite $K \subseteq \mathbb{N}$ such that %the model $\tilde{m}_k$ is fully linear for all 
{\color{red} $\tilde{M}_{k,i}$ is $(c_f,c_g)$  fully linear model of $f_i$ on  $B(X_k,\Delta_k)$ for all $i=1,...,q$ } and all $k \in K$. 
\end{theorem} 
\begin{proof} Notice that assumption A\ref{A5} implies the existence of $\bar{k}$ such that $\alpha_k^q> 0.5$ for all $k \geq \bar{k}$. 
Let us define a random variable 
\begin{equation}
\label{randomwalk}
    W_k=\sum_{s=\bar{k}}^{k} V_s,
\end{equation}
where {\color{red} $V_k|I_k, \F_k=1$  and $V_k| \bar{I}_k, \F_k=-1$ } otherwise. Moreover, 
\begin{eqnarray*}\nonumber
    E(V_{k+1} | \F_k)&=& P(I_k| \F_k)-P(\bar{I}_k| \F_k)=P(I_k| \F_k)-(1-P(I_k| \F_k))\\
    &=& 2 P(I_k| \F_k)-1\geq 2 \alpha_k^q-1>0.
    \end{eqnarray*} 
This implies  
$E(W_{k+1} | \F_k)=W_k+ E(V_{k+1}| \F_k)> W_k$.  We also have  $|W_{k+1}-W_k|=|V_{k+1}|=1$ and thus the  conditions of Theorem \ref{mart} are satisfied with $G_k=W_k$ and $M=1$.  Moreover, $|W_{k+1}-W_k|=1$ also indicates that the sequence of $W_k$ cannot be convergent and thus Theorem \ref{mart}  implies that a.s.
\begin{equation} \label{wkinf}
    \limsup_{k \to \infty} W_k=\infty. 
\end{equation}
The statement to be proved is that  $I_k$ happens infinitely many times a.s. Assume that this is not true. Then, there exists $\tilde{k}$ such that for each $k \geq \tilde{k}$ the event $\bar{I}_k$ happens so $ V_k  =-1. $ As $ W_k = W_{\tilde{k}} + (k-\tilde{k})V_k $ we get $\lim_{k \to \infty} W_k=-\infty$, which is in contradiction with  \eqref{wkinf}.

\end{proof}

\begin{theorem}
\label{liminf}
Suppose that the assumptions of Theorem \ref{prop1} hold. Then a.s.
$$\liminf_{k\rightarrow\infty}\omega(X_k)=0$$ 
\end{theorem}

\begin{proof} 
{\color{red} Recall that $\Omega$ stands for the set of all possible sample paths of SMOP algorithm.}  Suppose the contrary, that with positive probability none of the subsequences {\color{red} of $\{\omega(X_k(v))\}_{k \in \mathbb{N}}$ } converges to 0. In other words there exists $\hat{\Omega}\subset\Omega$ such that $P(\hat{\Omega})>0$ and {\color{red} $\{\omega(X_k(v))\}_{k \in \mathbb{N}}$} is bounded away from zero for all $v\in\hat{\Omega}$. 

Let us observe an arbitrary  $v\in\hat{\Omega}$ and the corresponding {\color{red} realization 
$\{\omega(X_k(v))\}_{k \in \mathbb{N}}$.
%$\{\omega(x_k)\}_{k \in \mathbb{N}}:=\{\omega(X_k(v))\}_{k \in \mathbb{N}}$.
Under the current assumption we know that there exists $\epsilon(v)>0$ and $k_1(v)$, such that $\omega(X_k(v))\geq\epsilon(v)>0$ for all $ k\geq k_1(v)$. Moreover, Theorem \ref{infFL} implies that there exists $\tilde{\Omega}\subseteq \Omega$ such that  $P(\tilde{\Omega})=1$ and for every $v \in \tilde{\Omega}$ there exists 
$  K(v)\subseteq \mathbb{N}$ such that  for all $ k\in K(v)$,   $\tilde{M}_{k,i}(v)$ are $(c_f,c_g)$  fully linear model of $f_i$ on  $B(X_k (v),\Delta_k(v))$ for all $i=1,...,q$. 

Now, let us observe $\overline{\Omega}:=\tilde{\Omega} \bigcap \hat{\Omega}$. Notice that $P(\overline{\Omega})>0$. Moreover, since $\Delta_k $ tends to 0 a.s. {\color{red}according to Theorem \ref{prop1}}, without loss of generality we can assume that $\lim_{k\rightarrow\infty}\Delta_k(v)=0$ for all $v \in \overline{\Omega}$.

Let  us take an arbitrary $v\in\overline{\Omega}$. %We will  omit writing $v$ further on for the sake of simplicity. 
 There exists $ k_2(v)$ such that for $k\geq k_2(v)$,
\begin{equation}
\label{deltastar}
    \Delta_k(v)< b(v):=\min\{\frac{\epsilon(v)}{2c_g},\frac{\epsilon(v)}{2\Theta},\frac{\epsilon(v)}{2c_{b}},\frac{\epsilon(v)(1-\eta_1)}{4c_{\tilde{\Phi}}}\}
\end{equation}
Let us denote by $\hat{K}(v)$ the set of all indices from $K(v)$ such that $k\geq k_3(v)=\max\{k_1(v),k_2(v)\}$. Thus, for all $k \in \hat{K}(v)$ we have fully linear models,   $\omega(X_k(v))\geq\epsilon(v)$ and $\Delta_k(v)$ is small enough. Furthermore, from Lemma \ref{margfun} and \eqref{deltastar} we obtain, 
$$|\omega(X_k(v))-\omega_m(X_k(v))|\leq c_g\Delta_k(v)\leq\frac{\epsilon(v)}{2}$$
and $\omega_m(X_k(v))\geq\frac{\epsilon(v)}{2}\geq\Delta_k(v)\Theta$. Moreover, for all $k\in\hat{K}(v)$, the condition \eqref{deltabound} 
%from Lemma \ref{thm} 
is satisfied
\begin{eqnarray*}
\Delta_k (v) & < & \min \{\frac{\epsilon(v)}{2c_g},\frac{\epsilon(v)}{2\Theta},\frac{\epsilon(v)}{2c_{b}},\frac{\epsilon(v)(1-\eta_1)}{4c_{\tilde{\Phi}}}\}  \\ 
& \leq & \min\{\frac{\omega_m(X_k(v))}{c_{b}},\frac{\omega_m(X_k(v))(1-\eta_1)}{2c_{\tilde{\Phi}}}\}.
\end{eqnarray*}
Thus we conclude  that ${\cal{R}}_k(v)\geq\eta_1$ which together with $\omega_m(X_k(v))\geq\Delta_k(v)\Theta$  implies  that all  iterations in $\hat{K}(v)$ are successful iterations. Therefore for  all $k\in\hat{K}(v)$ there holds $\Delta_{k+1}(v)=\Delta_k(v)\gamma_2>\Delta_k(v)$.
Let us define 
$$r_k(v):=\log_{\gamma_2}((b(v))^{-1}\Delta_k(v))$$
where $b(v)$ is defined in \eqref{deltastar}.
Notice that for $k\geq k_3(v), \Delta_k(v)<b(v)$, hence $\gamma_2^{r_k(v)}<1$ and   $r_k(v)<0$. Moreover,
$$r_{k+1}(v)=\log_{\gamma_2}((b(v))^{-1}\Delta_{k+1}(v))=\left\{\begin{array}{rcl}r_k(v)+1& if &\Delta_{k+1}(v)=\gamma_2\Delta_k(v)\\r_k(v)-1&  if &\Delta_{k+1}(v)=\frac{\Delta_k(v)}{\gamma_2}\end{array}\right.$$
Therefore, we have $r_{k+1}(v)=r_k(v)+1$ for all $k\in\hat{K}$. Notice that the increase of $r_k(v)$ can also happen  for some  $k\geq k_3(v)$ even if $k \notin \hat{K}(v)$. On the other hand, the increase of  $W_k(v)$ defined in  \eqref{randomwalk} is possible if and only if the models $\tilde{M}_{k,i}(v), i=1,...,q$, are fully linear. That means that for all $k\geq k_3(v)$, the increase $W_{k+1}(v)=W_k(v)+1$ happens only if $k \in \hat{K}(v)$, while in the remaining iterations $k\geq k_3(v), k \notin \hat{K}(v)$  we have $W_{k+1}(v)=W_k(v)-1$.  Thus for all $k>k_3(v)$  the increase of $r_k(v)$ happens in the same or bigger number of iterations than the increase of $W_k(v)$ and we conclude that  the following must hold
$$r_k(v)-r_{k_3}(v)\geq W_k(v)-W_{k_3}(v).$$
Since  \eqref{wkinf} holds almost surely, without loss of generality  we conclude that  $\limsup_{k\rightarrow\infty}W_k(v)=\infty$ and thus 
$\limsup_{k\rightarrow\infty}r_k(v)=\infty$ which contradicts the fact that $r_k(v)<0$ for all $k\geq k_3(v)$.
}
 \end{proof}
 
\begin{proposition} \label{sumdeltak}
Suppose that the assumptions of Theorem \ref{liminf} hold. If there exists an infinite subsequence $K\subseteq \mathbb{N}$ such that $\omega (X_k) \geq \varepsilon>0$ for all $k \in K$ then there holds $$E(\sum_{ k \in K} \Delta_k)<\infty.$$
Moreover,  $\sum_{ k \in K} \Delta_k<\infty$ a. s.
    
\end{proposition}
\begin{proof}
Let us observe iterations $k \in K$. {\color{red} We analyze the two possible scenarios regarding full linearity separately.} 

{\color{red} Let us consider $E(\Psi_{k+1}-\Psi_k|I_k, \mathcal{F}_k)$ first.  
According to \eqref{deltaksum} we have $\lim_{k \to \infty} \Delta_k=0$ a.s.  which, conditioning on  $I_k$, } together with Lemma \ref{margfun} implies the existence of $\tilde{\varepsilon}>0$ such that $\omega_{m}(X_k)\geq \tilde{\varepsilon}$ for each $k \in K$ sufficiently large. The above further implies that $\omega_{m}(X_k)\geq \Theta \Delta_k>c_b\Delta_k$  for each $k \in K$ sufficiently large, and thus, due to Lemma \ref{thm}, ${\cal{R}}_k \geq \eta_1$ for each $k \in K$ sufficiently large. Without loss of generality, let us assume that $K$ contains only those sufficiently large {\color{red} $ k $ } such that all the above holds. {\color{red} Then,  an arbitrary iteration  $k \in K$ under $I_k$ is a successful iteration of SMOP. In other words, $I_k $ implies $S_k$.  Thus, due to \eqref{final}, for each $k \in K$ there holds 
\begin{eqnarray*}
    \label{finalnew}
&&E(\phi(X_{k+1}) - \phi(X_k)|I_k,\F_k) \\\nonumber 
& \leq & E(2 c_f \Delta_k^2 - \frac{1}{2} w_m(X_k) \min \{\Delta_k,\frac{w_m(X_k)}{c_b}\} |I_k,\mathcal{F}_{k})\nonumber\\
& \leq & 2 c_f \Delta_k^2 - \frac{1}{2} \tilde{\varepsilon} \Delta_k =-\Delta_k (\frac{\tilde{\varepsilon}}{2} -  2 c_f \Delta_k).
\end{eqnarray*}

Once again, assuming that $k \in K$ are all sufficiently large, we obtain $E(\phi(X_{k+1}) - \phi(X_k)|I_k,\F_k)\leq -c_9 \Delta_k$, where $c_9=\frac{\tilde{\varepsilon}}{4} >0$, and thus we conclude 
\begin{eqnarray*} \label{newIk}
E(\psi_{k+1}-\psi_k|I_k,\F_k)&=&E(\nu(\phi(X_{k+1})-\phi(X_k))+(1-\nu)(\gamma_2^2-1)\Delta_k^2|I_k,\F_k)\\\nonumber
&\leq&  -c_{10} \Delta_k+c_{11} \Delta_k^2,
\end{eqnarray*}
}
where $c_{10}=\nu c_9>0$, and $c_{11}=(1-\nu)(\gamma_2^2-1)>0$. 

{\color{red} Now, let us consider $E(\Psi_{k+1}-\Psi_k|\bar{I}_k, \mathcal{F}_k)$. Considering \eqref{c1c} and \eqref{bezBkpsi} we conclude that \begin{equation*} \label{newIkbar} E(\phi(x_{k+1})-\phi(x_k) |\bar{I}_k, \mathcal{F}_k)\leq c_5
 \end{equation*} and thus 
 \begin{equation*} \label{bezBkpsinew} E(\Psi_{k+1}-\Psi_k|\bar{I}_k, \mathcal{F}_k)\leq c_5 +(1-\nu)(\gamma_2^2-1)\Delta_k^2=c_5+c_{11}\Delta_k^2.
 \end{equation*}}

 Now, combining both cases regarding $I_k$ we conclude that for all $k \in K$ there holds 
 \begin{eqnarray*}
     \label{newK}
     E(\Psi_{k+1}-\Psi_k| \mathcal{F}_k)&=& P(I_k| \mathcal{F}_k )E(\Psi_{k+1}-\Psi_k|I_k, \mathcal{F}_k)\\\nonumber
     &+& P(\bar{I}_k| \mathcal{F}_k )E(\Psi_{k+1}-\Psi_k|\bar{I}_k, \mathcal{F}_k)\\\nonumber &\leq& P(I_k| \mathcal{F}_k )(-c_{10} \Delta_k+c_{11} \Delta_k^2)\\\nonumber
     &+& P(\bar{I}_k| \mathcal{F}_k )(c_5+c_{11} \Delta_k^2)\\\nonumber &\leq& -\bar{\alpha}^q c_{10} \Delta_k+c_{11} \Delta_k^2+c_5 (1-\alpha_k^q)\\\nonumber &=:& -c_{12} \Delta_k+c_{11} \Delta_k^2+c_5 (1-\alpha_k^q).
 \end{eqnarray*}
 where $c_{12}=\bar{\alpha}^q c_{10}>0.$
 Applying the unconditional expectation we conclude that for all $k \in K$ there holds 
 \begin{equation*}
     \label{newE1}
     E(\Psi_{k+1}-\Psi_k)\leq -c_{12} E(\Delta_k)+c_{11} E(\Delta_k^2)+c_5 (1-\alpha_k^q).
 \end{equation*}
 On the other hand, \eqref{newEN} holds in all the iterations $k \in \mathbb{N}$ and applying the expectation we obtain 
 \begin{equation*}
     \label{newE2}
     E(\Psi_{k+1}-\Psi_k)\leq -c_8 E(\Delta_k^2) +c_7  (1-\alpha_k^q)\leq c_7  (1-\alpha_k^q).
 \end{equation*}
 Let us denote $\{k\}_{k \in K}=\{k_{(j)}\}_{j \in \mathbb{N}}$. Then, for each $j \in \mathbb{N}$ there holds 
 \begin{eqnarray*}
     \label{newKj}
     E(\Psi_{k_{(j+1)}}-\Psi_{k_{(j)}})&=& E(\Psi_{k_{(j)}+1}-\Psi_{k_{(j)}})+\sum_{i=k_{(j)+1}}^{k_{(j+1)}-1}E(\Psi_{i+1}-\Psi_{i})\\\nonumber
     &\leq & -c_{12} E(\Delta_{k_{(j)}})+c_{11} E(\Delta_{k_{(j)}}^2)+c_5 (1-\alpha_{k_{(j)}}^q)
     \\\nonumber
     &+ & c_7 \sum_{i=k_{(j)+1}}^{k_{(j+1)}-1}(1-\alpha^q_i)\\\nonumber
     &\leq & -c_{12} E(\Delta_{k_{(j)}})+c_{11} E(\Delta_{k_{(j)}}^2)
     \\\nonumber
     &+ & c_{13} \sum_{i=k_{(j)}}^{k_{(j+1)}-1}(1-\alpha^q_i),
 \end{eqnarray*}
 where $c_{13}=\max\{c_5, c_7\}$.
 Therefore, for every $m \in \mathbb{N}$ there holds 
 \begin{eqnarray*}
     \label{newEm}
     E(\Psi_{k_{(m)}}-\Psi_{k_{(0)}})\leq 
     -c_{12} E(\sum_{j=0}^{m-1} \Delta_{k_{(j)}})+ c_{11} E(\sum_{k=0}^{\infty}\Delta_k^2)+c_{13} c_{\alpha}. 
 \end{eqnarray*}
 Letting $m$ tend to infinity and using \eqref{deltaksum} together with the assumption of $\Psi$ being bounded from below, we conclude that 
 $$E(\sum_{k \in K} \Delta_k)= E(\sum_{j=0}^{\infty} \Delta_{k_{(j)}})<\infty.$$
 Finally, assuming that $\sum_{k \in K} \Delta_k=\infty$ with some positive probability yields the contradiction with the previous inequality and we conclude that \\ $\sum_{k \in K} \Delta_k<\infty$ a.s, which completes the proof. 
\end{proof}

\begin{theorem} \label{glavnate}
Suppose that the assumptions of Theorem \ref{liminf} hold. Then a.s. $$\lim_{k\rightarrow\infty}\omega(X_k)=0.$$ 
\end{theorem}
\begin{proof}
Suppose the contrary, that with positive probability there exists a subsequence $\omega(X_k)$ which does not converge to zero. More precisely, there exists $\hat{\Omega}\subset\Omega$ such that $P(\hat{\Omega})>0$  and for all $v \in \hat{\Omega}$ there exist $\epsilon(v)>0$ and $K(v) \subseteq \mathbb{N}$ such that for all $k \in K(v)$ there holds $$\omega(X_k(v)) \geq 2\epsilon(v).$$ 
On the other hand,  Theorem \ref{liminf} implies that {\color{red}  for almost every $v \in \Omega$ there exists    $K_l(v)\subset \mathbb{N}$ such that $\lim_{k \in K_l(v)} \omega(X_k(v))=0$. } Therefore, without loss of generality, we assume that $\omega(X_k(v))<\epsilon(v)$ for all $k \in K_l(v)$ and almost all $v \in \Omega$. {\color{red}Now, let us consider an arbitrary $v \in \hat{\Omega}$.}
Since both $K(v)$ and $K_l(v)$ are infinite, there exists $K_s (v)\subseteq K_l(v)$ such that for each $k \in K_s(v)$ we have both $$\omega(X_k(v))<\epsilon(v) \quad \mbox{and } \quad \omega(X_{k+1}(v))\geq\epsilon(v).$$ In other words, we observe the subsequence $K_s(v)$ of $K_l(v)$ such that $k \in K_l(v)$ and the subsequent iteration does not belong to $K_l(v)$, i.e., $k+1 \notin K_l(v)$. Furthermore, let us observe the pairs 
$(k_{j,1}(v),k_{j,2}(v)), j=1,2,...$, where $k_{j,1}(v)\in K_s(v)$ and $k_{j,2}(v)$ is the first $k >k_{j,1}$ that belongs to $K(v)$, i.e.,  $$\omega(X_{k_{j,1}}(v))<\epsilon(v) \quad \mbox{and } \quad  
\omega(X_{k_{j,2}}(v))\geq 2\epsilon(v), \quad  j=1,2,...$$ 
This also implies that for each $j \in \mathbb{N}$ there holds 
\begin{equation}
\label{newwdist} |\omega(X_{k_{j,1}}(v))-\omega(X_{k_{j,2}}(v))|\geq \epsilon(v).
\end{equation}
{\color{red} for all $v \in \hat{\Omega}$.} Notice that, by the construction of {\color{red} the above } subsequences, $k_{j,1}(v)$ represents the last iteration prior to $k_{j,2}(v)$ such that $\omega(X_{k_{j,1}}(v))<\epsilon(v)$. Therefore, if $k_{j,2} (v)\neq k_{j,1}(v)+1$, for all the intermediate iterations $k \in \{k_{j,1}(v)+1,...,k_{j,2}(v)-1\}$ and all $j \in \mathbb{N}$ there holds 
$$\omega(X_k(v))\geq \epsilon(v).$$
Moreover, Proposition \ref{sumdeltak} implies that  
\begin{equation}
    \label{newpropsum}
    \sum_{j=1}^{\infty}  
    \mbox{ }\sum_{i=k_{j,1}(v)+1}^{k_{j,2}(v)-1}\Delta_i(v) < \infty  
\end{equation}
{\color{red} holds for almost every $v \in \Omega$, and thus almost every $v \in \hat{\Omega}$. }
Notice that  $k_{j,1}(v)$ must be successful iteration of SMOP for all $j \in \mathbb{N}$, since the marginal function changes only when the step is accepted, and thus $\Delta_{k_{j,1}+1}(v)=\gamma_2\Delta_{k_{j,1}}(v)>\Delta_{k_{j,1}}(v)$. Therefore,  for all $j \in \mathbb{N}$, we have 
\begin{eqnarray}
    \label{newk1k2}
&&\|X_{k_{j,1}}(v)-X_{k_{j,2}}(v)\| \\\nonumber 
&=&\|X_{k_{j,1}}(v)-X_{k_{j,1}+1}(v)+X_{k_{j,1}+1}(v)-...-X_{k_{j,2}}(v)\|\\\nonumber &\leq& \sum_{i=k_{j,1}(v)}^{k_{j,2}(v)-1}\|X_i(v)-X_{i+1}(v)\|\leq 
\sum_{i=k_{j,1}(v)}^{k_{j,2}(v)-1}\Delta_i(v)=\Delta_{k_{j,1}}(v)+\sum_{i=k_{j,1}(v)+1}^{k_{j,2}(v)-1}\Delta_i(v)\\\nonumber &\leq&\Delta_{k_{j,1}+1}(v)+\sum_{i=k_{j,1}(v)+1}^{k_{j,2}(v)-1}\Delta_i(v)\leq 2\sum_{i=k_{j,1}(v)+1}^{k_{j,2}(v)-1}\Delta_i(v).
\end{eqnarray}
Thus, summing over $j$ we conclude that  
$\sum_{j=1}^{\infty} \|X_{k_{j,1}}(v)-X_{k_{j,2}}(v)\|< \infty$ for almost every $v \in \hat{\Omega}$
due to \eqref{newpropsum}.  
This further implies that $\lim_{j \to \infty}\|x_{k_{j,1}}(v)-x_{k_{j,2}}(v)\|=0$ {\color{red} for almost every $ v \in \hat{\Omega}$. } However, this further implies that  {\color{red} for almost every $ v \in \hat{\Omega}$ we have  } $\lim_{j \to \infty}|\omega(X_{k_{j,1}}(v))-\omega(X_{k_{j,2}}(v))|=0$  due to Lemma \ref{lmarginal}, d),  which is a contradiction with \eqref{newwdist}. 
\end{proof}

\section{Numerical results}
{\color{red}\subsection{Experiment overview}
 Several experiments are reported in this paper in order to demonstrate the efficiency of the first order SMOP algorithm. The first set of experiments focuses on the machine learning (ML) application, and the concept of model fairness, as discussed in \cite{LV}. In our tests, the problem of minimizing the logistic regression loss function is reformulated into a multi-objective optimization problem by splitting the dataset based on sensitive features. Here, the SMOP algorithm from Section \ref{Alg1} is compared with the deterministic trust region \cite{VOS} (DMOP), and the stochastic multi-gradient \cite{LV} (SMG). Additionally, a version SMOP-S which uses a  subsampling technique which does not satisfy the theoretical assumption is considered. This version proves to be rather efficient as the theoretical bounds are rather demanding. The Pareto front finding technique, which can be seen in \cite{CMV}, \cite{LV} is also employed. This allows the comparison of SMOP/SMOP-S with a deterministic trust region DMOP \cite{VOS}, which is presented using different metrics.

The second set of experiments focuses on multi objective benchmark test problems from \cite{NUM} and \cite{CMV}. Four different problems are considered, each one having a differently shaped Pareto front. Comprehensive representation of the Pareto front for convex, disconnected, mixed, and concave shapes is provided. In this part, a thorough comparison has been made between stochastic and deterministic trust region, and several different metrics have been used to access the quality of the resulting Pareto front.   
\subsection{Experimental configuration}
All considered experimental problems have two component functions that are in the form of a finite sum. Thus, the approximation of function and gradient values is achieved by exploiting its structure using an adaptive subsampling strategy. We demonstrate the first order algorithm, hence we assume  that $H^i_k=0$  for all $i=1,...,q$. Randomly, a subset  $\mathcal{N}_i^k\subseteq\mathcal{N}_i$, $i=1,2$ is sampled following a rule motivated by  the result from \cite[Lemma 4]{RM}. Namely,  for each subgroup  we get $P(|f_i(x_k)-\tilde{f}_i(x_k)|\leq\delta_k^2)\geq\alpha_k$ provided that 
\begin{equation*}
\label{subsamp}
    |\mathcal{N}_i^k|\geq \frac{{F_i(x_k)}^2}{\delta_k^4}\left(1+\sqrt{8\log(\frac{1}{1-\alpha_k})}\right)^2
\end{equation*}
where $ F_i(x_k)$ is the upper bound of $|f_i(x_k)|$.

Although this kind of bound is not easy to obtain in general, for logistic regression problems it is possible to use e.g. 
$$F_i(x_k)=e^{\|x_k\|}\max_{j} \|a_j\| +\log(2)+\frac{\lambda_i}{2} \|x_k\|^2.$$

In our tests, the upper-bound $F_i$ is replaced by a constant, such that $|\mathcal{N}_i^0|=N^{min}_{i}=\max\{0.01|\mathcal{N}_i|,2\}$, hence for SMOP the sample size behaves like 
$$  \mathcal{O}\left(\frac{1}{\delta_k^4}\left(1+\sqrt{8\log(\frac{1}{1-\alpha_k})}\right)^2\right).$$
 The probability parameters are defined as $\alpha_k=\sqrt{1-0.99^k}$, which satisfies theoretical assumptions.  We have set $\delta_{min}=10^{-4}$, $\delta_{max}=8$, $\gamma_1=0.5,\gamma_2=2$ and $\delta_0=1$ in all experiments. This strategy showed large updates in subsampling size, due to $\delta_k^4$ being present in the denominator, hence we introduced SMOP-S, a version of SMOP which controls the increases and decreases of the updates. The subsampling sizes of SMOP-S follow the update rule:
 \begin{equation}
     |\mathcal{N}_i^k|=\max\{\min\{j_kS_i,|\mathcal{N}_i|\},N_{i}^{min}\},
 \end{equation}
 where $\delta^4_k=0.5^{j_k}$, i.e. at each iteration $j_k$ is calculated as $j_k=\log_{2}{\frac{1}{\delta_k^4}}$ for $\delta_k\leq1$, and $S_i= |\mathcal{N}_i|/16$. For $\delta_k \geq 1$ the subsampling size is minimal, $N^{min}_{i}=\max\{0.01|\mathcal{N}_i|,2\}$. %Since we have chosen $\gamma_1=0.5$, the algorithm needs 4 unsuccessful iterations in a row to reach the full sample. 
 This method connects the trust region radius to the stochastic average approximation error, that is, the approximation error follows the movement of $\delta_k$.
 
 The randomization is done by randomly and uniformly shuffling the indices $1,...,|\mathcal{N}_i|$ without replacement at the start, and then slicing the first $N_i^k$ elements at each iteration.

 In Step 2 of Algorithm 1, a descent direction needs to be calculated using approximate function and gradient values. A common strategy in finding a descent direction in multi-objective optimization is to solve the problem
\begin{equation*}
\min_{\beta\in\mathbb{R},d\in\mathbb{R}^n}\beta+\frac{1}{2}\|d\|^2,\quad s.t.\quad \langle \nabla f_i(x_k),d\rangle\leq\beta,\quad i=1,...,q.
\end{equation*}
If $x_k$ is Pareto critical, then the solution is $d_k=0$, $\beta_k=0$, and if it is not Pareto critical, then $\langle \nabla f_i(x_k),d_k\rangle\leq\beta<0$ for all $i$, see \cite[Lemma 1]{FS}.
The dual of this problem can be written as
\begin{equation*}
\min_{c_1,c_2\in\mathbb{R}}\|c_1\nabla f_1(x_k)+c_2\nabla f_2(x_k)\|^2    ,\quad s.t.\quad c_1,c_2\geq 0,\quad c_1+c_2=1,
\end{equation*}
 see \cite[Subsection 7]{FS}. This form offers an easily computable solution for two component functions, hence it is convenient for implementation. 
 %, especially for deterministic methods.  
 Since the true gradients are unknown, by replacing these values with approximations, the dual problem becomes
\begin{equation}
\label{dualapp}
\min_{c_1,c_2\in\mathbb{R}}\|c_1g_1(x_k)+c_2g_2(x_k)\|^2    ,\quad s.t.\quad c_1,c_2\geq 0,\quad c_1+c_2=1.
\end{equation}
This dual problem produces a stochastic multi-gradient direction $d_k=-c^*_1g_1(x_k)-c_2^*g_2(x_k)$, see \cite{LV}. In all of our experiments, scaling this stochastic multi-gradient direction onto the trust region, i.e. using $\frac{d_k}{\|d_k\|}\delta_k$ as the direction, showed good performance. The direction calculated by solving \eqref{dualapp} satisfies \eqref{Cauchy}, since  $\tilde{m}_k(d_k)\leq \tilde{m}_k(d_k^c)$, where $d_k^c$ is the Cauchy direction defined in Lemma \ref{sufred}. Projecting it onto the trust-region may however only achieve a fraction of the Cauchy decrease \eqref{Cauchy}. 

When testing whether the second condition in Step 3 holds, i.e. $\omega_m(x_k)\geq\Theta\delta_k$, it is convenient to check instead if $-\max_{i}\langle g_i(x_k),\frac{d_k}{\|d_k\|}\rangle>\Theta\delta_k$. Since $\omega_m(x_k)\geq-\max_{i}\langle g_i(x_k),\frac{d_k}{\|d_k\|}\rangle$ holds, if the right hand side of the inequality is large enough, the desired condition is implied. This allows SMOP to skip the calculation of $\omega_m(x_k)$ entirely.

The implementation of SMG was done using the configuration showed in \cite{LV} and the code available on GitHub page \cite{gh} of the same authors. For DMOP we used the same configuration as in SMOP.

\subsection{Finding the Pareto front}
Using a front finding technique,  \cite{CMV}, \cite{LV}, we approximate the Pareto front for different problems. First, a set of random points $\mathcal{L}$ is taken as an approximation of the front. At each iteration, it is expanded by generating perturbed points around the existing elements in the set. The predetermined number of iterations of the chosen algorithm (SMOP, SMOP-S or DMOP) is then applied to the existing points, in order to come closer to the set of optimal points, after which the results are also added to the approximation set $\mathcal{L}$. To refine the approximation, all dominated points are removed from the set, leaving only the non dominated points to serve as the Pareto approximation. The point $x\in\mathcal{L}$ is said to be dominated if there exists $y\in\mathcal{L}$, such that $f(y)<f(x)$, i.e. $f_i(y)<f_i(x)$ for $i=1,...,q$.  The procedure is formally described in the following way: \\
\textbf{Algorithm 2.}\cite{CMV}\\
\textit{(Pareto front procedure)}
\begin{itemize}
\item[]\textbf{Step 0}. Generate the initial Pareto front $\mathcal{L}_0.$ Select parameters $n_p,n_q,n_r\in\mathbb{N}$.
\item[]\textbf{Step 1}. Set $\mathcal{L}_{k+1}=\mathcal{L}_{k}$. For each point $x$ in $\mathcal{L}_{k+1}$, add $n_r$ points to $\mathcal{L}_{k+1}$ from the neighborhood of $x$.
\item[]\textbf{Step 2}. For each point $x$ in $\mathcal{L}_{k+1}$, repeat $n_p$ times: Apply $n_q$ iterations of chosen method  with $x$ as a starting point. Add the final iteration to $\mathcal{L}_{k+1}$
\item[]\textbf{Step 3}. Remove all dominated points from $\mathcal{L}_{k+1}$. Set $k=k+1$ and go to Step 1.  
\end{itemize}

The only difference between the procedures is in Step 2, where we choose either SMOP, SMOP-S or DMOP as the underlying algorithm. All three procedures were implemented; however, since  the SMOP-S version was superior to SMOP in terms of time while yielding similar quality fronts, only the comparison between the procedure with SMOP-S and DMOP will be presented and further discussed.  By selecting different configurations $n_p,n_q$ and $n_r$, it is possible to control the resulting front and get either a sparser front with fewer details or a denser one. In our experiments, the following parameters were set the same for SMOP-S and DMOP: $n_q=5, n_p=1, n_r=10$. In order to reduce the number of points generated each iteration, a strategy from \cite{LV} is employed. At Step 1,  $n_r$ points is generated only for the pair of points with largest holes in the Pareto front. This requires sorting values of both component functions, and finding the original indices corresponding to the largest differences of consecutive function values in the sorted list. 

The starting Pareto size was always 30 points, while the exiting criteria was set to be 1500 points, which was always the terminating factor. In both procedures, the true values of the functions were available to determine whether the points are dominated or not in Step 3, however in SMOP-S procedure, when applying $n_q$ iterations in Step 2, the true function values were not visible, i.e. the $n_q$ steps were calculated using the approximated values. For all problems, we measured the full time needed to find the Pareto front. Since the trust region radius was initialized as $\delta_0=1$ everywhere, after 5 iterations of each procedure, it is halved. This helped both SMOP-S and DMOP reach higher quality fronts.

To measure the quality of the approximated front, we have opted for three different metrics, as seen in \cite{CMV}, \cite{LV}. The Purity estimates the percentage of true nondominated points in the Pareto approximation. It does so by comparing the observed front to the combined front of all available Pareto approximations . If $\mathcal{L}_S$ and $\mathcal{L}_D$ are fronts generated by SMOP-S and DMOP respectively, and $\mathcal{L}_{SD}$ is the front we get by removing the nondominated points from $\mathcal{L}_S\cup\mathcal{L}_D$, then the Purity is given by:
\begin{equation}
    \label{smoppurity}
    P_S:=\frac{|\mathcal{L}_S\cap\mathcal{L}_{SD}|}{|\mathcal{L}_S|},\quad P_D:=\frac{|\mathcal{L}_D\cap\mathcal{L}_{SD}|}{|\mathcal{L}_D|}.
\end{equation}
The closer to $1$ this ratio is the better for the procedure, since a higher percentage of points is 'truly' on the Pareto front. The Spread metrics are used to determine how well the points are spread throughout the Pareto front approximation, as the name suggests. The $\Gamma$-spread metric measures the maximum size of the hole in the Pareto front. To calculate it the function values within the front approximations are sorted in a nondecreasing order for both components, $f_i(x^0)\leq ...\leq f_i(x^{|\mathcal{L}|}).$
The measure is then calculated as:
\begin{equation}
\Gamma_S:=\max_{i=1,2}\big(\max_{j=1,...,|\mathcal{L}_S|}l_{i,j}\big),\quad \Gamma_D:=\max_{i=1,2}\big(\max_{j=1,...,|\mathcal{L}_D|}l_{i,j}\big)    
\end{equation}
where $l_{i,j}=f_i(x^{j+1})-f_i(x^j)$.
The second $\Delta$-spread metric measures how well the points are distributed in the approximated front. It is computed as:
\begin{equation}
\Delta_S:=\max_{i=1,2}{\huge (}\frac{l_{i,0}+l_{i,|\mathcal{L}_S|}+\sum_{j=1}^{|\mathcal{L}_S|}|l_{i,j}-\overline{l}_i|}{l_{i,0}+l_{i,|\mathcal{L}_S|}+(|\mathcal{L}_S|-1)\overline{l}_i}{\huge)},
\end{equation}
\begin{equation}
    \Delta_D:=\max_{i=1,2}{\huge (}\frac{l_{i,0}+l_{i,|\mathcal{L}_D|}+\sum_{j=1}^{|\mathcal{L}_D|}|l_{i,j}-\overline{l}_i|}{l_{i,0}+l_{i,|\mathcal{L}_D|}+(|\mathcal{L}_D|-1)\overline{l}_i}{\huge)}
\end{equation}
where $\overline{l}_i$ is the mean of $l_{i,j}$ for the respective procedure.
For both Spread metrics, the lower the value, the better the distribution of the Pareto front is.

\subsection{Machine learning (logistic regression)}
When handling sensitive data in machine learning, there is a significant risk of developing models that exhibit discriminatory behavior. Unfairness emerges when the performance of a model, measured in terms of accuracy or another metric, varies across subgroups of data. Such differences can have real-world consequences, particularly in applications where the subgroups represent actual individuals, such as in hiring systems, healthcare diagnostics, or judicial decision-making.  By splitting the data based on the sensitive attributes and treating the performance on each subgroup as a separate objective, it is possible to train models that are more balanced.  For more on this topic, and different ways to measure fairness, see \cite{FER1},\cite{FER2}, \cite{FER3},\cite{FER4},\cite{FER5}.

The problem we consider is minimization of a regularized logistic regression loss function, as in \cite{LV}:
\begin{equation}
\label{logreg}
    \min_{x}f(x):=\frac{1}{N}\sum_{j\in N}\log(1+e^{(-y_j(x^Ta_j))})+\frac{\lambda}{2}\|\hat{x}\|^2,i=1,2.
\end{equation}
where $x$ represents model coefficients we are trying to find, $\hat{x}$ coefficient vector without the intercept,  $a_j$ the feature vector of $j-th$ sample, $y_j$ its respective label, $N$ the training set size, and $\lambda$ is the regularizer.

In order to create a multi-objective problem, we choose a feature, split the data with respect to it, and create a function for each subgroup. For the sake of simplicity, for each dataset, two subgroups $G_1$ and $G_2$ were made.
The loss functions for such problem are:
$$f_i(x)=\frac{1}{|\mathcal{N}_i|}\sum_{j\in \mathcal{N}_i}\log(1+e^{(-y_j(x^Ta_j))})+\frac{\lambda_i}{2}\|\hat{x}\|^2,i=1,2.$$
where $|\mathcal{N}_i|$ is the size of the $i-th$ 
 and hence the problem becomes:
\begin{equation}
    \label{moplogreg}
    \min_{x}(f_1(x),f_2(x))^T
\end{equation}
In the experiment the regularization was set to  $\lambda_i=10^{-3},$ for $i=1,2,$ as in \cite{LV}.   We consider six datasets which vary in size, namely two larger: covtype and mnist, and four smaller sets: svmguide, german, australian and heart, see \cite{LIB}.
The baseline unfairness is demonstrated by solving the scalar problem \eqref{logreg} and evaluating the logistic regression model on the entire dataset, and on groups $G_1$ and $G_2$ separately. Minimization of \eqref{logreg} was done in Python, using stochastic gradient descent, with maximum of 1000 iterations, and diminishing step size. The following Table \ref{tab1} showcases tested unfairness, and how different Pareto points can create fair models. The "Split" column indicates the critical column of the dataset, by which the split was made. "Full" column shows the accuracies of the scalar logistic regression, trained on the full dataset, and evaluated on the training set. The column "$G_1/G_2$" represents the accuracies of of the same model evaluated on groups $G_1$ and $G_2$ respectively. We also present the accuracies of the multi-objective logistic regression models associated with three representative Pareto critical points produced by the SMOP-S. For each model, the training accuracy is evaluated separately for each group, and the results are reported in distinct rows.   It can be seen that moving around the Pareto front  reduction or increase of disparities between groups is possible for the tested datasets and hence desired group accuracies can be achieved. For larger datasets, the front was more difficult to evaluate, hence for covtype dataset, a minimal reduction of disparity is achieved. The split for australian, heart, svmguide and german was done using the known split technique, as in \cite{LV}, however for mnist and covtype the split has been done by experimentally finding a column which creates groups with accuracy disparity. 
\begin{table}[h!]
\begin{tabular}{|c|c|c|c|c|c|c|c|}
\hline
           & $N_1$/$N_2$                                             & Split & Full & $G_1/G_2$                                               & $P_1$                                                     & $P_2$                                                   & $P_3$                                                   \\ \hline

covtype    & \begin{tabular}[c]{@{}c@{}}174048\\ 406964\end{tabular} & 1     & 92.4\%     & \begin{tabular}[c]{@{}c@{}}84.6\%\\ 95.8\%\end{tabular} & \begin{tabular}[c]{@{}c@{}}83.5\% \\ 96.1\%\end{tabular}          & \begin{tabular}[c]{@{}c@{}}83.8\%\\ 95.8\%\end{tabular}         & \begin{tabular}[c]{@{}c@{}}84.8\%\\ 94.8\%\end{tabular}         \\ \hline
mnist      & \begin{tabular}[c]{@{}c@{}}13919\\ 47\end{tabular}      & 106   & 97.6\%     & \begin{tabular}[c]{@{}c@{}}98.0\%\\ 91.5\%\end{tabular} & \begin{tabular}[c]{@{}c@{}}97.5\% \\ 93.6\%\end{tabular}          & \begin{tabular}[c]{@{}c@{}}97.4\%\\ 95.7\%\end{tabular}         & \begin{tabular}[c]{@{}c@{}}97.1\%\\ 97.5\%\end{tabular}         \\ \hline
svmguide3  & \begin{tabular}[c]{@{}c@{}}1182\\ 61\end{tabular}       & 10    & 82.1\%     & \begin{tabular}[c]{@{}c@{}}83.5\%\\ 55.7\%\end{tabular} & \begin{tabular}[c]{@{}c@{}}83.2\% \\ 65.5\%\end{tabular}  & \begin{tabular}[c]{@{}c@{}}80.1\%\\ 73.7\%\end{tabular} & \begin{tabular}[c]{@{}c@{}}74.3\%\\ 75.4\%\end{tabular} \\ \hline
german     & \begin{tabular}[c]{@{}c@{}}630\\ 370\end{tabular}       & 24    & 78.3\%     & \begin{tabular}[c]{@{}c@{}}80.4\%\\ 74.5\%\end{tabular} & \begin{tabular}[c]{@{}c@{}}80.4\% \\ 76.2 \%\end{tabular} & \begin{tabular}[c]{@{}c@{}}79.2\%\\ 75.9\%\end{tabular} & \begin{tabular}[c]{@{}c@{}}78.8\%\\ 77.2\%\end{tabular} \\ \hline
australian & \begin{tabular}[c]{@{}c@{}}468\\ 222\end{tabular}       & 1     & 86.5\%     & \begin{tabular}[c]{@{}c@{}}84.8\%\\ 90.1\%\end{tabular} & \begin{tabular}[c]{@{}c@{}}85.1\%\\ 91.4\%\end{tabular}   & \begin{tabular}[c]{@{}c@{}}86.5\%\\ 90.1\%\end{tabular} & \begin{tabular}[c]{@{}c@{}}86.9\%\\ 89.1\%\end{tabular} \\ \hline
heart      & \begin{tabular}[c]{@{}c@{}}183\\ 87\end{tabular}        & 2     & 85.5\%     & \begin{tabular}[c]{@{}c@{}}81.9\%\\ 93.1\%\end{tabular} & \begin{tabular}[c]{@{}c@{}}79.7\% \\ 94.1\%\end{tabular}  & \begin{tabular}[c]{@{}c@{}}82.5\%\\ 91.3\%\end{tabular} & \begin{tabular}[c]{@{}c@{}}83.1\%\\ 90.8\%\end{tabular} \\ \hline
\end{tabular}

\caption{Dataset parameters and classification accuracies.}
\label{tab1}
\end{table}

As mentioned previously, we first tested SMOP, against DMOP, SMG and SMOP-S in finding critical points. The performance of algorithms is evaluated by measuring the value of the true marginal function $\omega(x_k)$ \eqref{marginal} in terms of the number of function evaluations FEV and time in seconds.  In SMOP, at each iteration we evaluate the function approximation two times, in point $x_k$ and $x_k+d_k$, hence for each function evaluation of $\tilde{f}$, we account $|\mathcal{N}_1^k|+|\mathcal{N}_2^k|$ evaluations. For SMG, at each iteration the number of function evaluation depends on the line-search technique, and achieving the sufficient descent. Both DMOP and SMG use full function values, hence for these algorithms $|\mathcal{N}_1|+|\mathcal{N}_2|$ is added when the function is evaluated. The number of function evaluation showed to be a dominant factor for large datasets and the most time consuming operation for all algorithms, which can be seen in figures for mnist and covtype. However, for smaller datasets, FEV wasn't the only significant operation. For example, the difference between SMOP and DMOP, was also in the acceptance criteria in Step 3. of Algorithm 1, which involves additional scalar products in the stochastic case.

In this manner, we demonstrate how SMOP and SMOP-S  algorithms achieve significant improvements in performance while utilizing only a fraction of the available data for large datasets, and stay comparable with DMOP in the case of smaller datasets. This highlights the stochastic algorithm effectiveness in leveraging limited resources, making it particularly valuable in scenarios where data collection is expensive. The parameters for SMOP and SMOP-S have been chosen the same for all problems: $\eta_1=0.25,\Theta=0.25$, and the starting point $x_0=(0.1,0.1,...,0.1),$ while the rest of the parameters can be seen in Experiment configuration subsection.
The following Figures \ref{fig1.1}, \ref{fig1.2} and \ref{fig1.3} show the behavior of the four algorithms for the datasets from Table \ref{tab1}. For large datasets, Figure \ref{fig1.1} a) and b), the time agrees with the FEV measure. Here, both SMOP and SMOP-S perform efficiently in terms of FEV and time. For datasets seen in Figure \ref{fig1.2} c) and d), the advantage is visible mostly in terms of FEV. It can be seen that for smallest datasets, Figure \ref{fig1.3} e) and f), both versions of SMOP demonstrate slight advantage in terms of FEV and time at start, but later slow down.  The time needed to reach near optimality for all algorithms for c)-f) is less than 0.005 seconds. The subsampling sizes of SMOP show considerable changes, unlike SMOP-S, where the growth is more controlled.

\begin{figure}[htbp]
\centering

\includegraphics[width=0.82\textwidth]{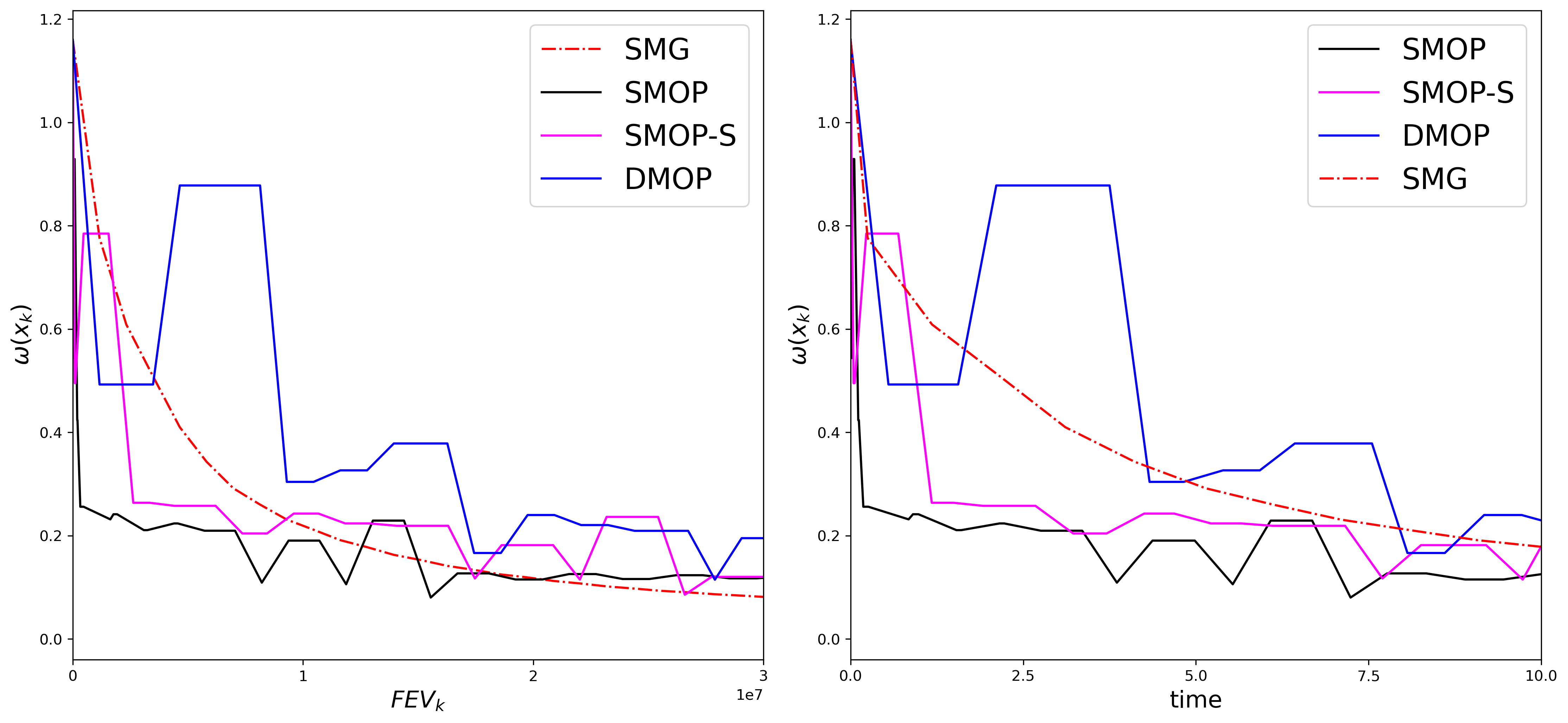} \\
\hspace{-12pt}
\includegraphics[width=0.85\textwidth]{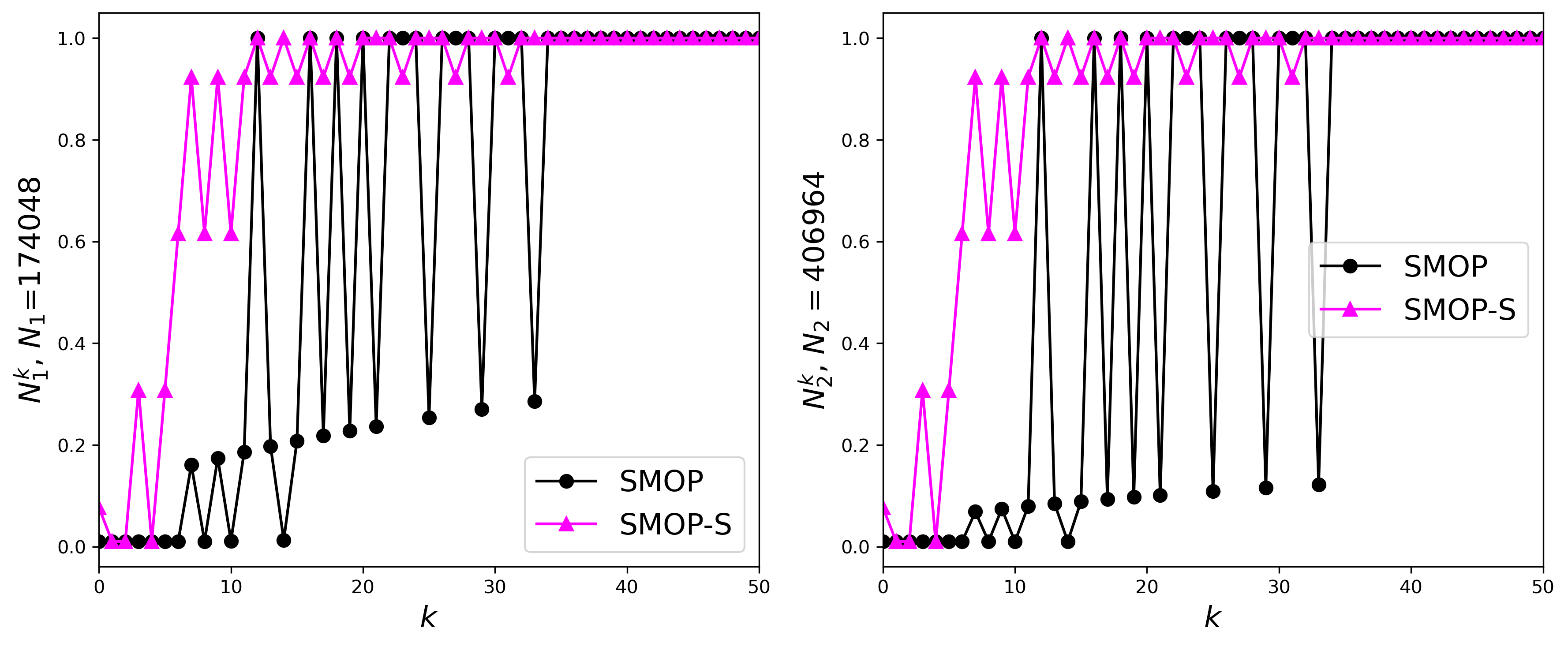} \\
\textbf{a)} \\[0.2cm]

\includegraphics[width=0.82\textwidth]{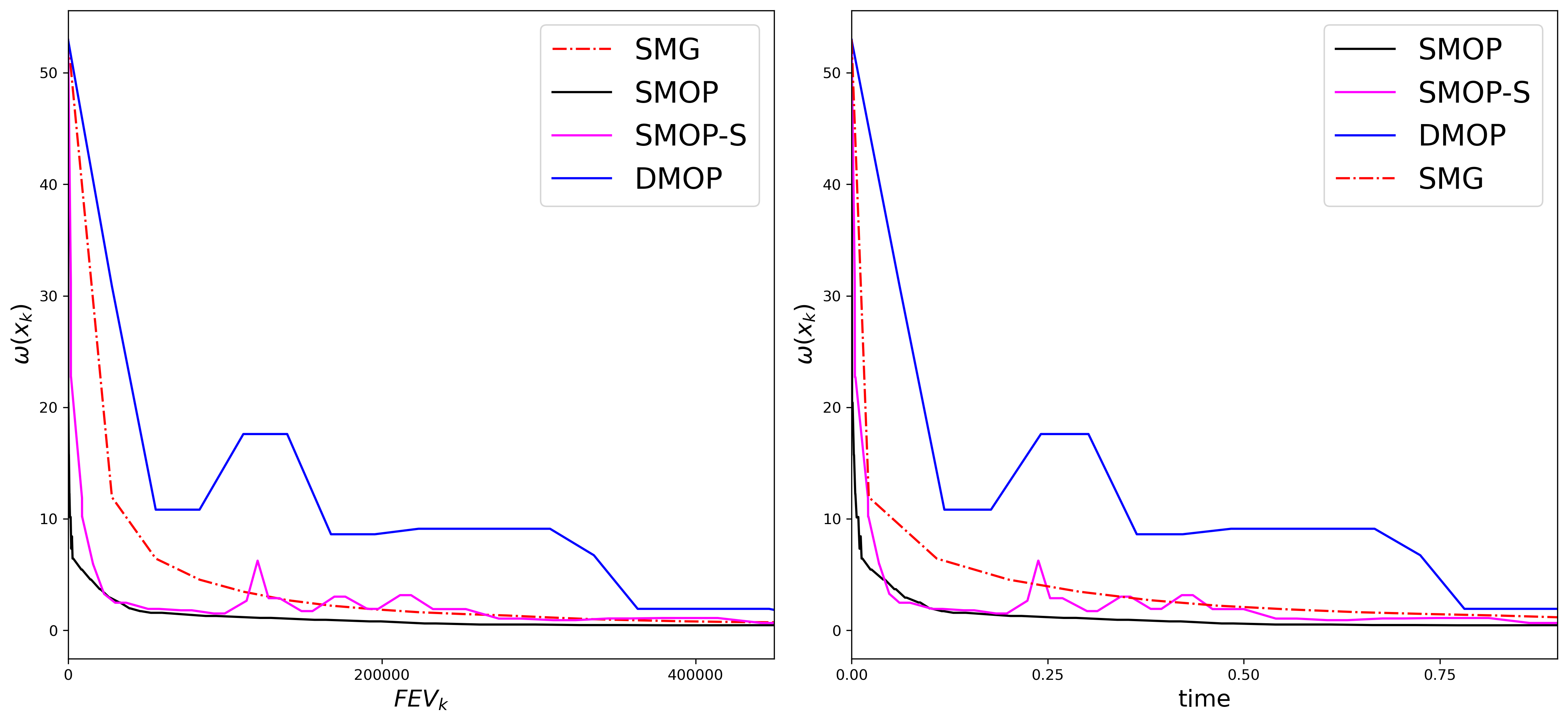} \\
\hspace{-12pt}
\includegraphics[width=0.85\textwidth]{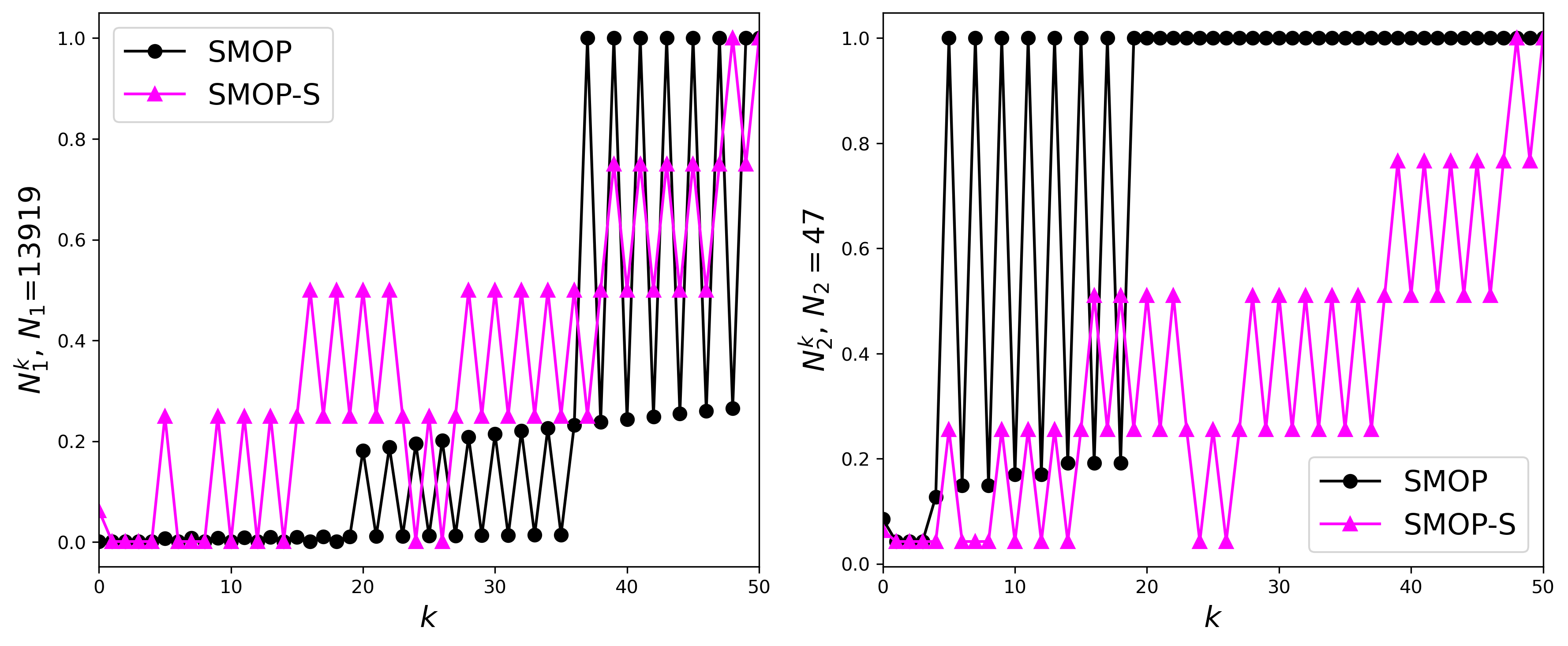} \\
\textbf{b)} \\

\caption{Performance comparison - $\omega(x_k)$ in terms of FEV/time, subsample sizes through iterations for datasets: a) covtype, b) mnist.}
\label{fig1.1}
\end{figure}

\begin{figure}[htbp]
\centering

\includegraphics[width=0.82\textwidth,keepaspectratio]{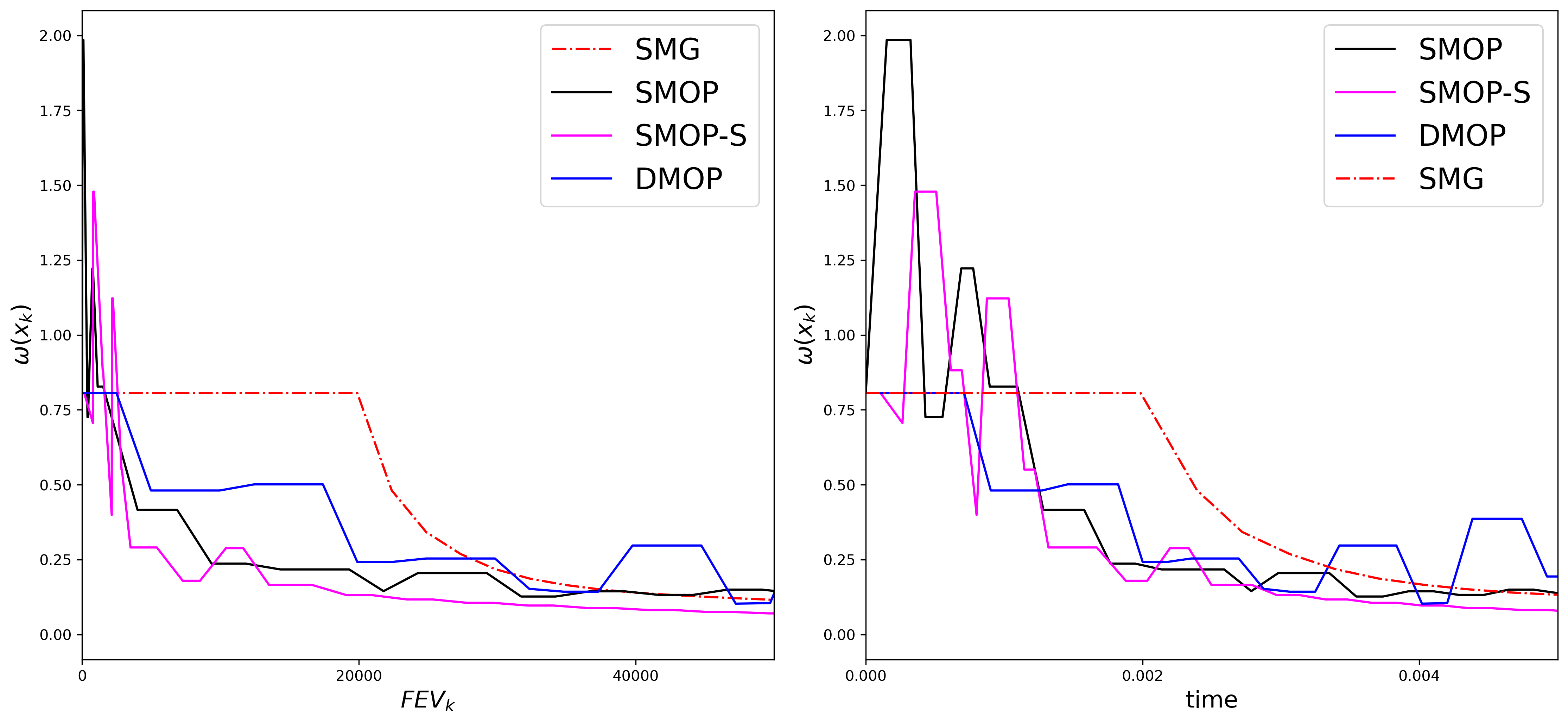} \\
\hspace{-12pt}
\includegraphics[width=0.85\textwidth,keepaspectratio]{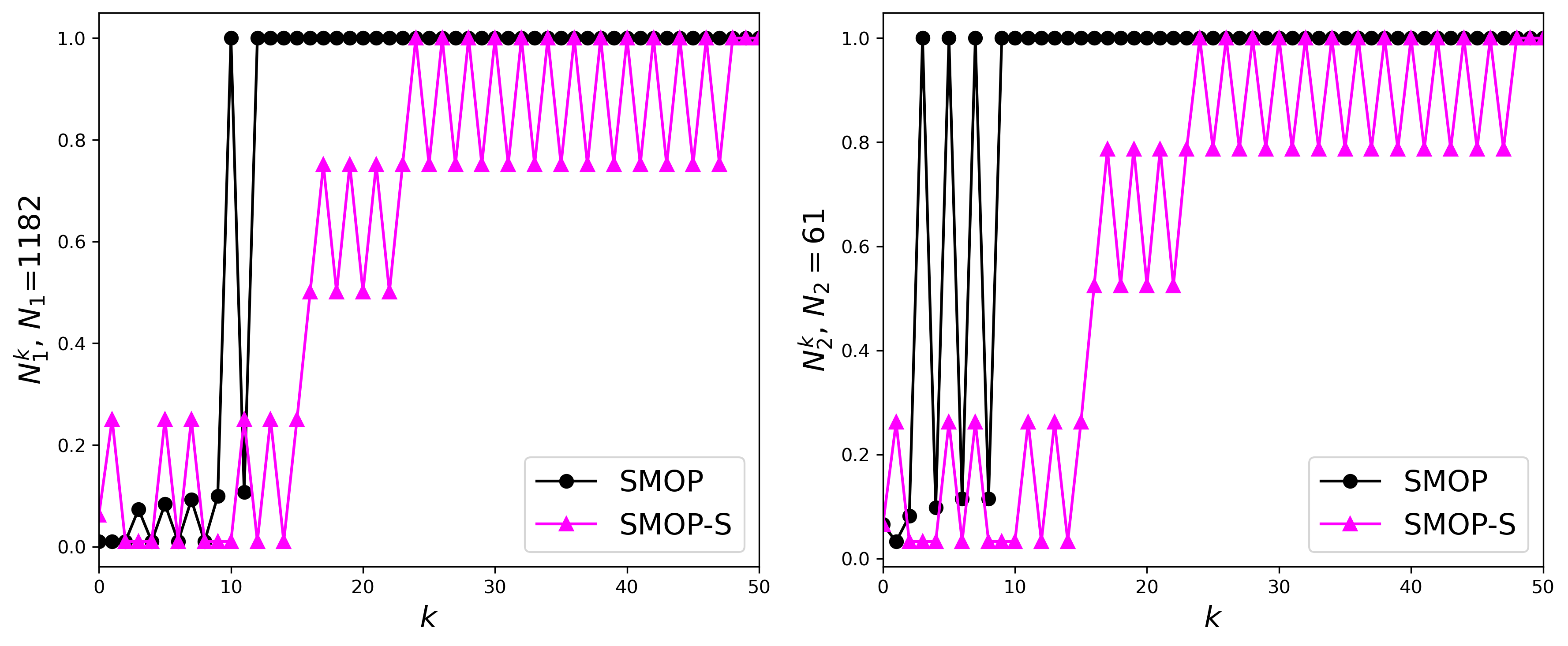} \\
\textbf{c)} \\[0.2cm]

\includegraphics[width=0.82\textwidth,keepaspectratio]{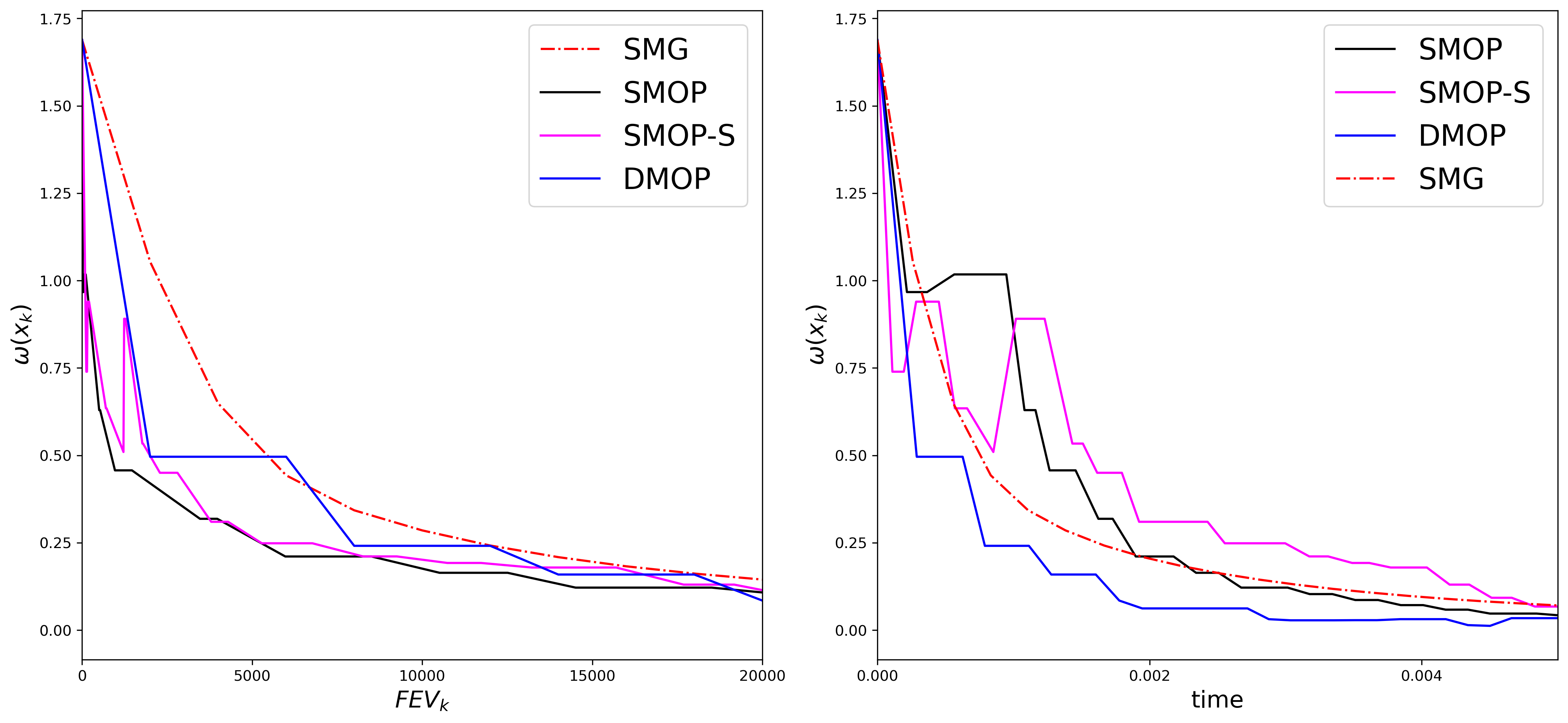} \\
\hspace{-12pt}
\includegraphics[width=0.85\textwidth,keepaspectratio]{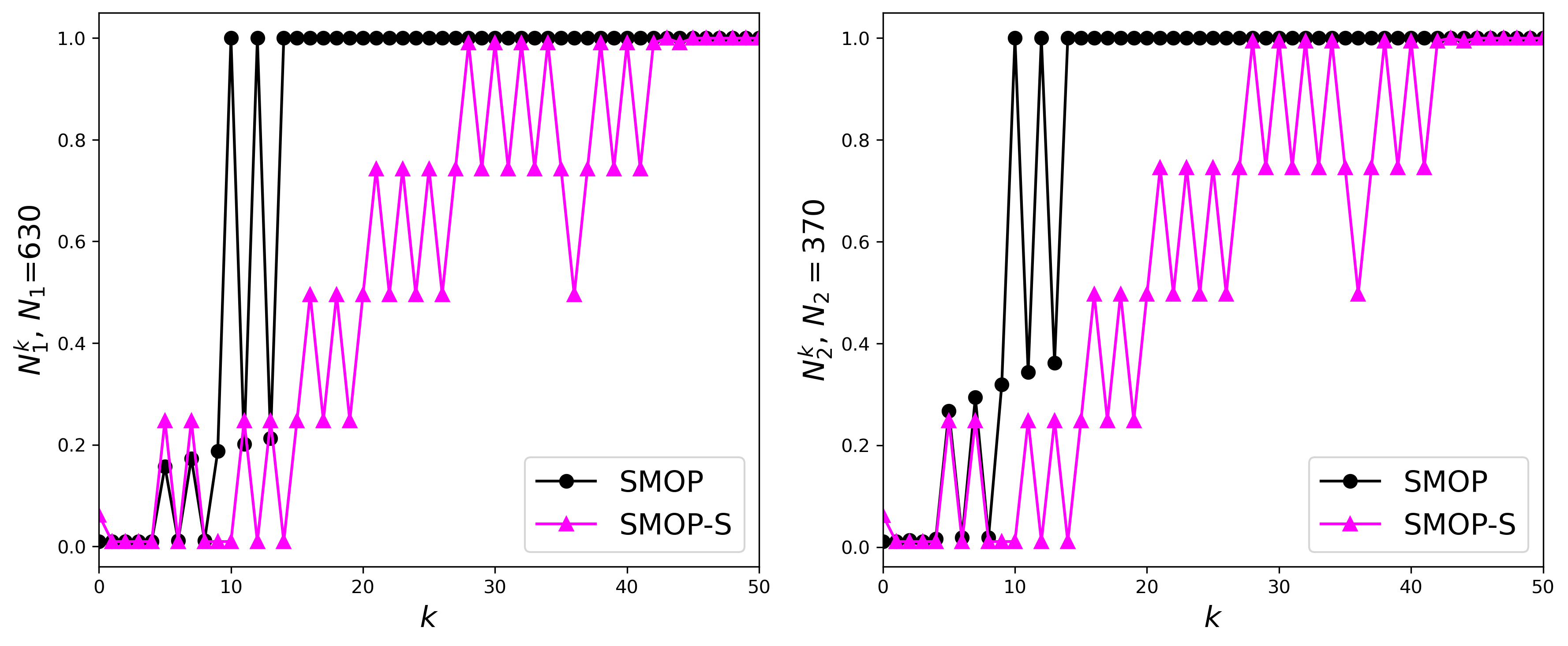} \\
\textbf{d)} \\

\caption{Performance comparison - $\omega(x_k)$ in terms of FEV/time, subsample sizes through iterations for datasets: c) svmguide, d) german.}
\label{fig1.2}
\end{figure}

\begin{figure}[htbp]
\centering

\includegraphics[width=0.82\textwidth,keepaspectratio]{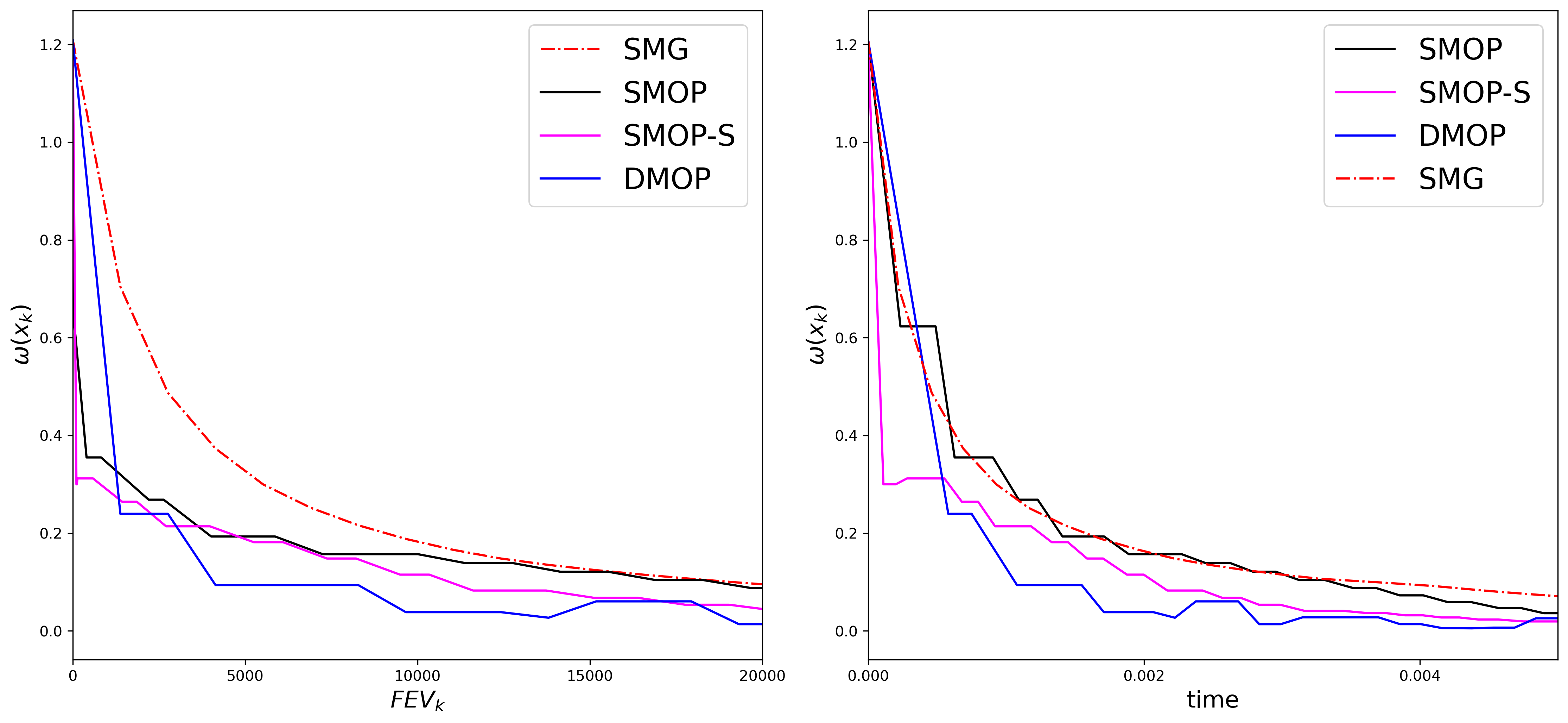} \\
\hspace{-12pt}
\includegraphics[width=0.85\textwidth,keepaspectratio]{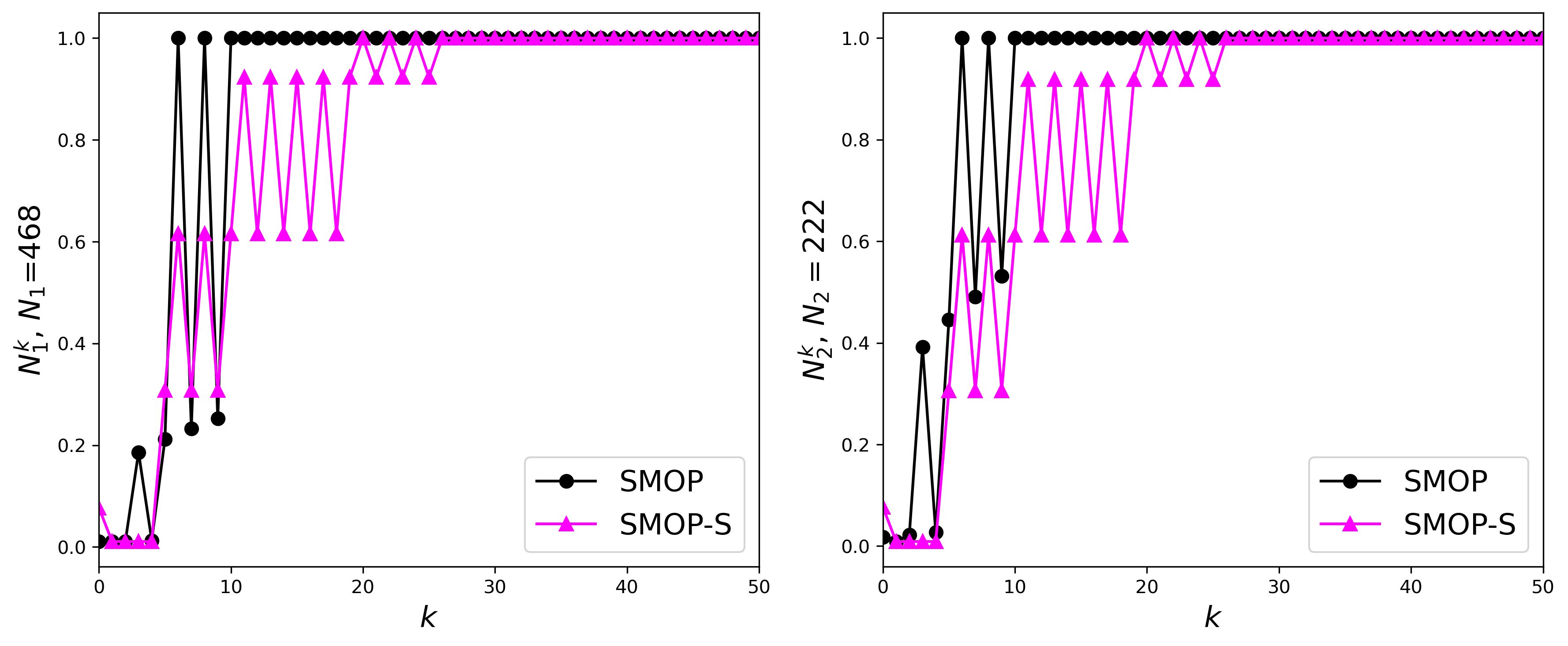} \\
\textbf{e)} \\[0.2cm]

\includegraphics[width=0.82\textwidth,keepaspectratio]{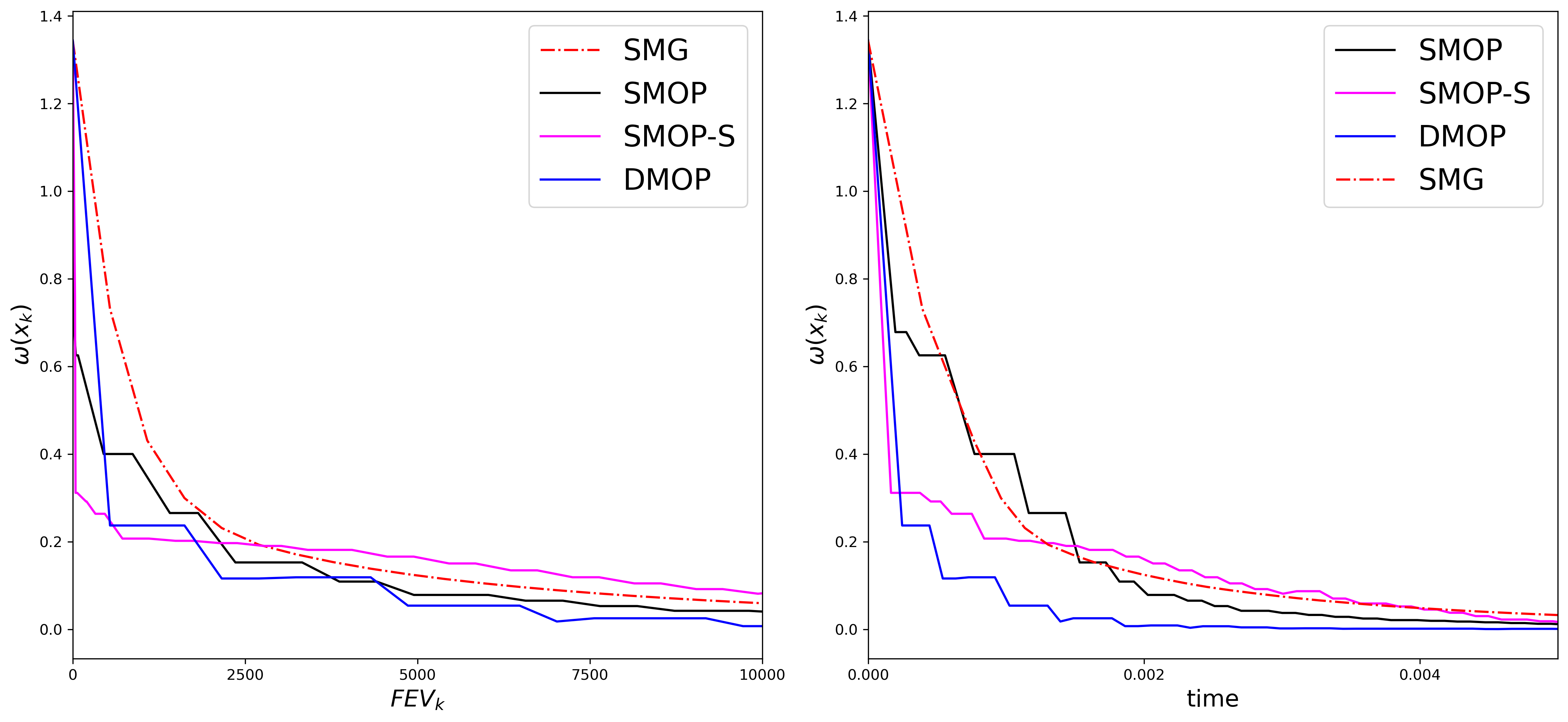} \\
\hspace{-12pt}
\includegraphics[width=0.85\textwidth,keepaspectratio]{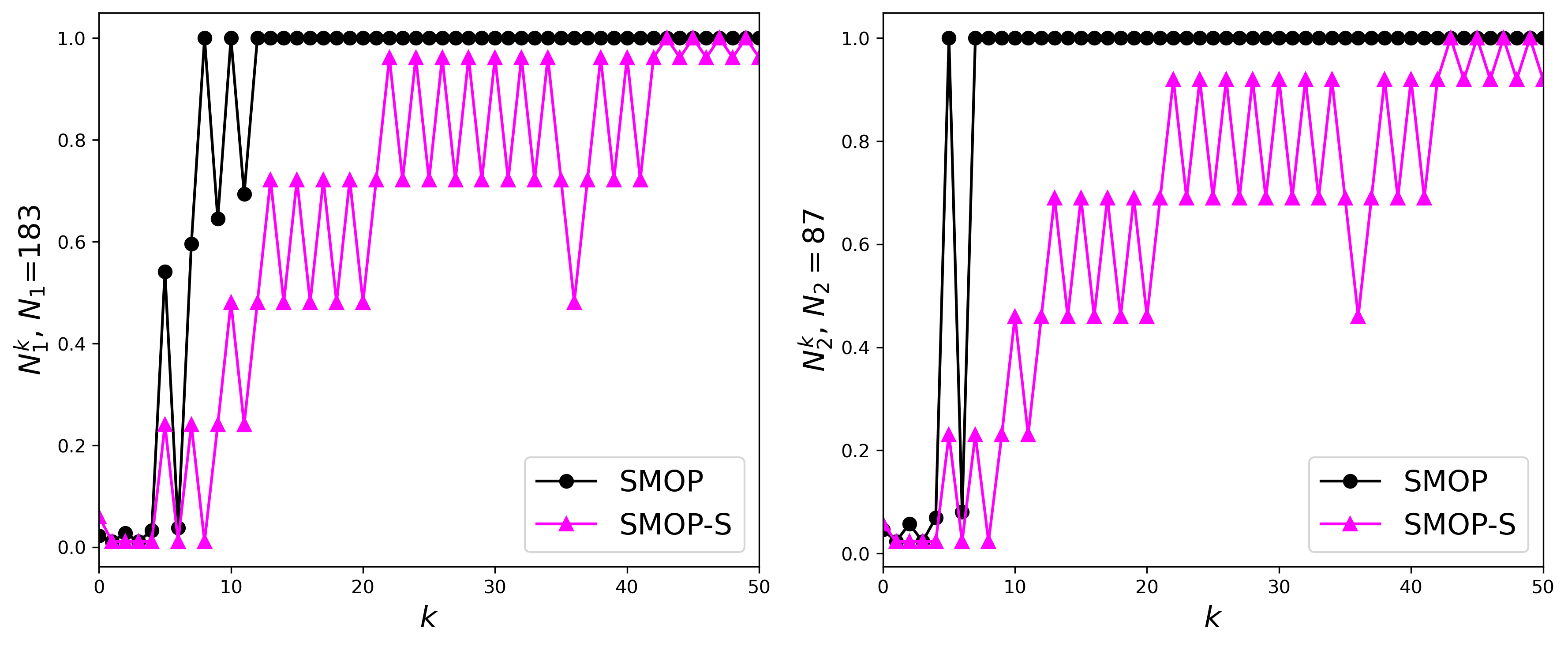} \\
\textbf{f)} \\

\caption{Performance comparison - $\omega(x_k)$ in terms of FEV/time, subsample sizes through iterations for datasets: e) australian, f) heart.}
\label{fig1.3}
\end{figure}

Further, a Pareto front finding technique from the previous section was employed, and the comparison was made between the SMOP-S and DMOP procedures. We provide a table with the results, showing Purity, $\Gamma$ and $\Delta$ spread, together with the sizes of the approximated fronts, and the time needed in seconds for finding the front. In order to show fair values, all presented results are the averages of 5 simulations, except for the covtype, which was calculated once due to high dimensions. For covtype the starting Pareto front was 3000 random points, instead of 30, and $n_r$ was changed to 100. For mnist, we do not provide the Pareto front analysis, as it was to large; however the Pareto critical points we get from running SMOP-S multiple times provide us fair models, which can be seen in Table \ref{tab1}. From Table \ref{tab2}, we can see that most of the time DMOP produces higher quality approximations of the Pareto front, having larger Purity, and smaller Spread metrics. However, SMOP-S approximated the front more efficiently in terms of time, with DMOP requiring 2, 1.1, 1.4, 1.2, and 0.9 times the execution time of SMOP-S across the respective datasets. For heart dataset, DMOP was faster, however for the other problems SMOP-S had time advantage. From tests, we have noticed that DMOP procedure makes fewer iterations which add more points per iteration.

\begin{table}[h!]
\begin{tabular}{|c|c|c|c|c|c|c|l|}
\hline
           & Algorithm                                             & Purity                                              & $\Gamma$                                              & $\Delta$                                            & $|\mathcal{L}|$                                     & time (s)                                                  & \#iter                                          \\ \hline
covtype    & \begin{tabular}[c]{@{}c@{}}SMOP-S\\ DMOP\end{tabular} & \begin{tabular}[c]{@{}c@{}}0.96\\ 0.99\end{tabular} & \begin{tabular}[c]{@{}c@{}}0.02\\ 0.001\end{tabular}  & \begin{tabular}[c]{@{}c@{}}1.88\\ 1.89\end{tabular} & \begin{tabular}[c]{@{}c@{}}2709\\ 1878\end{tabular} & \begin{tabular}[c]{@{}c@{}}31491.12\\ 62301.16\end{tabular} &                                            \begin{tabular}[c]{@{}c@{}}105\\66 \end{tabular}     \\ \hline
svmguide3  & \begin{tabular}[c]{@{}c@{}}SMOP-S\\ DMOP\end{tabular} & \begin{tabular}[c]{@{}c@{}}0.92\\ 0.98\end{tabular} & \begin{tabular}[c]{@{}c@{}}0.044\\ 0.024\end{tabular} & \begin{tabular}[c]{@{}c@{}}1.79\\ 1.78\end{tabular} & \begin{tabular}[c]{@{}c@{}}1728\\ 1878\end{tabular} & \begin{tabular}[c]{@{}c@{}}5.58\\ 6.12\end{tabular}       & \begin{tabular}[c]{@{}l@{}}26\\ 18\end{tabular} \\ \hline
german     & \begin{tabular}[c]{@{}c@{}}SMOP-S\\ DMOP\end{tabular} & \begin{tabular}[c]{@{}c@{}}0.98\\ 0.99\end{tabular} & \begin{tabular}[c]{@{}c@{}}0.006\\ 0.005\end{tabular} & \begin{tabular}[c]{@{}c@{}}1.61\\ 1.87\end{tabular} & \begin{tabular}[c]{@{}c@{}}2088\\ 1995\end{tabular} & \begin{tabular}[c]{@{}c@{}}4.68\\ 6.56\end{tabular}       & \begin{tabular}[c]{@{}l@{}}26\\ 17\end{tabular} \\ \hline
australian & \begin{tabular}[c]{@{}c@{}}SMOP-S\\ DMOP\end{tabular} & \begin{tabular}[c]{@{}c@{}}0.91\\ 0.97\end{tabular} & \begin{tabular}[c]{@{}c@{}}0.005\\ 0.004\end{tabular} & \begin{tabular}[c]{@{}c@{}}1.67\\ 1.81\end{tabular} & \begin{tabular}[c]{@{}c@{}}1831\\ 1788\end{tabular} & \begin{tabular}[c]{@{}c@{}}5.24\\ 6.15\end{tabular}       & \begin{tabular}[c]{@{}l@{}}27\\ 23\end{tabular} \\ \hline
heart      & \begin{tabular}[c]{@{}c@{}}SMOP-S\\ DMOP\end{tabular} & \begin{tabular}[c]{@{}c@{}}0.93\\ 0.92\end{tabular} & \begin{tabular}[c]{@{}c@{}}0.009\\ 0.017\end{tabular} & \begin{tabular}[c]{@{}c@{}}1.80\\ 1.87\end{tabular} & \begin{tabular}[c]{@{}c@{}}1947\\ 2145\end{tabular} & \begin{tabular}[c]{@{}c@{}}2.11\\ 1.92\end{tabular}       & \begin{tabular}[c]{@{}l@{}}35\\ 15\end{tabular} \\ \hline
\end{tabular}
\caption{Comparison metrics between resulting Pareto front approximations from SMOP and DMOP on different datasets}
\label{tab2}
\end{table}
Additionally, we show Pareto front approximations for both SMOP-S and DMOP in Figure \ref{fig2.1}. The tests showed the SMOP-S iterations are faster in the beginning, and  slow down as the points go closer to the Pareto front, since the trust region radius approaches 0 in later iterations. For homogeneous problems, these fast iterations can be explained as generating points closer to the front in early iterations, whereas for heterogeneous problems, these fast approximations act as an additional randomization
factor which perturbs the points, and not necessarily improving the front.
\begin{figure}[htbp]
\centering
\begin{tabular}{cc}
\includegraphics[width=0.39\textwidth,keepaspectratio]{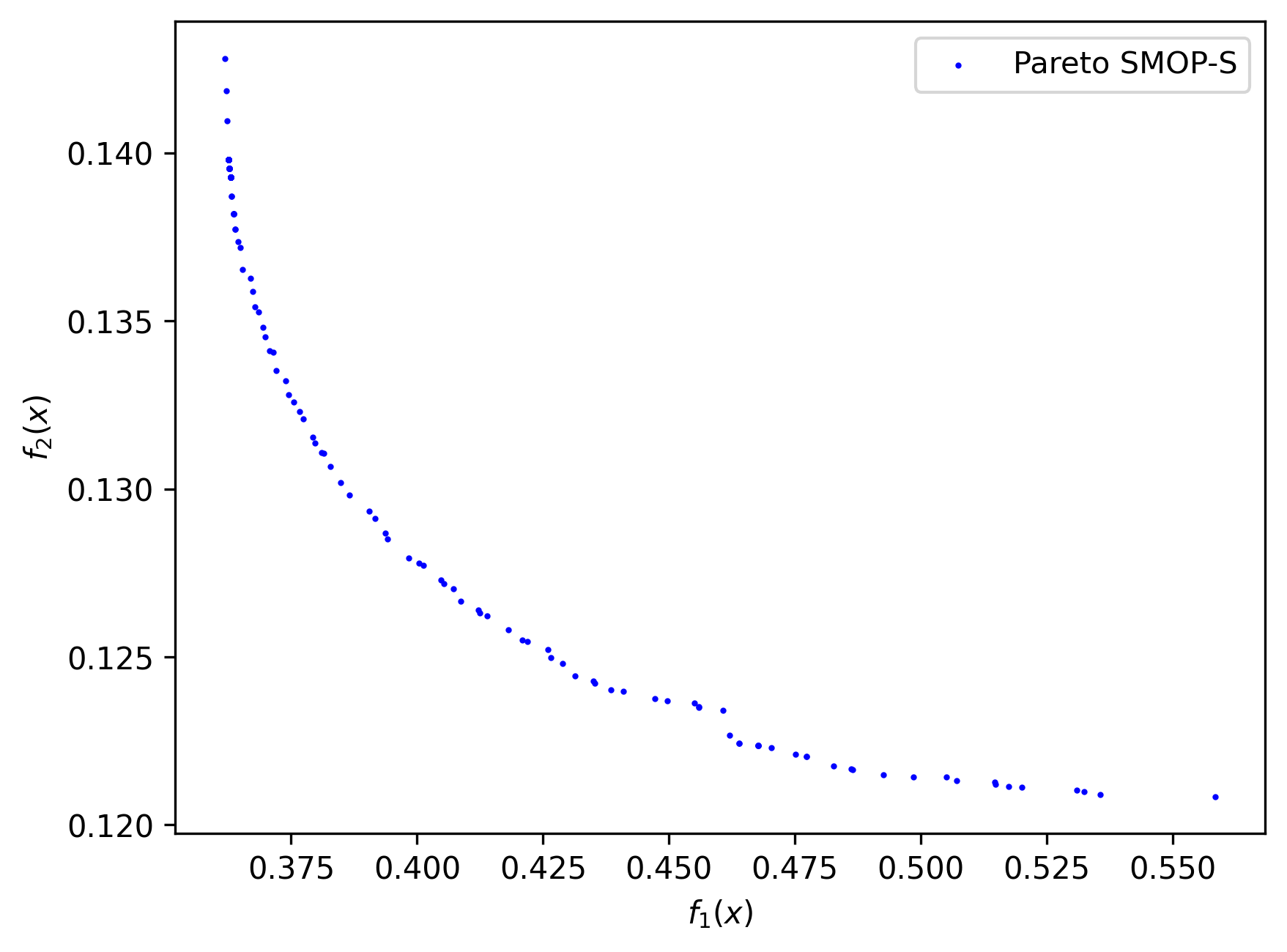} &
\includegraphics[width=0.39\textwidth,keepaspectratio]{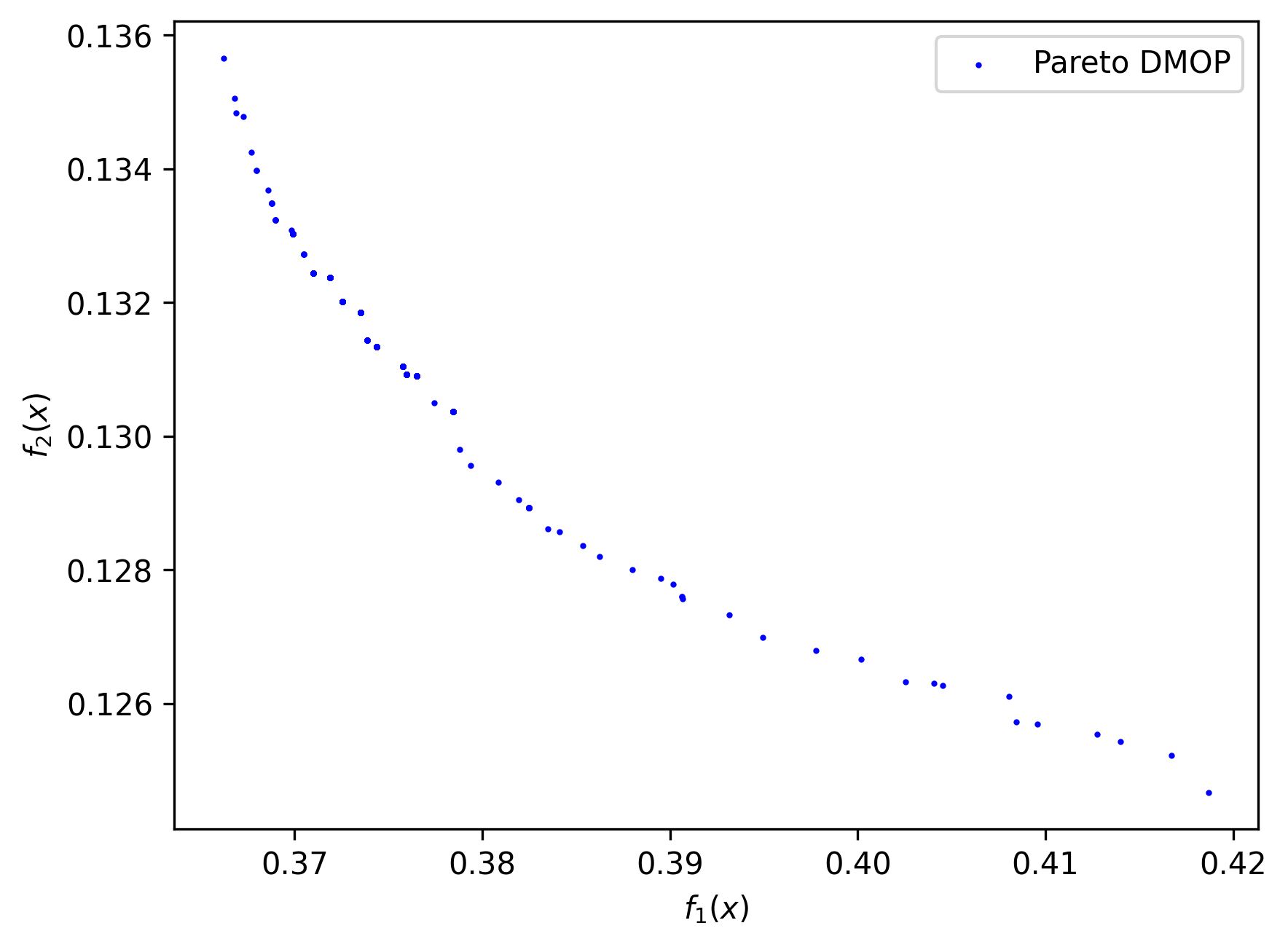} \\
\multicolumn{2}{c}{\textbf{a)}} \\
\includegraphics[width=0.39\textwidth,keepaspectratio]{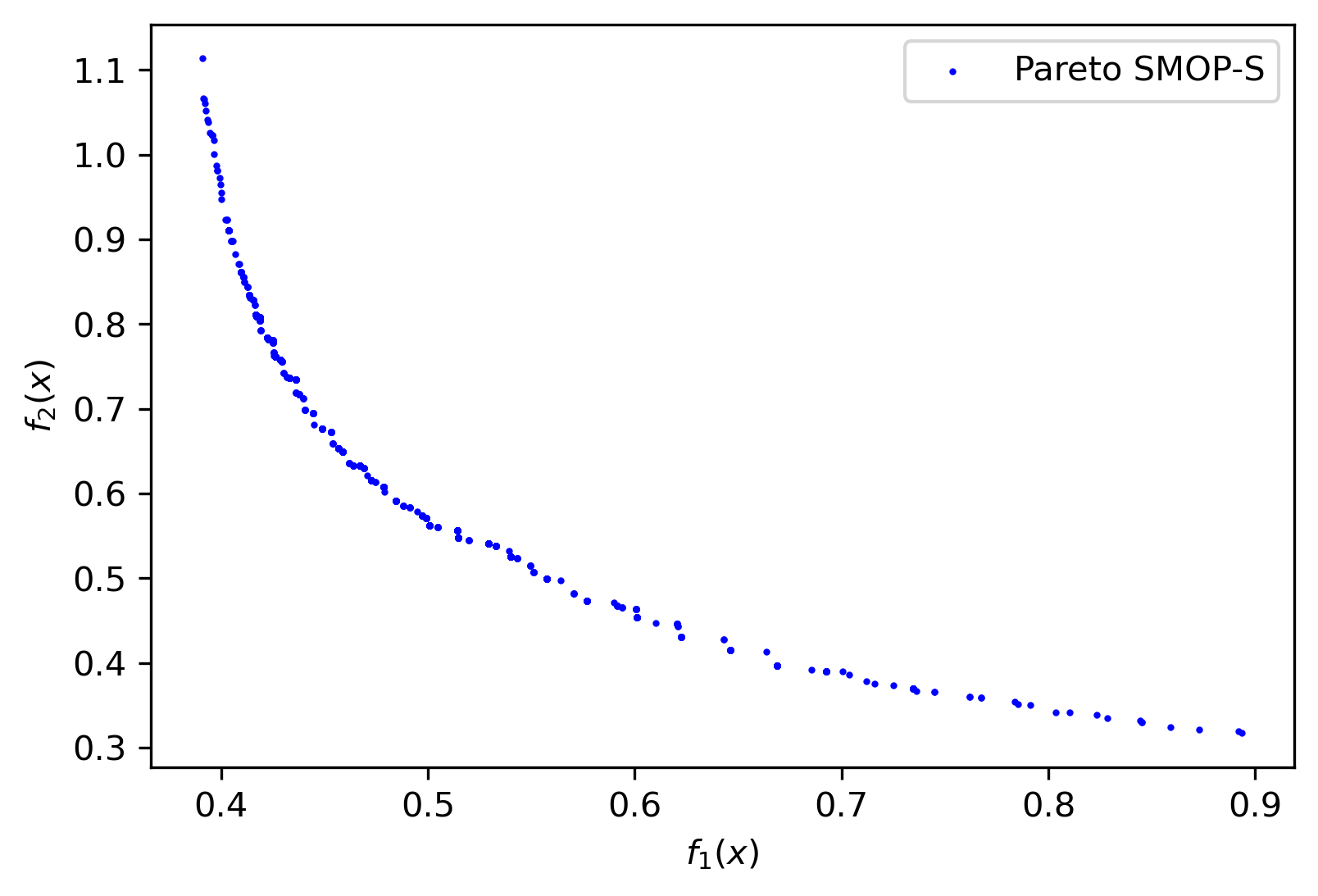} &
\includegraphics[width=0.39\textwidth,keepaspectratio]{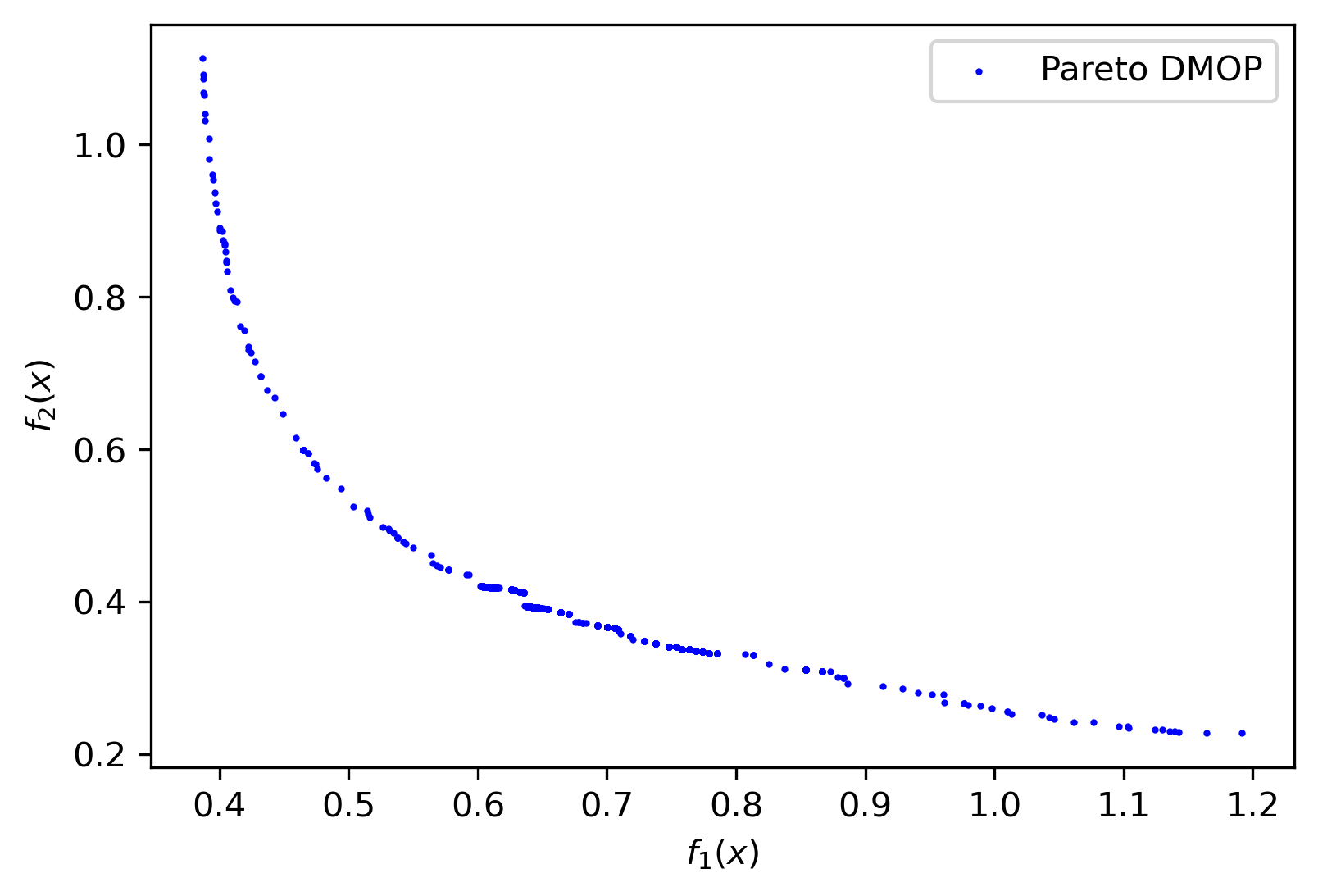} \\
\multicolumn{2}{c}{\textbf{b)}} \\[0.1cm]
\includegraphics[width=0.39\textwidth,keepaspectratio]{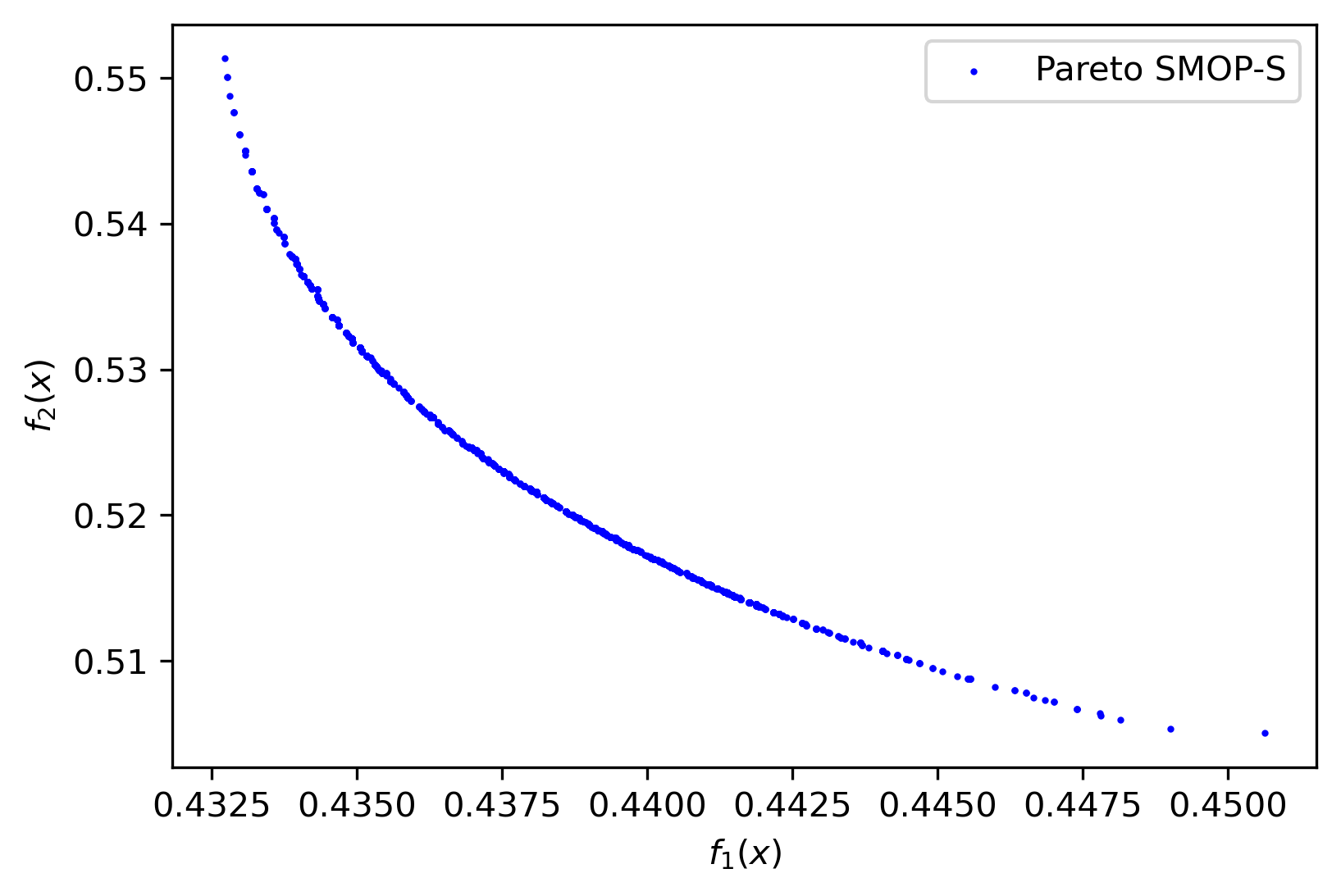} &
\includegraphics[width=0.39\textwidth,keepaspectratio]{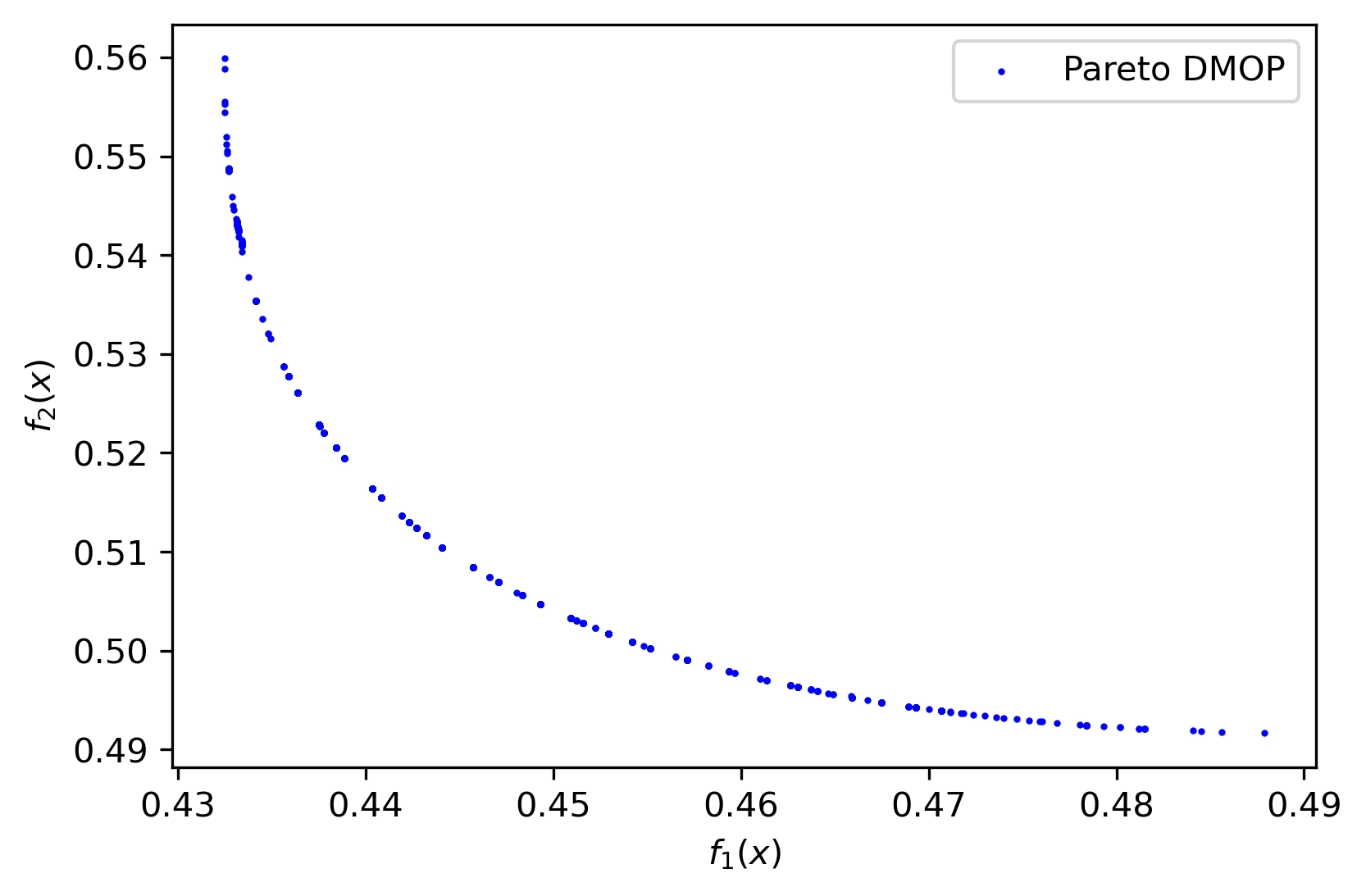} \\
\multicolumn{2}{c}{\textbf{c)}} \\[0.1cm]
\includegraphics[width=0.39\textwidth,keepaspectratio]{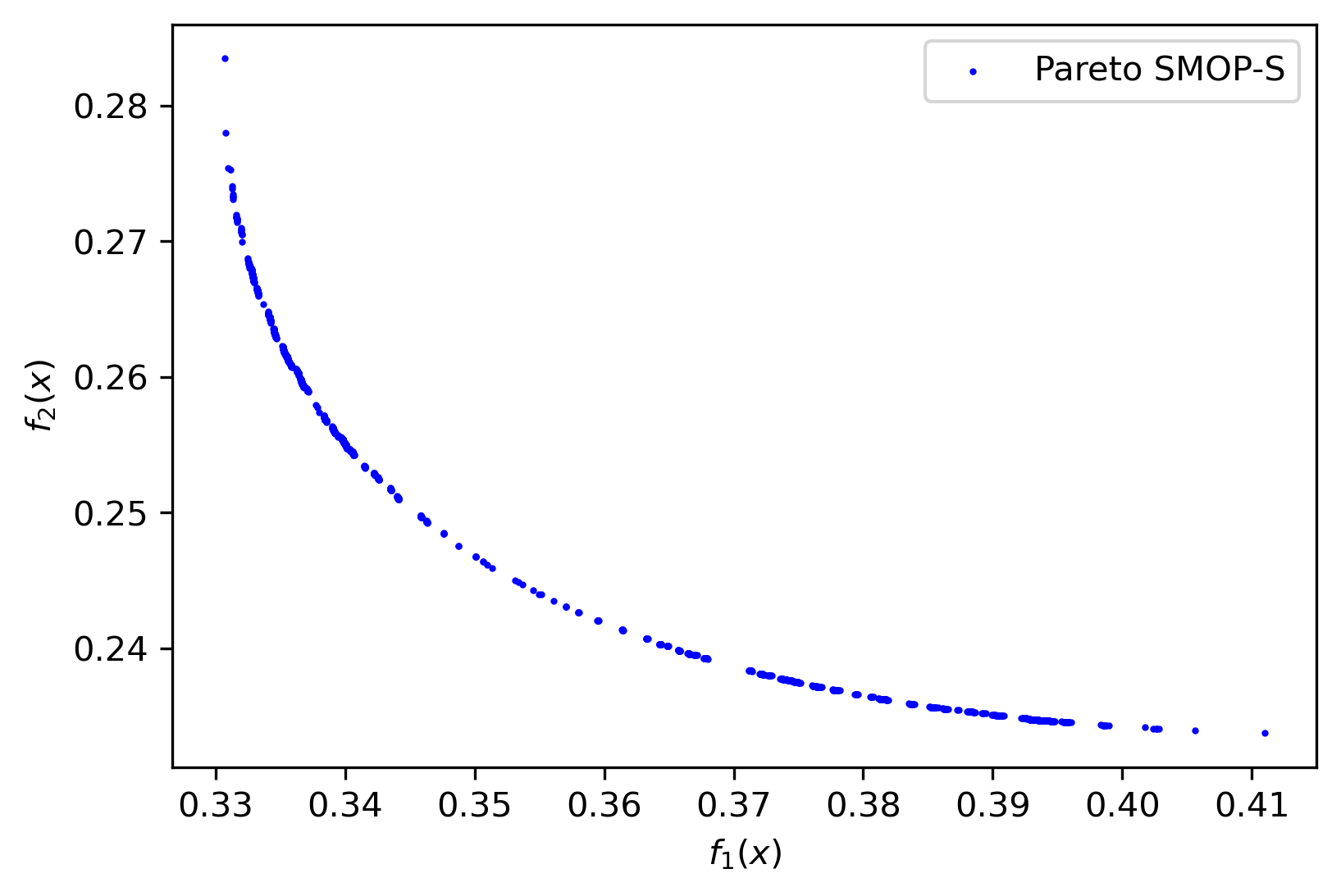} &
\includegraphics[width=0.39\textwidth,keepaspectratio]{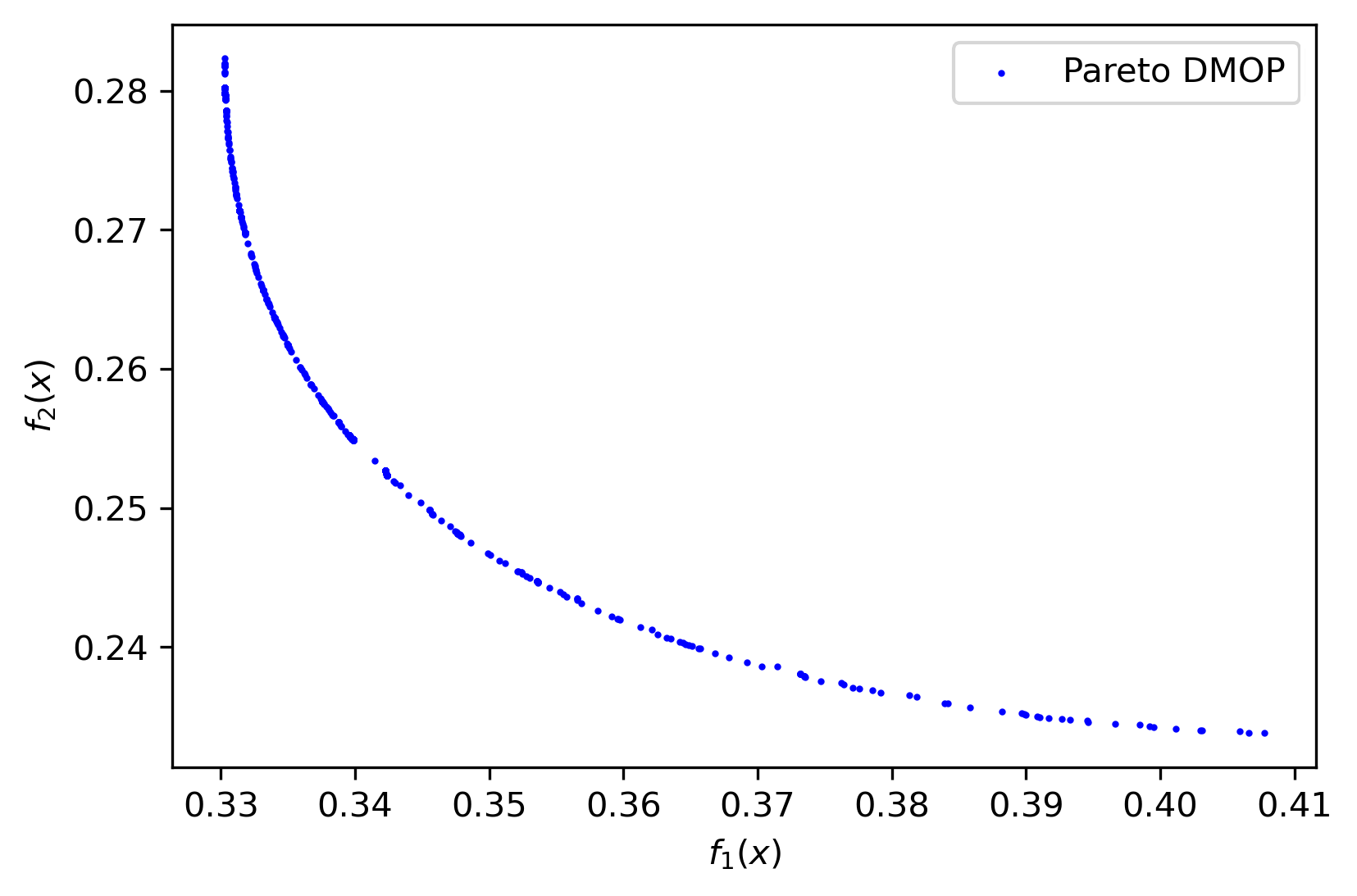} \\
\multicolumn{2}{c}{\textbf{d)}} \\[0.1cm]
\includegraphics[width=0.39\textwidth,keepaspectratio]{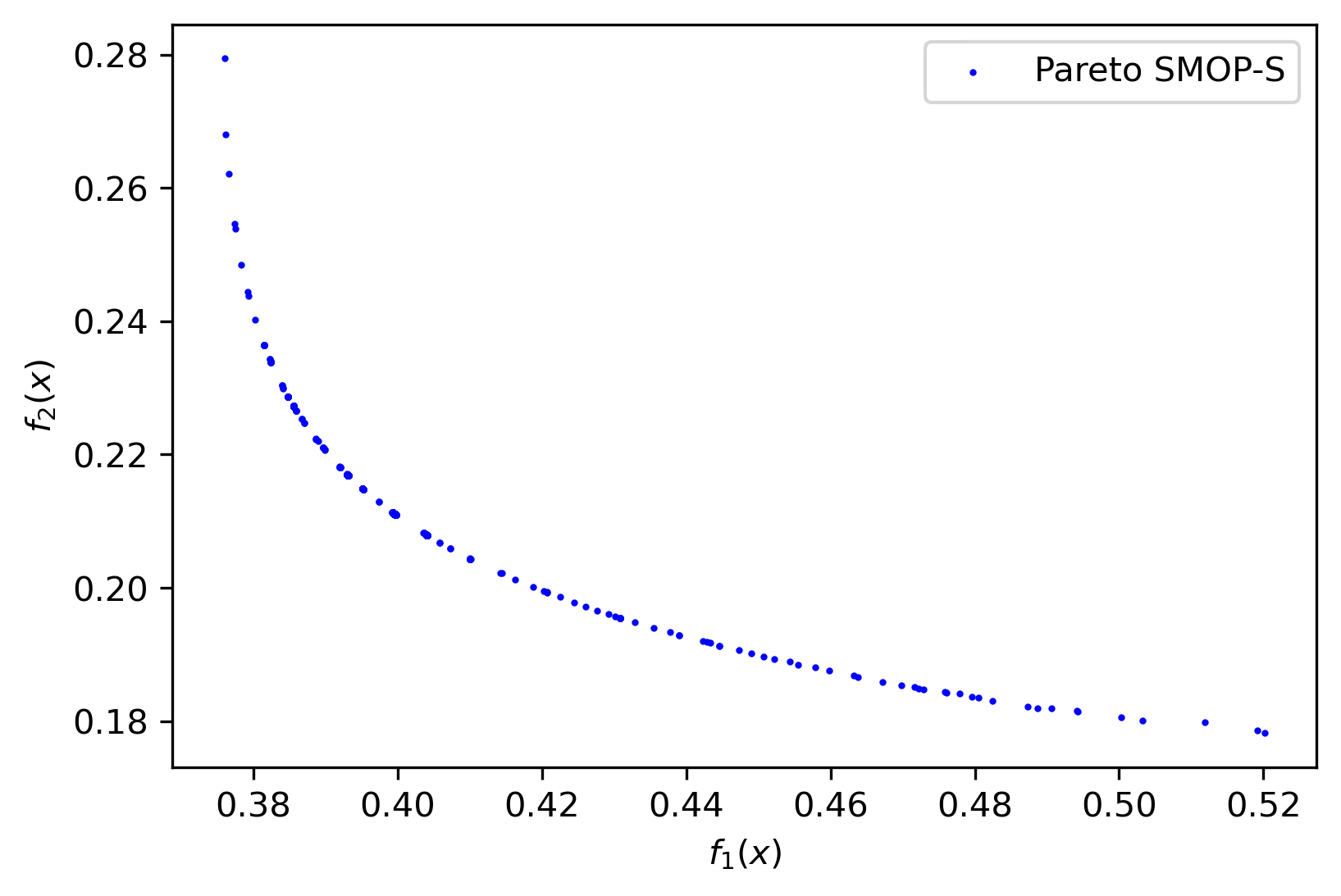} &
\includegraphics[width=0.39\textwidth,keepaspectratio]{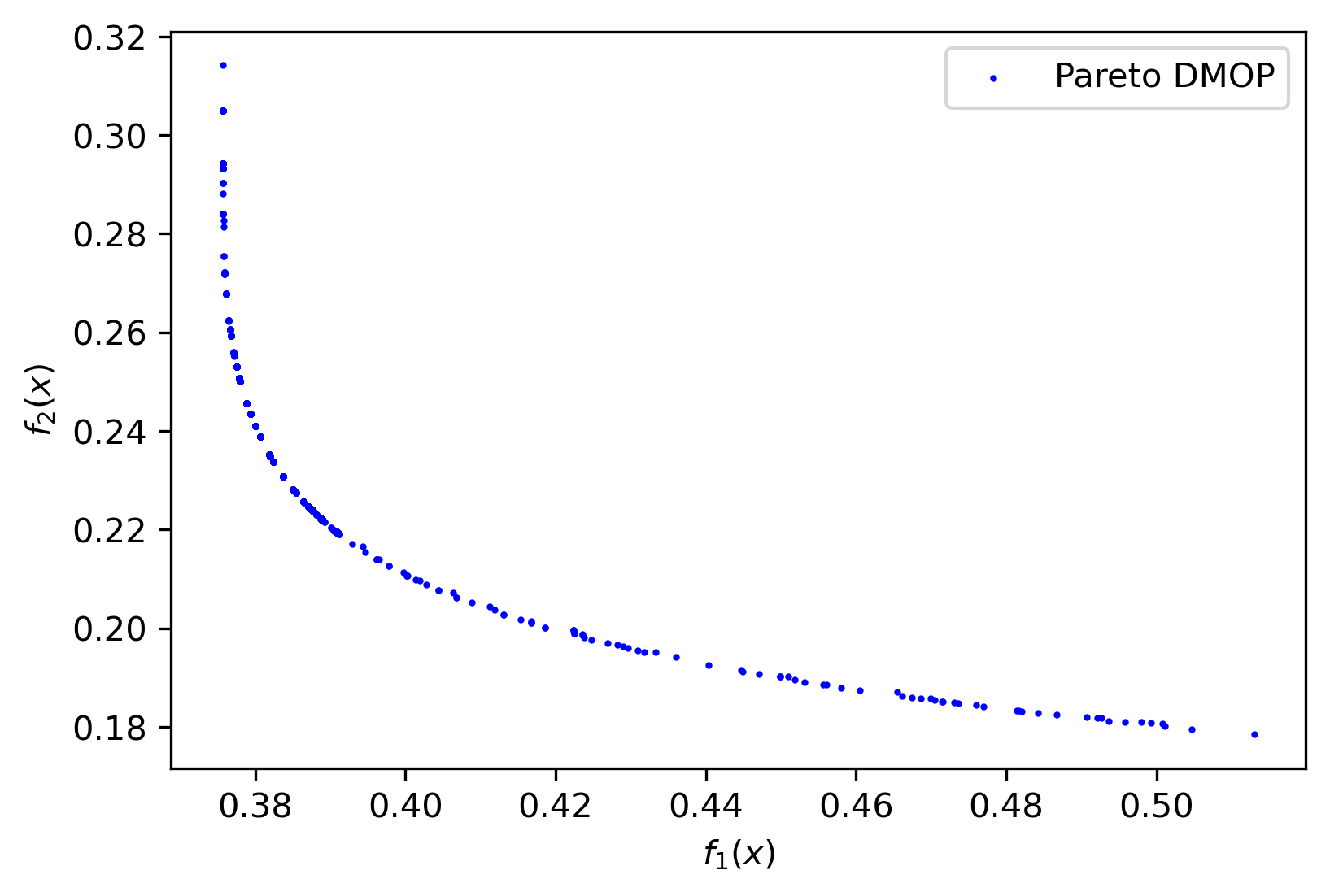} \\
\multicolumn{2}{c}{\textbf{e)}} \\

\end{tabular}
\caption{Pareto front approximations using SMOP (left) and DMOP (right): a) covtype b) svmguide c) german d) australian, e) heart, }
\label{fig2.1}
\end{figure}

\subsection{Test problems with added noise}
The second set of experiments consists of unconstrained multi objective problems gathered from \cite{CMV} and \cite{NUM}. The presented problems are low dimension and they differ in terms of the shape of the Pareto front. Namely, four shapes are presented: convex, concave, mixed (neither convex nor concave) and disconnected. The problem we are solving is the following:

\begin{equation}
    \min_{x}f(x)=\min_{x}(f_1(x),f_2(x))^T
    \label{tp1}
\end{equation}
In order to simulate stochasticity we choose to perturb points at each iteration, as done in \cite{LV},\cite{TEST}, which turns the problem into
\begin{equation}
    \min_{x}F(x)=\min_x(E[f_1(x+\omega)],E[ f_2(x+\omega)])^T
    \label{testprob}
\end{equation}
where $\omega$ is a random uniform vector with mean zero. Using Stochastic Average Approximation, \eqref{testprob} is transformed into the finite sum problem.

\begin{equation}
     \min_{x}F(x)\approx\min_x(\frac{1}{N}\sum_{i=1}^{N}f_1(x+\omega_i),\frac{1}{N}\sum_{i=1}^{N}f_2(x+\omega_i))^T
    \label{saa}
\end{equation}
where $\omega_i$ are i.i.d. samples of the random variable $\omega$. Similarly as in the previous experiment, we test SMOP-S against DMOP Pareto front finding procedure on the problem \eqref{saa}. The SMOP-S configuration follows the one explained in previous sections. For each problem, we generate $N=500$ i.i.d. uniformly sampled vectors with mean zero and length of the interval equal to $0.1$.  The following Table \ref{tab3} showcases the basic information for each considered problem. Exponential functions and its derivatives in FF1 and T2 are considered harder to calculate \cite{NUM}, hence the evaluation of these functions was a time consuming factor, similarly as in the logistic regression.
\begin{table}[h!]
\centering
\begin{tabular}{|l|l|l|l|}
\hline
    & n & Functions       & Shape        \\ \hline
SP1 & 2 &  \begin{tabular}[l]{@{}l@{}}{\small$f_1(x)=(x_1-1)^2+(x_1-x_2)^2$}\\{\small$f_2(x)=(x_2-3)^2+(x_1-x_2)^2$}\end{tabular}     & convex       \\ \hline
SK1 & 1 &  \begin{tabular}[l]{@{}l@{}}{\small$f_1(x)=x^4+3x^3-10x^2-10x-10$}\\{\small$f_2(x)=0.5x^4-2x^3-10x^2+10x-5$}\end{tabular}         & disconnected \\ \hline
FF1 & 2 &  \begin{tabular}[l]{@{}l@{}}{\small$f_1(x)=1-e^{-(x_1-1)^2-(x_2+1)^2}$}\\ {\small$f_2(x)=1-e^{-(x_1+1)^2-(x_2-1)^2}$}\end{tabular}                                     & concave      \\ \hline
T2  & 2 &  \begin{tabular}[l]{@{}l@{}}{\small$f_1(x)=\sin(x_2)$}\\{\small$f_2(x)=1-e^{-(x_1-\frac{1}{\sqrt{2}})^2-(x_2-\frac{1}{\sqrt{2}})^2}$} \end{tabular}                   & mixed        \\ \hline
\end{tabular}
\caption{Problem description for \eqref{saa}}
\label{tab3}
\end{table}

As in the previous experiments, the results show that the DMOP generates  fronts with better purity. Nonetheless, SMOP-S demonstrated the ability to find the front more quickly with high enough quality. The average metrics of 5 simulations can be seen in Table \eqref{tab4}.
\begin{table}[h!]
\begin{tabular}{|c|c|c|c|c|c|c|c|}
\hline
    & Algorithm                                             & Purity                                              & $\Gamma$                                              & $\Delta$                                            & $|\mathcal{L}_k|$                                   & time (s)                                            & \#iter                                        \\ \hline
SP1 & \begin{tabular}[c]{@{}c@{}}SMOP-S\\ DMOP\end{tabular} & \begin{tabular}[c]{@{}c@{}}0.95\\ 0.96\end{tabular} & \begin{tabular}[c]{@{}c@{}}0.17\\ 0.16\end{tabular}   & \begin{tabular}[c]{@{}c@{}}1.86\\ 1.84\end{tabular} & \begin{tabular}[c]{@{}c@{}}2475\\ 2377\end{tabular} & \begin{tabular}[c]{@{}c@{}}38.4\\ 81.1\end{tabular} & \begin{tabular}[c]{@{}c@{}}6\\ 5\end{tabular} \\ \hline
SK1 & \begin{tabular}[c]{@{}c@{}}SMOP-S\\ DMOP\end{tabular} & \begin{tabular}[c]{@{}c@{}}0.99\\ 1.00\end{tabular} & \begin{tabular}[c]{@{}c@{}}24.68\\ 29.25\end{tabular} & \begin{tabular}[c]{@{}c@{}}1.86\\ 1.76\end{tabular} & \begin{tabular}[c]{@{}c@{}}2434\\ 1682\end{tabular} & \begin{tabular}[c]{@{}c@{}}184\\ 253\end{tabular}   & \begin{tabular}[c]{@{}c@{}}4\\ 4\end{tabular} \\ \hline
FF1 & \begin{tabular}[c]{@{}c@{}}SMOP-S\\ DMOP\end{tabular} & \begin{tabular}[c]{@{}c@{}}0.94\\ 0.95\end{tabular} & \begin{tabular}[c]{@{}c@{}}0.9\\ 0.07\end{tabular}   & \begin{tabular}[c]{@{}c@{}}1.82\\ 1.80\end{tabular} & \begin{tabular}[c]{@{}c@{}}2208\\ 1958\end{tabular} & \begin{tabular}[c]{@{}c@{}}42.4\\ 97.0\end{tabular} & \begin{tabular}[c]{@{}c@{}}6\\ 5\end{tabular} \\ \hline
T2  & \begin{tabular}[c]{@{}c@{}}SMOP-S\\ DMOP\end{tabular} & \begin{tabular}[c]{@{}c@{}}0.91\\ 0.94\end{tabular} & \begin{tabular}[c]{@{}c@{}}0.05\\ 0.05\end{tabular}   & \begin{tabular}[c]{@{}c@{}}1.84\\ 1.86\end{tabular} & \begin{tabular}[c]{@{}c@{}}2953\\ 3303\end{tabular} & \begin{tabular}[c]{@{}c@{}}84.1\\ 208.1\end{tabular}    & \begin{tabular}[c]{@{}c@{}}5\\ 6\end{tabular} \\ \hline
\end{tabular}

\caption{Performance profile for problems from Table \ref{tab3}}
\label{tab4}
\end{table}

The Purity of the fronts generated by SMOP-S is lower than DMOP, which implicates that SMOP-S generated some points which were dominated by the those generated by DMOP. The Spread confirms that the gaps in SMOP-S fronts are larger in most problems. For problem SK1 the $\Gamma$ spread is large since the front is disconnected, and it takes the disconnected part into the calculation. The average time needed to calculate the front showed that SMOP-S procedure was 2.1, 1.3, 2.2, and 2.5 times faster than DMOP respectively. The following Figure \ref{fig4} illustrates the approximate Pareto front for both procedures. 
\begin{figure}[]
\centering
\begin{tabular}{cc}

\includegraphics[width=0.45\textwidth,keepaspectratio]{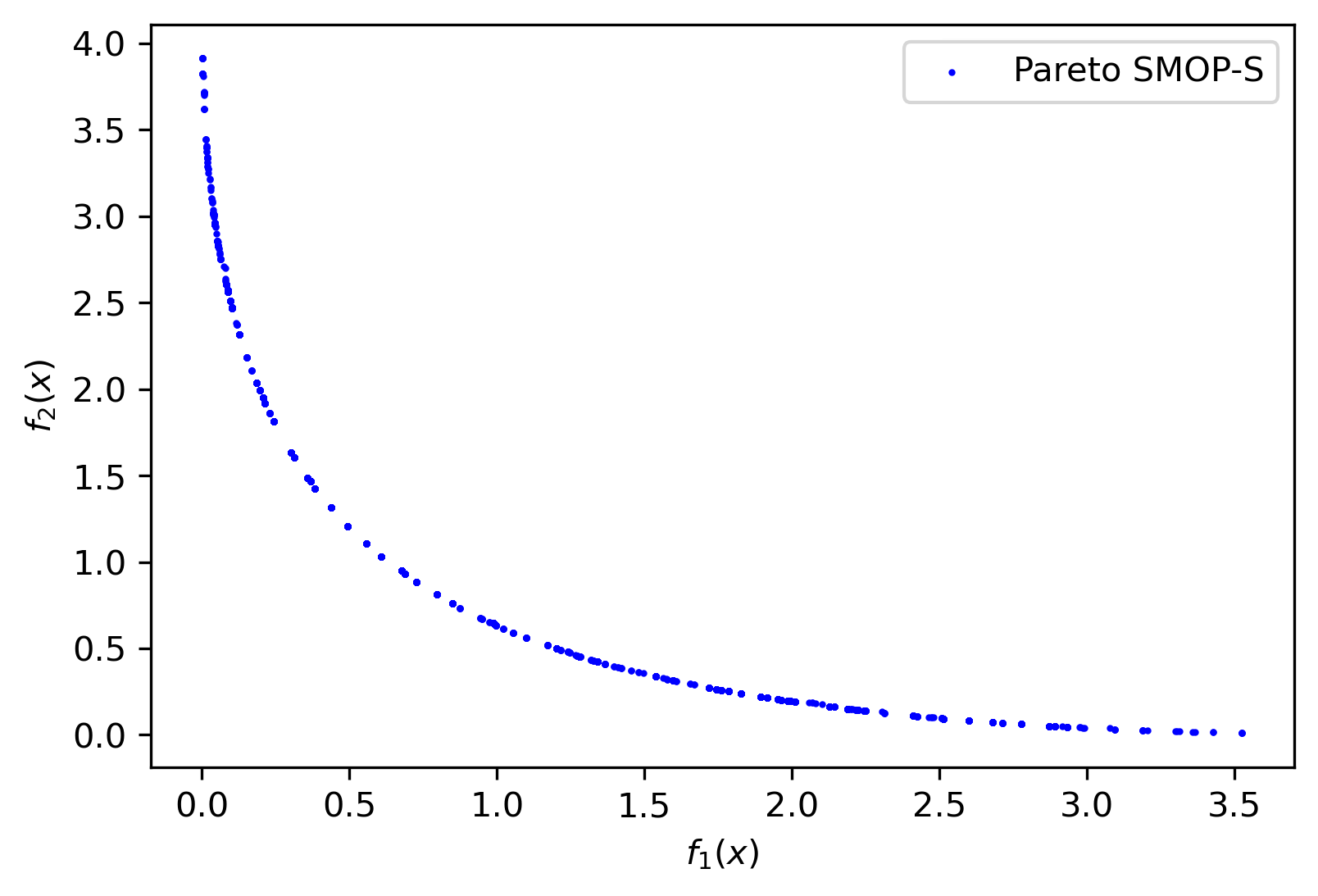} &
\includegraphics[width=0.45\textwidth,keepaspectratio]{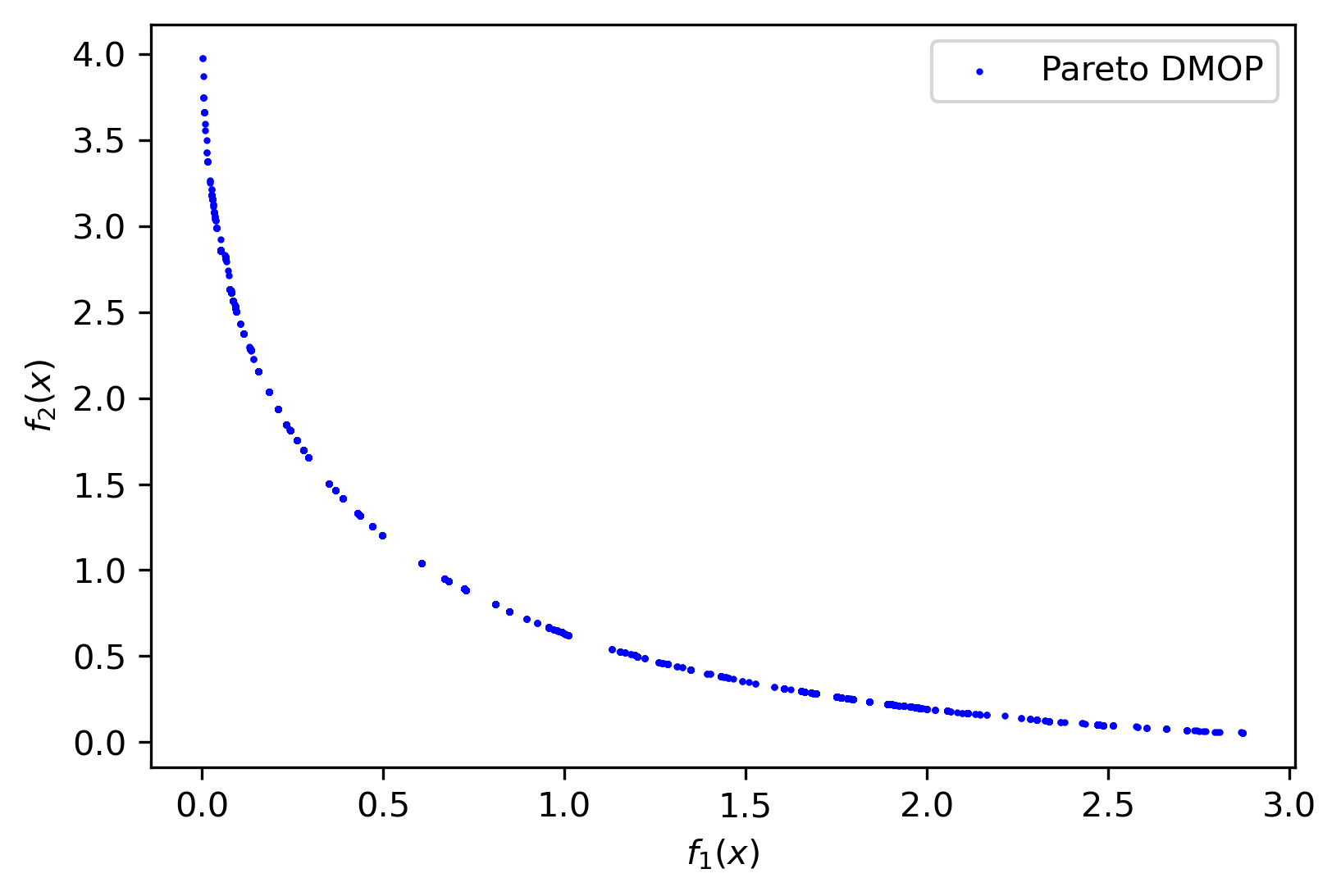} \\
\multicolumn{2}{c}{a)} \\[0.4cm]

\includegraphics[width=0.45\textwidth,keepaspectratio]{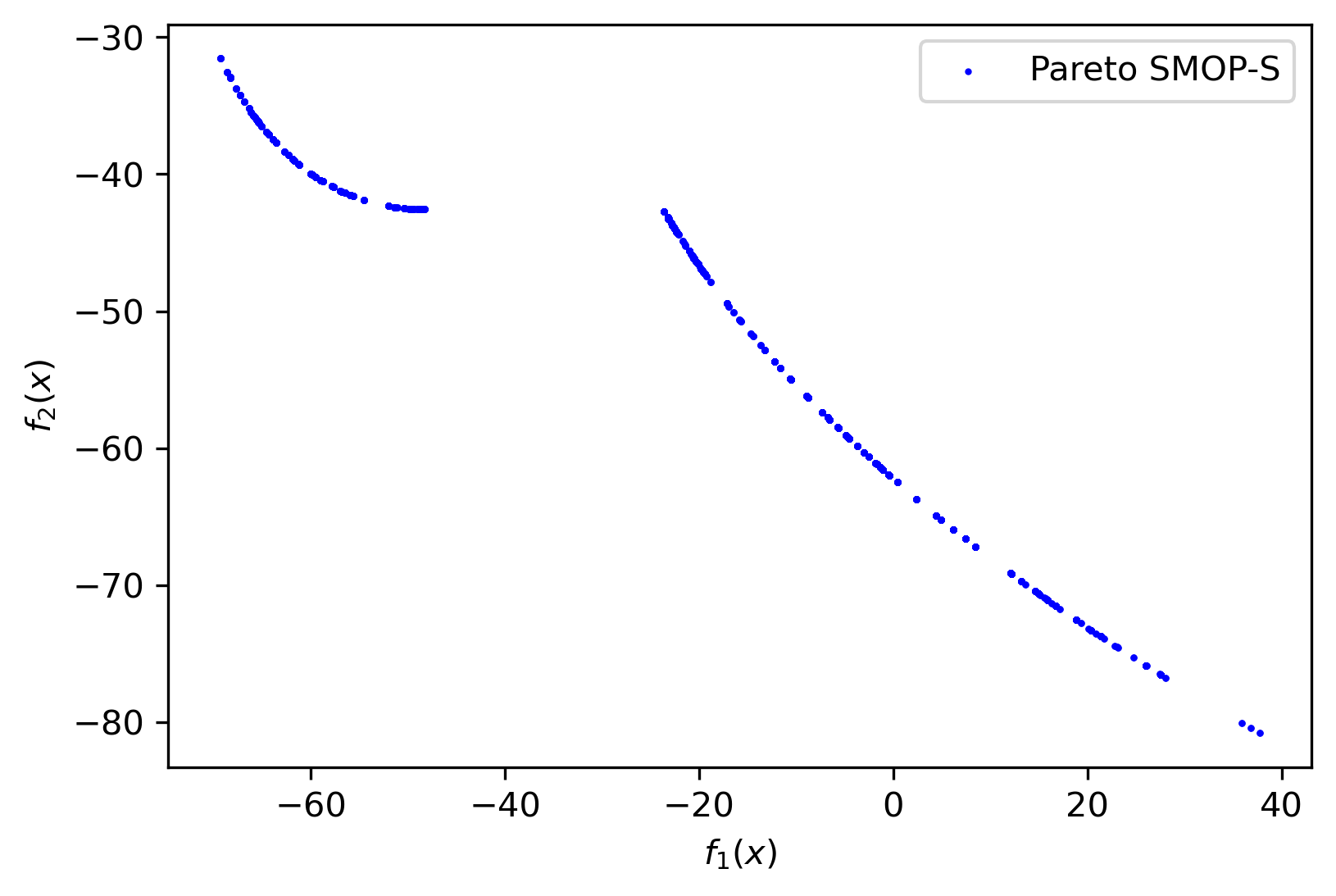} &
\includegraphics[width=0.45\textwidth,keepaspectratio]{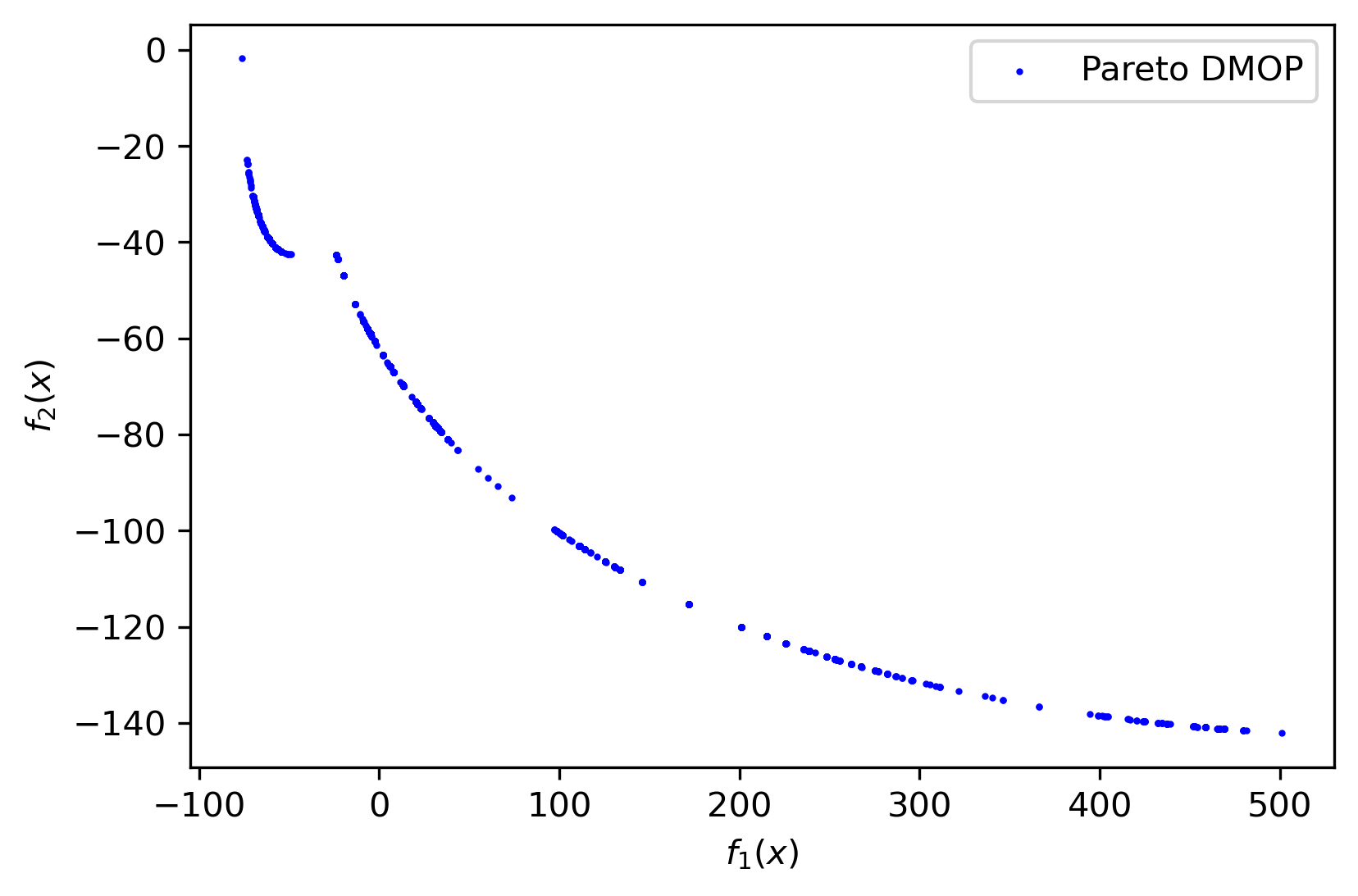} \\
\multicolumn{2}{c}{b)} \\[0.4cm]

\includegraphics[width=0.45\textwidth,keepaspectratio]{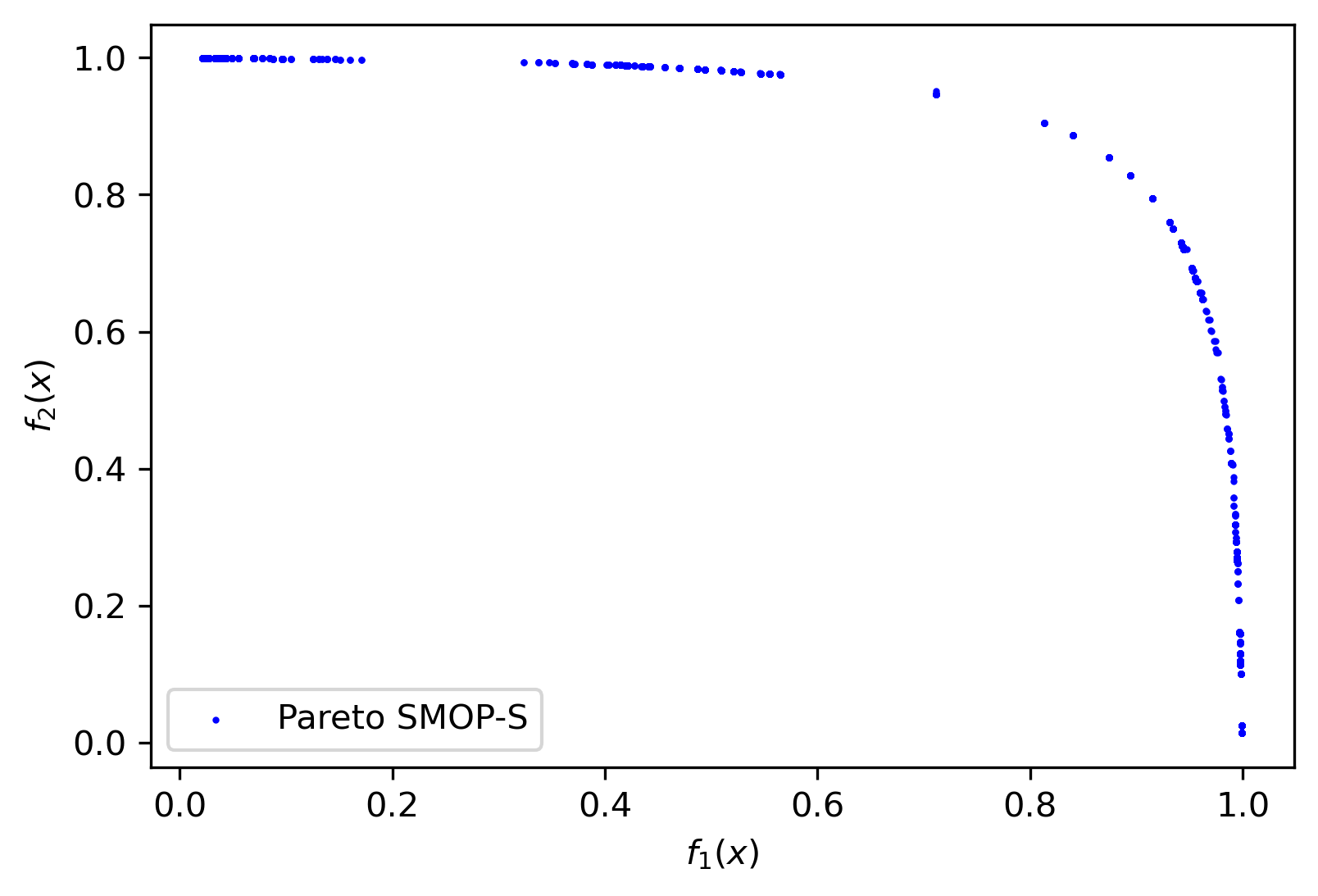} &
\includegraphics[width=0.45\textwidth,keepaspectratio]{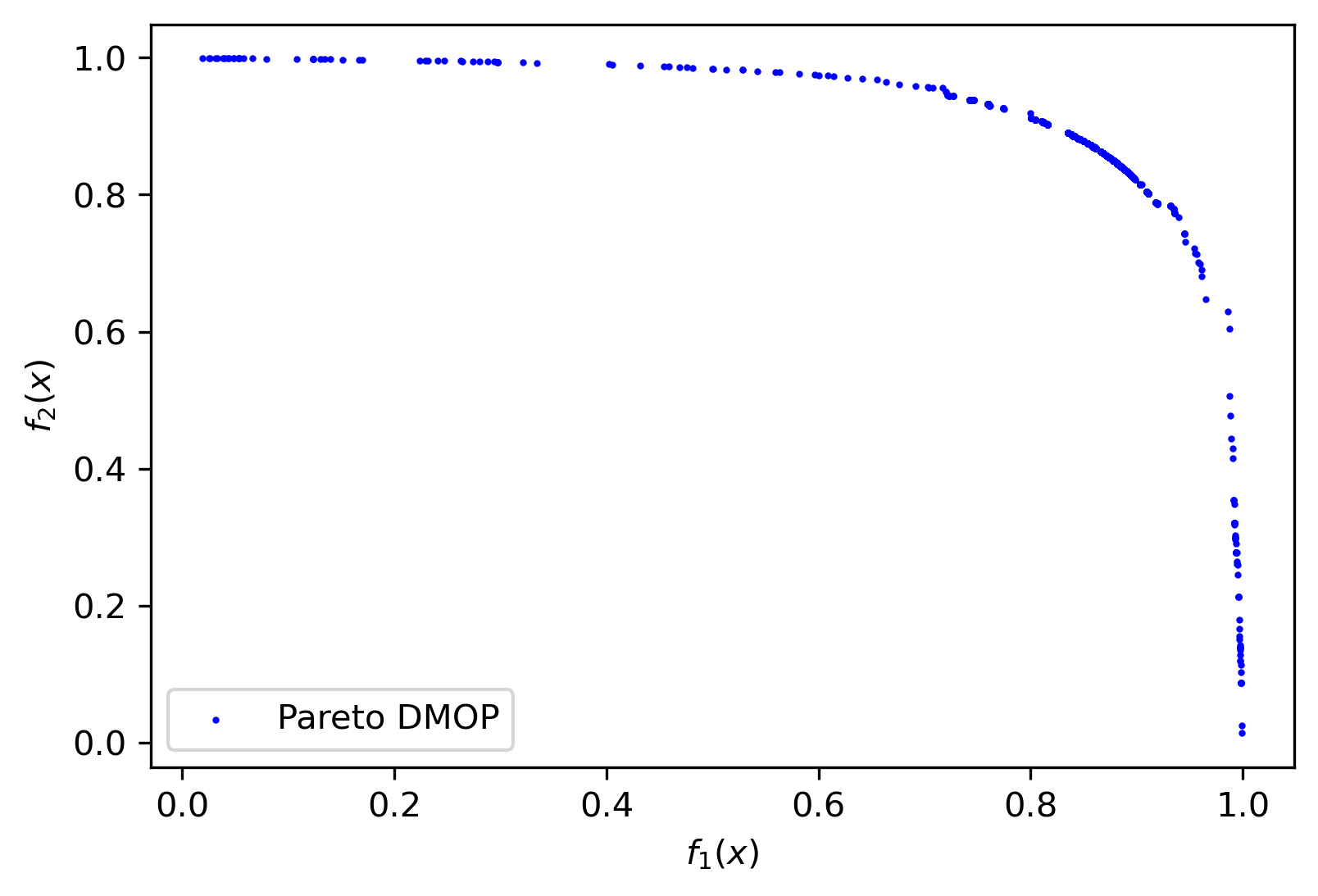} \\
\multicolumn{2}{c}{c)} \\[0.4cm]

\includegraphics[width=0.45\textwidth,keepaspectratio]{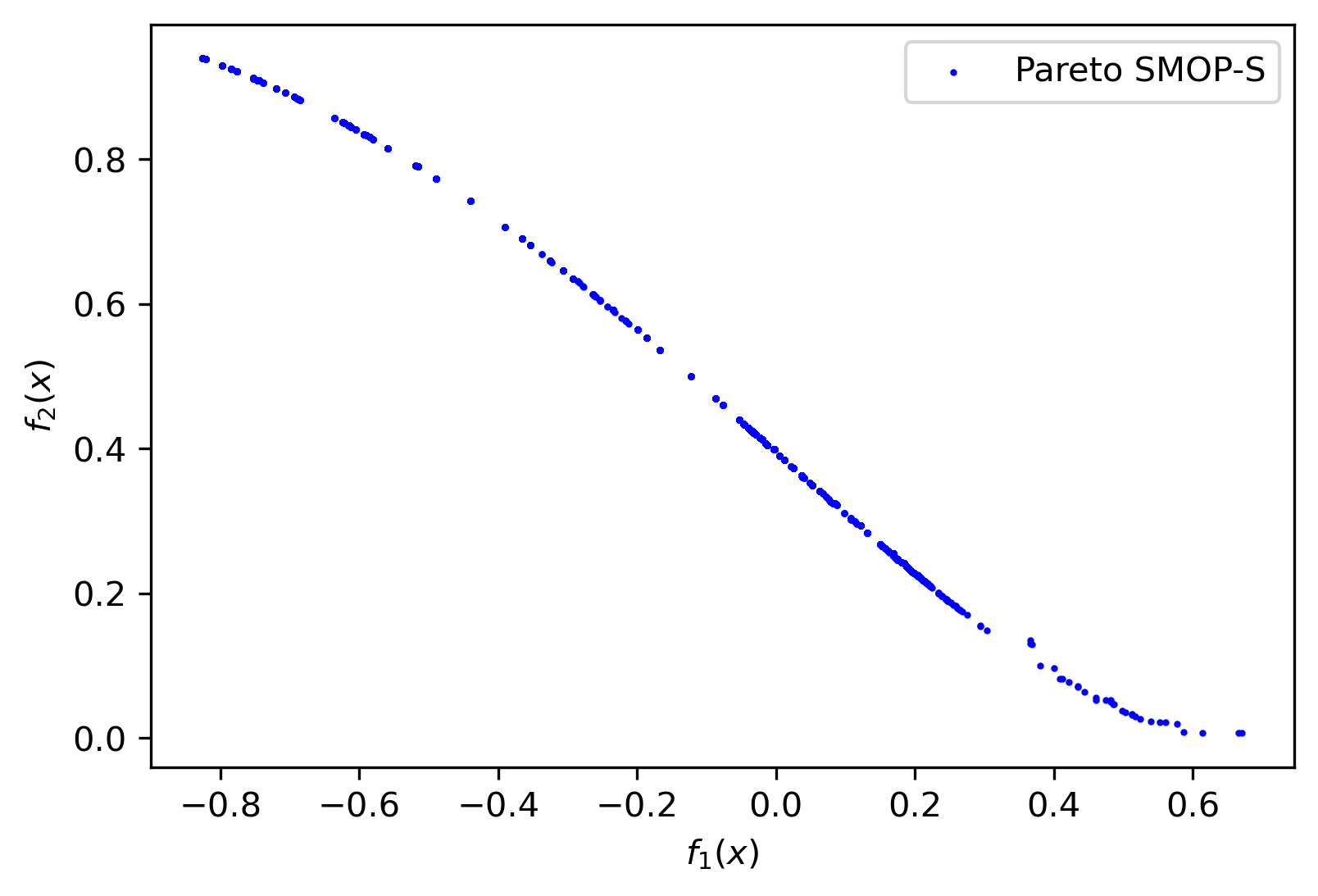} &
\includegraphics[width=0.45\textwidth,keepaspectratio]{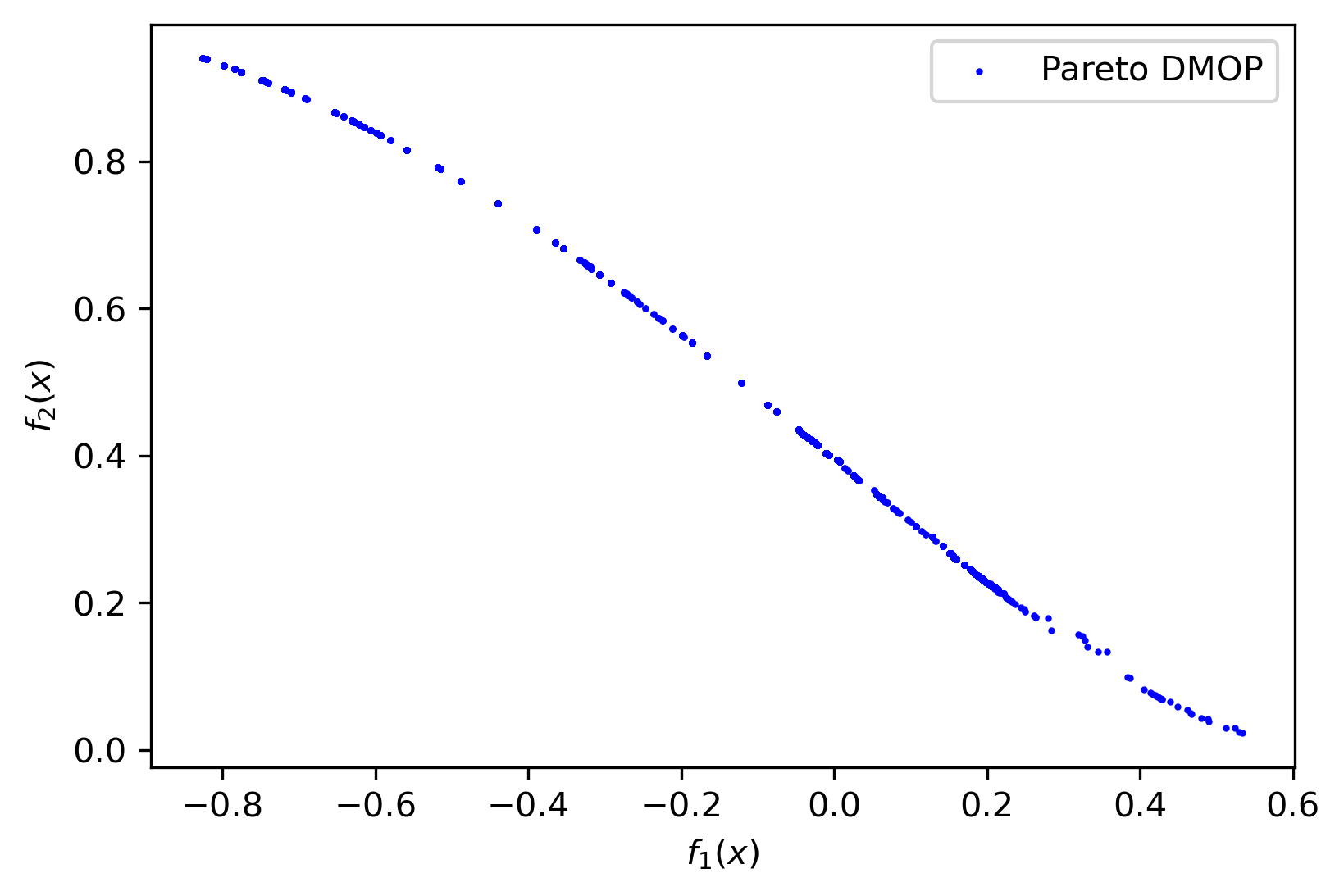} \\
\multicolumn{2}{c}{d)}

\end{tabular}

\caption{Pareto front approximations using SMOP (left) and DMOP (right) on different problems from Table \ref{tab3}: a) SP1, b) SK1, c) FF1, d) T2.}
\label{fig4}
\end{figure}

\section{Conclusion}
 {\color{red} We have proposed a trust region method for multiobjective problems based on probabilistically fully linear models for each function.  The method is designed to operate under the key assumptions that the approximation accuracy of per function models is achieved with high enough probability. The concept of Jointly Independent Fully Linear Probabilistic models is introduced to deal with the fact that the objective scalarization function is nonsmooth and hence the concept of full linearity can not be extended. However we prove that probabilistically fully linear models per function yield a satisfactory  random model for a nonsmooth scalarization function $\phi$.  The theoretical contribution of this work is the proof of almost sure convergence to a Pareto critical point. Possible  applications go beyond the scope of multi-objective optimization. We presented several numerical experiments that showcase the algorithm's efficient practical performance, especially for large scale problems and large data sets. %We highlighted its capability to effectively solve multi-objective optimization problems while maintaining low computational cost. 
 Additionally, we implemented a Pareto finding routine, and made a thorough comparison between the stochastic and deterministic approach. Future work could include the generalization of fully quadratic models or techniques such as additional sampling.  }}

\section*{Acknowledgments}
We are very grateful to the anonymous referees whose insightful comments helped us a lot to improve the paper.

 This work   was supported by the Science Fund of the Republic of Serbia, Grant no. 7359, Project LASCADO.

\end{document}